\documentclass[11pt]{amsart}

\title{Anticanonical tropical cubic del Pezzos contain exactly 27 lines}
\author[M.A.~Cueto and A.~Deopurkar]
       {Maria Angelica Cueto${}^{\S}$ \and Anand Deopurkar}

\thanks{${\S}$ \emph{Corresponding author}}
\date{\today}

\usepackage{xcolor}
\usepackage[draft,inline,nomargin]{fixme}

\fxsetup{theme=color,mode=multiuser}
\FXRegisterAuthor{mac}{amac}{\color{red}Maria}
\FXRegisterAuthor{ad}{aad}{\color{blue}Anand} 

\usepackage[left=1in,top=1in,right=1in,bottom=1in]{geometry}

\usepackage{enumerate}
\usepackage{amsthm}
\usepackage{thmtools}
\usepackage{amsmath}
\usepackage{amsfonts,amssymb}
\usepackage{mathtools}
\usepackage{mathrsfs}
\usepackage{graphicx}
\usepackage{hhline}
\usepackage{blkarray}
\usepackage[unicode=true,colorlinks=true,linkcolor=blue,urlcolor=blue, citecolor=blue]{hyperref}

\usepackage{pdflscape}
\usepackage{afterpage}
\usepackage{caption}

\usepackage{float}

\usepackage{amscd}
\usepackage{xcolor}
\usepackage{ifpdf}
\usepackage{tikz} 
\usepackage[all]{xy}
\ifpdf
\else
\fi 
\numberwithin{equation}{section}
\numberwithin{table}{section}

\makeatletter
\@namedef{subjclassname@2010}{ \textup{2010} Mathematics Subject Classification}
\makeatother

\declaretheorem[style=plain,parent=section]{theorem}
\declaretheorem[style=plain,sibling=theorem]{corollary}
\declaretheorem[style=plain,sibling=theorem]{lemma}
\declaretheorem[style=plain,sibling=theorem]{proposition}

\declaretheorem[style=definition,sibling=theorem]{definition}

\declaretheorem[style=definition, qed=\hfill $\diamond$, sibling=definition]{example}
\declaretheorem[style=remark,sibling=theorem]{remark}
\declaretheorem[style=plain]{claim}

\providecommand {\Z}{\ensuremath{\mathbb{Z}}}\providecommand {\Q}{\ensuremath{\mathbb{Q}}}\providecommand {\R}{\ensuremath{\mathbb{R}}}\providecommand {\CC}{\ensuremath{\mathbb{C}}} 
\newcommand{\KK}{\ensuremath{\mathbb{K}}}
\providecommand{\resK}{\ensuremath{\widetilde{\KK}}}

\renewcommand {\P}{\ensuremath{\mathbb{P}}}\providecommand {\Gr}{{\bf Gr}}

\providecommand{\SL}{\operatorname{SL}}

\providecommand{\GG}[1]{\mathbb{G}_m^{#1}}
\providecommand{\Gm}{\mathbb{G}_m}

\providecommand{\EGp}[1]{\operatorname{E_{#1}}}
\providecommand{\Wgp}{\operatorname{W(\EGp{6})}}
\providecommand{\Sn}[1]{\ensuremath{\mathfrak{S}_{#1}}}

\providecommand {\from}{{\colon}}

\providecommand{\spec}{\operatorname{Spec}}
\providecommand{\proj}{\operatorname{Proj}}

\providecommand{\Pic}{\operatorname{Pic}}
\providecommand{\Sym}{\operatorname{Sym}}

\providecommand{\charF}{\operatorname{char}}
\DeclareMathOperator{\trop}{trop}
\newcommand \cT {\mathcal{T}}
\DeclareMathOperator{\Trop}{\cT\!} \newcommand \TP {\mathbb{T}}
\newcommand{\TPr} {\ensuremath{\TP\P}}
\providecommand{\Rbar}{\overline{\R}}
\newcommand{\PS}{\ensuremath{\CC\{\!\{t\}\!\}}}
\newcommand{\zspan}[1]{\ensuremath{\Z\! \cdot\!#1}}

\newcommand{\val}{\operatorname{val}}
\newcommand{\expVal}{\operatorname{exp.val}}
\providecommand{\tperm}{\ensuremath{\operatorname{perm}}}
\newcommand{\init}{\operatorname{init}}

\newcommand \cY {\mathcal{Y}}
\newcommand \cB {\mathcal{B}}
\newcommand \cI {\mathcal{I}}
\newcommand \Ccurve {\mathcal{C}}
\renewcommand \O {\mathcal{O}} \newcommand{\fh}{\mathfrak{h}}
\newcommand{\cF}{\mathcal{F}}
\newcommand{\cR}{\mathcal{R}}

\newcommand{\scrS}{\mathscr{S}}
\newcommand{\scrC}{\mathscr{C}}

\newcommand{\E}[1]{\ensuremath{E_{#1}}}
\newcommand{\F}[1]{\ensuremath{F_{#1}}}
\newcommand{\Gc}[1]{\ensuremath{G_{#1}}}

\DeclareMathOperator{\Cox}{Cox}

\newcommand{\Yos}[1]{\ensuremath{\mathcal{Y}_{#1}}}
\newcommand{\Cross}[1]{\ensuremath{\operatorname{Cross}_{#1}}}

\newcommand{\TMexp}{\ensuremath{M_{\exp}}}
\newcommand{\TMnone}{\ensuremath{M_{\operatorname{true}}}}
\newcommand{\Naruki}{\ensuremath{\mathcal{N}}}

\newcommand {\sage} {\texttt{Sage}}
\newcommand {\python}{\texttt{Python}}

\usepackage[boxed,vlined,lined,algoruled]{algorithm2e}
\newcommand{\Subsets}{\operatorname{Subsets}}
\newcommand{\expec}{\ensuremath{\operatorname{exp}}}
\newcommand{\noNones}{\ensuremath{\operatorname{true}}}
\newcommand{\new}{\ensuremath{\operatorname{new}}}

\newcommand{\ovalpha}{\overline{\alpha}}
\newcommand{\ovLambda}{\overline{\Lambda}}

\DeclareMathOperator{\Irr}{Irr}

\newcommand{\Fl}{\mathcal{F}}

\newcommand{\reg}{\ensuremath{\operatorname{reg}}}

\newcommand{\sed}{\ensuremath{\mathfrak{s}}}
\newcommand{\ted}{\ensuremath{\mathfrak{t}}}

\newcommand{\rspanone}{\ensuremath{\R\!\cdot\!\mathbf{1}}}
\newcommand{\ww}{\ensuremath{\omega}}

\begin{document}

\begin{abstract}
  The classical statement of Cayley-Salmon that there are 27 lines on every smooth cubic surface in $\P^3$ fails to hold under tropicalization: a tropical cubic surface in $\TPr^3$ often contains infinitely many tropical lines. Under mild genericity assumptions, we show that when embedded using the Eckardt triangles in the anticanonical system, tropical cubic del Pezzo surfaces contain exactly 27 tropical lines.
  In the non-generic case, which we identify explicitly, we find up to 27 extra  lines, no multiple of which lifts to a curve on the cubic surface.

  We realize the moduli space of stable anticanonical tropical cubics as a four-dimensional fan in $\R^{40}$ with an action of the Weyl group $\Wgp$. In the absence of Eckardt points, we show the combinatorial types of these tropical surfaces are determined by the boundary arrangement of 27 metric trees corresponding to the tropicalization of the classical 27 lines on the smooth algebraic cubic surfaces. Tropical convexity and the combinatorics of the root system $\EGp{6}$ play a central role in our analysis. 
\end{abstract}

\maketitle
\section{Introduction}\label{sec:introduction}

A persistent theme in tropical geometry has been to explore tropical analogues of classical results in algebraic geometry. In this paper, we address the tropicalization of the well-known theorem of Cayley and Salmon~\cite{cay:1849} that there are 27 lines on any smooth cubic surface in $\P^3$ over an algebraically closed field. Work of Vigeland~\cite{vig:10} shows that this statement, taken literally, fails in tropical geometry. He provides examples of cubic surfaces $X$ over valued fields whose tropicalization  in $\TPr^3$ contains infinitely many tropical lines, moving along one-parameter families in the interior of the tropical surface $\Trop X$. Recent work of Panizzut-Vigeland~\cite[Theorem 1]{pan.vig:19} complements this result with a complete classification of the combinatorial position of all tropical lines on each smooth tropical cubic surface, building on~\cite[Section 6.2]{ham.jos:17}.

Since tropical varieties are coordinate dependent~\cite{mac.stu:15}, it is conceivable that a different choice of embeddding could correct the discrepancy between the count of classical and tropical lines.
The present paper proposes the \emph{anticanonical embedding} as the appropriate model that corrects this pathology, ending a decade-long search. More explicitly, we show that the tropicalization of the (linearly degenerate) embedding of $X$ in $\P^{44}$ defined by the 45 distinguished sections of the anticanonical bundle corresponding to the 45 (Eckardt) tritangent planes of $X \subset \P^3$ prevents the superabundance of tropical lines on smooth tropical  cubic surfaces in almost all cases.

Throughout this paper, we let $\KK$ be  an algebraically closed valued field with $\charF\, \KK\neq 2,3$ with residue field $\widetilde{\KK}$ with equal characteristic restrictions.
By an \emph{anticanonical} cubic del Pezzo surface over $\KK$, we mean a cubic del Pezzo surface embedded in $\P^{44}_{\KK}$ as described above.

\begin{theorem}\label{thm:anticanonicalNoLines}
  Let $\Trop X \subset \TPr^{44}$ be an anticanonical tropical del Pezzo surface
  associated to a non-zero point in the moduli space of stable tropical cubic surfaces defined over $\KK$.
  Then, $\Trop X$ contains exactly 27 tropical lines, all of which lie in its boundary.
\end{theorem}

For a more precise statement, we refer to~\autoref{thm:nonTrivialValNoLifting}. Our work builds on recent developments on tropicalization of classical moduli spaces by Ren, Sam, Shaw and Sturmfels~\cite{ren.sam.stu:14, ren.sha.stu:16}. The moduli space of \emph{stable} tropical del Pezzo surfaces mentioned above is the Naruki fan from~\cite{ren.sha.stu:16}. Unstable surfaces are those that cannot be obtained as a limit of tropical surfaces arising from the maximal cones of this fan.

The apex of the  Naruki fan is the single cone not addressed in~\autoref{thm:anticanonicalNoLines}. Its unique associated stable tropical surface corresponds to tropicalizations of smooth cubics $X\subset \P^{44}$ defined over trivially valued fields. In this case, we get exactly 27 extra tropical lines in the interior of  $\Trop X$. In the unstable case, the construction gives an upper bound for the number of extra tropical lines.

\emph{Realizability} of tropical cycles supported on subsets of these extra lines by algebraic curves in $X$  arises as a natural question~\cite{bog.kat:12, bru.sha:15,mak.rud:19, mik:05}. Intersection theory techniques rule out liftings of tropical cycles supported on a single tropical line. Furthermore, even though cubic surfaces contain pencils of residual conics~\cite{ste:1857}, we show that their tropicalizations do not agree with these pairs of extra tropical lines. In summary:

\begin{theorem}\label{thm:extraLinesApex}
  Tropical surfaces associated to the apex of the Naruki fan contain up to  27 additional non-generic tropical lines in its interior: each one is a star tree with five rays. This bound is attained when the surface is stable. In all cases, no pairwise combinations nor multiples of these extra lines are realized by curves on the  cubic surface in $\P^{44}$.
\end{theorem}

We now sketch the main ideas underlying our study of tropical anticanonical cubic surfaces. The first key input is the rich combinatorics of the incidence relations among the 27 lines discovered by  Cayley~\cite{cay:1849}, Clebsch~\cite{cle:1861} and Salomon~\cite{sal:1849}. Recall that the pairwise intersection patterns of the 27 lines is independent of the cubic surface. Encoded in a graph, it forms the 10-regular graph on 27 vertices called the \emph{Schl\"afli graph}, which is also the edge-graph of Gosset's six-dimensional polytope $2_{21}$~\cite{cox:1940,sch:1910}. Although pairwise intersections bring no surprises, there are some surfaces on which unexpected triple intersections happen---the three lines forming a tritangent plane (i.e., an anticanonical triangle) may degenerate to three concurrent lines. In this case, the point of concurrency is called an \emph{Eckardt point}~\cite{eck:1876}.  Throughout this paper, we restrict ourselves to cubic surfaces with no Eckardt points.

A second key input in this paper is a uniformization of the moduli space of (classical) cubic del Pezzo surfaces by the root system of type $\E6$. Recall that cubic del Pezzos can be obtained by blowing up $\P^2$ at six points in general position. By choosing suitable planar coordinates, we may assume that these six points lie on the cuspidal cubic $y^3 = x^2z$, which has the rational parametrization $t \mapsto [1:t:t^3]$. As a result, we represent the six points by six parameters $d_1, \dots, d_6$ in $\KK$. The six points are in general position (no three on a line, no six on a conic) if and only if the parameters satisfy
\begin{equation}\label{eq:rootsE6}
d_i-d_j \neq 0 \;\text{ (for } i<j \text{)} \; \text{ , } \quad d_i+d_j+d_k\neq 0 \;\text{ (for } i<j<k \text{)} \quad \text{ and } \quad d_1+\ldots+d_6\neq 0.
\end{equation}
These expressions form the set $\Phi^+$ of 36 positive roots of $\EGp{6}$~\cite{ren.sha.stu:16}. The complement of the root hyperplane arrangement in the projectivized $\E6$ lattice thus yields a parameter space for smooth cubic del Pezzo surfaces. 
The astute reader may have observed that this parameter space is five-dimensional, whereas the moduli space of cubic del Pezzo surfaces is four-dimensional. The discrepancy exists because there is a one-parameter family of choices of $(d_1, \dots, d_6)$ that leads to projectively equivalent sextuples.
The existence of an Eckardt point on a given tritangent plane can be detected by the vanishing of a  quintic polynomial in the six parameters $d_1,\ldots, d_6$. We refer to it as an Eckardt quintic. By assumption, all 45 Eckardt quintics are non-zero, and hence have
finite valuation.

We can recover the moduli space of cubic surfaces from the parameter space described above by identifying functions in $d_1, \dots, d_6$ that descend to the moduli space. These functions have been explicitly described by Coble using invariant theory~\cite{cob:61}. Indeed, Coble provides a $\Wgp$-equivariant set  of 40 degree nine monomials $\cY_0, \ldots, \cY_{39}$ in the positive roots $\Phi^+$ from~\eqref{eq:rootsE6}, called the \emph{Yoshida functions} (see~\autoref{tab:yoshida}), that generates the ring of functions on the moduli space. Needless to say, they are central to this paper.

In addition to the Yoshida functions, another set of 135 functions, called \emph{Cross functions} (see \autoref{tab:crosses}), is essential to our work. Each of them can be expressed as a binomial linear combination of Yoshida functions, and contains an Eckardt quintic as a  factor (the remaining factors are four roots in $\Phi^+$.) Both the Yoshida functions and the Cross functions arise as discriminants of certain root subsystems of $\E6$~\cite{col.van-gee.loo:09}.

The Yoshida functions play a key role in describing not only the classical but also the  tropical moduli space of  cubic del Pezzos. Recall that the process of tropicalization turns varieties into polyhedral complexes by means of valuations. The moduli space of stable tropical cubic surfaces, namely the Naruki fan, is constructed as a complex in $\R^{40}/{\rspanone}$ by taking the valuations of all 40 Yoshida functions~\cite{hac.kee.tev:09,ren.sha.stu:16}.

This fan and its combinatorial properties were studied in~\cite[Section 6]{ren.sam.stu:14}. Each of its 24 cones (up to $\Wgp$-symmetry) determines the combinatorics of the tropicalization of the complement of the 27 lines on a cubic surface, embedded in the torus $\Gm^{27}/\Gm^7$ over $\KK$ via the Cox ring~\cite[Table 1]{ren.sha.stu:16}. By design, the 27 lines on $X$  become an arrangement of metric trees in the boundary of $\Trop X$. We show that the same fan classifies anticanonically embedded stable tropical del Pezzo cubics:
\begin{theorem}\label{thm:Naruki}
  The Naruki fan $\Naruki$ is the moduli space of stable anticanonical smooth tropical cubic surfaces in $\TPr^{44}$. The combinatorial type of the stable tropical cubic surface associated to each point in $\Naruki$ is determined by the arrangement of 27 metric trees at infinity (see Figures~\ref{fig:treesaa2a3a4} and~\ref{fig:treesaa2a3b}.) The metric structure on each tree is a piecewise linear function on the Naruki fan (see~\autoref{tab:treeLabeling}.)
\end{theorem}
In particular, the stable tropical surface associated to the apex of the fan is the cone over the Schl\"afli graph: it has 27 vertices (one of each line) and 135 edges. As we move to higher-dimensional cones, each vertex is replaced by a metric tree with 10 leaves. The unstable tropical surfaces will only be detected by recording the valuations of all Cross functions. In turn, the moduli space of all smooth tropical cubic surfaces with no Eckardt points will be obtained as a \emph{tropical modification}~\cite{ite.mik.shu:09, mik:06} of the Naruki fan along the tropical Cross functions.

\medskip

Having described the main results, we turn to some details regarding the proof of~\autoref{thm:anticanonicalNoLines}. The rigidity of the boundary structure of $\Trop X$ discussed in~\autoref{thm:Naruki} ensures that any extra tropical line must lie in the interior of $\Trop X$. A simple combinatorial analysis implies that any such potential tropical line must meet the boundary of $\Trop X$ at precisely five points: the ones determined by the link of a vertex in the Sch\"afli graph. There are 27 such 5-tuples. Our description of the defining ideal of $X$ in $\P^{44}$ allows us to express the coordinates of all 135 intersection points of pairs of  lines in $X$ as Laurent monomials in Yoshida and Cross functions.

The information recorded by the Naruki fan is, a priori, not enough to determine the valuation of the 135 intersection points. However, for each Naruki cone, except the apex, we can determine the valuations of enough coordinates of these 5-tuples of boundary points to conclude they cannot be tropically collinear in $\TPr^{44}$. We do so by exhibiting an explicit tropically non-singular $3\times 3$ minor in each associated $5\times 45$ matrix with entries in $\R\cup \{\infty\}$, as seen in Tables~\ref{tab:coneA} and~\ref{tab:allCones}.  This novel technique lies at the heart of \emph{tropical convexity}~\cite{dev.san.stu:05} and we expect it be applicable to study  more general {Fano schemes}.

For the apex of the Naruki fan, we determine tropical collinearity of these boundary points in terms of valuations of tuples of five Cross functions. This criterion provides only an upper bound for the number of extra lines discussed in~\autoref{thm:extraLinesApex} since  each extra line may not be contained in the tropical cubic surface if the later is unstable.

\textcolor{blue}{Theorems}~\ref{thm:anticanonicalNoLines} and~\ref{thm:extraLinesApex} allow us to reinterpret the superabundance phenonena observed in Vigeland's examples~\cite{vig:10}. Indeed, the space  of anticanonical global sections of $X$ is four-dimensional, but has no natural basis. As a result, there is no canonical way to  embed $X$ into its linear span. However, \autoref{thm:anticanonicalNoLines} shows that the linearly degenerate embedding of $X$ in $\P^{44}$ is more natural from the point of view of tropicalization.

By construction, the linear span of $X$ in $\P^{44}$ is a $\P^3$. It would be interesting to characterize the valuated matroid determining the tropicalization of this linear space in $\TPr^{44}$. Such problem requires knowledge of the valuations of the Pl\"ucker coordinates of the point in $\Gr(4,45)$ associated to $\P^3 \subset \P^{44}$. Although our explicit description of the anticanonical embedding gives precise formul\ae\ for the Pl\"ucker coordinates, determining their valuations is a subtle matter, since their factorization involves not only Laurent monomials in Yoshida and Cross functions, but also some new septic and octic factors in $d_1, \dots, d_6$.

These non-monomial factors give a second motivation to restrict our combinatorial study to tropical stable cubic surfaces: the Naruki fan does not provide a complete description of the tropicalization of the linear span of $X$. That is, the valuated matroid associated to the pair $\P^3\subset \P^{44}$ can vary within a given Naruki cone. As was mentioned earlier, tropical modifications will allow us to overcome this and other difficulties arising from non-generic choices. However, the polyhedral subdivisions required to carry this task in practice will be extremely delicate given the sheer size of the input subspace arrangement. We plan to address it in future work.

 The rest of the paper is organized as follows. In~\autoref{sec:moduli-cubics} we recall the classical construction of the moduli space of smooth marked del Pezzo surfaces and its compactification (the Naruki space.) In this setting, the 27 lines on each surface corresponds to its set of exceptional classes. We provide an equivariant uniformization of this space in terms of the $\EGp{6}$ arrangement complement in $\P^5$, and introduce the two main sets of invariants: the Yoshida and Cross functions. The latter characterizes which surfaces have Eckardt points. \autoref{sec:anti-canon-emb} computes the Cox and anticanonical embeddings of the universal cubic surface. In both cases, the boundary is supported on the 27 lines.  \autoref{sec:appendix1} provides explicit formulas for all Yoshida and Cross functions that complement these results. They are an essential component of all the computations appearing in the paper.

\autoref{sec:bergman-fan-Naruki-fan} discusses the tropicalization of the Naruki space as the image of the Bergman fan of the $\EGp{6}$-reflection arrangement via a linear map, encoded by the Yoshida matrix. We endow this four-dimensional space with a fan structure compatible with the tropicalization of all Cross functions. We discuss the combinatorial data of this fan up to the action of the Weyl group $\Wgp$ and characterize the fibers of the Yoshida matrix over the relative interior of its maximal cones.

\autoref{sec:bergman-fan-Naruki-fan} presents the main techniques involved in the proof of our three main results. First, the characterization of tropical convexity by vanishing of tropical $3\times 3$ determinants, and second, an expicit algorithm to reconstruct any tropical line from its collection of boundary points.

The proof of~\autoref{thm:anticanonicalNoLines} is carried over in~\textcolor{blue}{Sections}~\ref{sec:comb-extra-trop} and~\ref{sec:trop-lines-trop}. Tables~\ref{tab:allCones} and~\ref{tab:coneA} exhibit non-singular tropical minors ruling out potential extra lines. The proof of~\textcolor{blue}{Theorems}~\ref{thm:extraLinesApex} and~\ref{thm:Naruki}  require the explicit knowledge of the generic valuations of all Cross functions. This is carried out in~\autoref{sec:find-repr-cross} and in~\autoref{pr:newExpValCr15}. \autoref{sec:tropical-lines-trivial} contains the proof of~\autoref{thm:extraLinesApex}.

The proof of~\autoref{thm:Naruki} spans~\textcolor{blue}{Sections}~\ref{sec:comb-types-tree} and~\ref{sec:trop-moduli-space}. In~\autoref{sec:comb-types-tree}, we construct the boundary trees for each tropical surface. \autoref{tab:treeLabeling} gives the leaf labeling and metric structure for those surfaces that are stable. The combinatorial types are shown in Figures~\ref{fig:treesaa2a3a4} and \ref{fig:treesaa2a3b}. \autoref{sec:planarConf} gives an alternative way to build these trees in terms of tropically generic configurations of six points in $\TPr^2$. \autoref{sec:trop-moduli-space} provides a test for tropical stability in terms of the valuations of the Cross functions on the algebraic counterparts. In particular,~\autoref{thm:aaaaIsStable} shows that this test is always satisfied for surfaces associated to points in the relative interior of maximal cones in the Naruki fan. We use this property to conclude that this fan is the moduli space of stable tropical cubic surfaces.

Finally, \autoref{sec:comb-types} describes the combinatorics of each stable tropical cubic surface and certifies that it agrees with those obtained from Cox embeddings.

\section*{Supplementary material}\label{sec:suppl-mater}

Many of the results in this paper rely on calculations performed using \sage~\cite{the:15}. Several of them required new implementations within this platform using~\python. We have created supplementary files so that the reader can reproduce all the claimed assertions done via explicit computations. These files can be found at:
\begin{center}
  \url{https://people.math.osu.edu/cueto.5/anticanonicalTropDelPezzoCubics/}
\end{center}

  \noindent
  In addition to all \sage\ scripts, we include all input and output files (in \sage\ and plain text format.)
       
  All computations are performed symbolically using either our own implementation of group-actions on polynomial rings, tropical operations ($\min$ for tropical addition and usual addition for tropical multiplication), or built-in functions for computations with Weyl groups, polyhedra and factorizations of rational functions over the rational numbers. They were performed on a 2.4 GHz Intel(R) Core 2 Duo with 3MB cache and 2GB RAM. The implementation of the construction of the Bergman fan of a general matroid from its nested sets is new and exploits symmetries whenever possible. The computation of the Bergman fan for the $\EGp{6}$-arrangement  takes about one hour to finish. The time is split evenly between the calculation of adjacencies and the whole fan.

  The most time-demanding computations are those in~\autoref{sec:trop-lines-trop}. The calculation of the tropicalization of each $5\times 45$ matrix of rational functions associated to each cone in the Naruki fan takes about 20 minutes.
  The computation-time required to search for tropically singular $3\times 3$-minors for each such matrix is not uniform.  With the exception of a single cone, each calculation takes about 30 seconds. For the problematic cone and  each choice of three rows in the corresponding matrices, the computation takes ten minutes since several exceptional curves have no tropical non-singular $3\times 3$ minors, and the calculation exhausts all 3-element subsets from $\{0,\ldots, 44\}$.
  \section*{Acknowledgements}
  
The authors wish to thank Sachin Gautam, Kristian Ranestad, Dhruv Ranganathan,  Kristin Shaw, Bernd Sturmfels and Jenia Tevelev for fruitful conversations. All the implementations in this paper where done using the software package \sage~\cite{the:15}.
The first author was partially supported by an  NSF postdoctoral
fellowship DMS-1103857 and NSF Standard Grant DMS-1700194 (USA.) The second author was partially supported by a Simons Foundation Travel Grant (USA) and the Discovery Early Career Research Award DE180101360 (Australia.)
Both authors acknowledge the Mathematics Department of both Columbia University and The Ohio State University where most of this project was carried out. Finally, the second author thanks the hospitality of The Fields Institute for Research in Mathematical Sciences and the organizers of the Major Thematic Program on Combinatorial Algebraic Geometry (July-December 2016) where crucial stages of this project were completed.

\section{Moduli of cubic del Pezzo surfaces and the action of $\Wgp$}\label{sec:moduli-cubics}

In this section, we review the classical construction of the moduli space of marked del Pezzo surfaces, originating in the work of Coble.
Our main references are \cite{col.van-gee.loo:09} and \cite{ren.sam.stu:14}, which describe classical and tropical moduli spaces of del Pezzo surfaces of arbitrary degree.
To simplify and focus our exposition, we restrict ourselves to cubic del Pezzos. All varieties in this paper are to be defined over an algebraically closed field $\KK$ of characteristic different from 2 and 3, equipped with a (possibly trivial) non-Archimedean valuation. Its residue field $\resK$ has $\charF \resK \neq 2,3$. For an arithmetic perspective on the enumerative geometry of del Pezzos over such fields, we refer to the recent work of Kass and Wickelgren~\cite{kas.wic:17}.

\begin{definition}\label{def:cubicDelPezzo}
  A \emph{cubic del Pezzo surface} is a smooth projective surface with ample anticanonical bundle whose class has self-intersection three.
\end{definition}
The following equivalent definition will be used throught this work. A cubic del Pezzo surface is a surface obtained from $\P^2$ by blowing up 6 distinct points, no three of which are collection and not all six lie on a conic.

\begin{remark}\label{rm:FanoConvention}
  The authors of \cite{col.van-gee.loo:09} use the term ``del Pezzo'' to mean surfaces with semi-ample canonical bundle, reserving the term ``Fano'' for the ones with ample canonical bundle.
  We alert the reader that our terminology is slightly different.
\end{remark}

Let $X$ be a cubic del Pezzo surface.
Since $X$ is obtained by six blow-ups from $\P^2$, the Picard group of $X$ is isomorphic to $\Z^7$; it is generated by the canonical class $K_X$ and the classes of the six exceptional divisors.
The Picard group contains 27 \emph{exceptional divisor classes}, namely classes $E$ with $K_X \cdot E = -1$ and $E^2 = -1$. These are precisely the classes of the 27 lines.

Each exceptional class is represented by a unique effective divisor.
An ordered collection of six exceptional classes $\E1, \dots, \E6$ with $\E{i} \cdot \E{j} = - \delta_{ij}$ is called a \emph{marking} of $X$.
For example, if $X$ is obtained from $\P^2$ by blowing up six distinct points, then the six exceptional divisors give a marking.
It turns out that there are $72 \cdot 6!$ markings of a del Pezzo surface (72 when disregarding the order.)
So there are essentially 72 ways in which a general cubic surface arises as a blow-up of $\P^2$.

The blow-up construction shows that the moduli space of marked cubic del Pezzos is isomorphic to a dense open $U \subset (\P^2)^6$.
The set $U$ consists of tuples of six  distinct points, no three of which are collinear, and the six do not lie on a conic.
To highlight the role of the markings, we denote this moduli space by $M_{m,3}^\circ$.

The group of automorphisms of the lattice $\langle \Pic X, \cdot \rangle$ that fix $K_X$ is the Weyl group $\Wgp$.
This group acts on the markings, and hence, on $M_{m,3}^\circ$.
The quotient is an open subset of the moduli of (unmarked) del Pezzo cubics, which we denote by $M_{3}^\circ$.
The latter will only play an auxiliary role.

\subsection{The Naruki space and its Coble covariants}\label{sec:naruki-space-coble}
The space $M_3^\circ$ admits a natural compactification $M_3^*$ using Geometric Invariant Theory (GIT.)
The compactified moduli space $M_3^*$ is the GIT quotient $\P \Sym^3 (\CC^4) / \SL(4)$.
 The line bundles $\O(n)$ on $\P \Sym^3 (\CC^4)$ descend to rank-one sheaves on $M_3^*$ (they are line bundles if $M_3^*$ is considered as a stack rather than a coarse space, but we will ignore this point.)
 It is convenient to denote by $\O_{M_3^*}(1)$ the sheaf descended from $\O(4)$.

 Let $M_{m,3}^*$ be the normalization of $M_3^*$ in $M^\circ_{m,3}$.
Then, $M_{m,3}^*$ is a compactification of $M_{m,3}^\circ$, called the \emph{Naruki space}.
The action of $\Wgp$ on $M_{m,3}^\circ$ extends to an action on $M_{m,3}^*$ and the map $M_{m,3}^* \to M_{3}^*$ is the quotient.
Denote by $\O_{M_{m,3}^*}(1)$ the pullback of $\O_{M_{3}^*}(1)$.

\begin{proposition}\label{prop:CobleVSpace}
  The vector space $H^0(M_{m,3}^*, \O(1))$ is ten-dimensional and  irreducible as a representation of $\Wgp$.
  It yields an embedding $M_{m,3}^* \hookrightarrow \P^9$.
\end{proposition}
\noindent For a proof we refer to \cite[Corollary~5.9]{col.van-gee.loo:09}  and the preceding discussion.
The global sections of $\O_{M_{m,3}}^*(1)$ are called \emph{Coble covariants}, and the image of $M_{m,3}^*$ in $\P^9$ is called the \emph{Naruki space}.
Our next subsection gives an explicit description of these two notions. 

\subsection{An equivariant uniformization}\label{sec:equiv-uniform}
Let ${\fh}_6^*$ be the the $\KK$-vector space spanned by the root lattice of $\EGp{6}$. Explicitly, ${\fh}_6^*$ is generated by six elements $d_1, \dots, d_6$ with a bilinear form given by
\[ d_i \cdot d_j =
  \begin{cases}
    -1/9 & \text{if $i \neq j$},\\
    \hspace{2ex} 8/9 & \text{if $i = j$}.
  \end{cases}
\]
The elements $(d_j - d_i)$ for $i < j$, $(d_i+d_j+d_k)$ for $i < j < k$, and $(d_1 + \dots + d_6)$ together form the set of 36 positive roots $\Phi^+$ of $\E6$ associated to the simple roots
\[
\alpha_1=d_2-d_1,\quad \alpha_2=d_1+d_2+d_3 \; \text{ and }\; \alpha_i= d_{i}-d_{i-1}\, \text{ for }i=3,4,5,6.
\]
The choice of simple roots $\alpha_i$ follows Bourbaki's convention for labeling the Dynkin diagram of type $\EGp{6}$ shown in~\autoref{fig:E6label}.

\begin{figure}[htb]
  \includegraphics[scale=0.3]{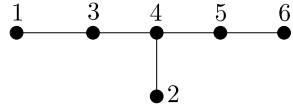}
  \caption{The labeled Dynkin diagram of the
  root system $\EGp{6}$. Each label $i$ corresponds to the simple root $\alpha_i$.\label{fig:E6label}}
\end{figure}

Let $\fh_6$ be the dual of $\fh_6^*$.
We have a map $\fh_6 \to (\P^2)^6$ given by
\[ p \mapsto \left([1:d_1(p):d_1^3(p)], \dots, [1:d_6(p):d_6^3(p)]\right).\]
Denote by $\fh_6^{\reg}$ the complement in $\fh_6$ of the zeros of the roots.
For $p \in \fh_6^{\reg}$, the points $[1:d_i(p): d_i^3(p)]$ for $1 \leq i \leq 6$ are distinct, no three of them lie on a line, and the six do not lie on a conic.
Therefore, the blow-up of $\P^2$ at these points gives a marked cubic del Pezzo surface, where the marking is given by the six exceptional divisors.
Scaling $p$ by $t \in \KK^*$ produces a different set of six points, but they are related to the original six by the automorphism of $\P^2$ defined by $[X:Y:Z] \mapsto [X:tY:t^3Z]$.
As a result, the resulting marked surfaces are isomorphic.
We thus get a morphism
\begin{equation}\label{eq:uniform}
  \P(\fh_6^{\reg}):=\fh_6^{\reg}/\KK^* \longrightarrow M_{m,3}^\circ.
\end{equation}
It is easy to check that this map is equivariant with respect to the action of $\Wgp$. Furthermore, it is surjective and flat~\cite[Theorem~3.1]{col.van-gee.loo:09} of relative dimension one.

Since $M_{m,3}^*$ is complete and $\P(\fh_6)$ is non-singular, the map in \eqref{eq:uniform} extends to a regular map away from a set $Z$ of codimension at least 2.
It turns out that the pullback of $\O_{M_{m,3}^*}(1)$ to $\P(\fh_6)$ is isomorphic to $\O_{\P(\fh_6)}(9)$ \cite[Proposition~4.10]{col.van-gee.loo:09}. 
As a result, we can write (the pullbacks of) the Coble covariants as homogeneous polynomials of degree nine in $d_1, \dots, d_6$.

\subsection{Yoshida and Cross functions}\label{sec:yosh-cross-funct}
Two sets of Coble covariants play a key role in this paper.
Both sets are $\Wgp$-invariant and have a beautiful description in terms of root subsystems of the root system $\E6$ (see~\cite{col.van-gee.loo:09}.), which we discuss below. In addition, we provide  explicit formulas for all invariants that will be heavily exploited in our tropical computations.

We start with some basic definitions involving root systems. By the \emph{discriminant} $\Delta$ of a root system, we mean the square root of the product of all the roots, both positive and negative.
This is a polynomial, well-defined up to a sign.
Prescribing a set of positive roots $R^+$ pins down the sign -- we simply take the product of all the positive roots. Thus, $\Delta = \prod_{\alpha\in R^+}\alpha$.

The first set of Coble covariants are the \emph{Yoshida functions}.
These are  the discriminants of type $A_2^{\oplus 3}$ root subsystems  of $\fh_6$.
For example, the subsystem $S$  below yields the Yoshida function $\Yos{}(S)$:
\begin{equation}\label{eq:yoshida1inD}
  \begin{minipage}{0.45\textwidth}
    \[\begin{aligned}
    S =  &\langle d_5 - d_6, d_1 + d_2 + d_6, d_1 + d_2 + d_5 \rangle \oplus \\
    &\langle d_3 - d_4, d_4 + d_5 + d_6, d_3 + d_5 + d_6 \rangle \oplus \\
    & \langle d_1 - d_2, d_2 + d_3 + d_4, d_1 + d_3 + d_4 \rangle
    \end{aligned}\]
      \end{minipage}
 ;\    \begin{minipage}{0.45\textwidth}
      \[\begin{aligned}
    \Yos{}(S) = &  (d_{5} -  d_{6})(d_{1} + d_{2} + d_{6})(d_1+d_2+d_5) \\
    &(d_3-d_4)(d_4+d_5+d_6)(d_3+d_5+d_6) \\
    &(d_1-d_2)(d_2+d_3+d_4)(d_1+d_3+d_4).
      \end{aligned}\]
  \end{minipage}
\end{equation}

\begin{remark}\label{rm:WeylActionYoshidas}
  The group $\Wgp$ acts transitively on the Yoshida functions, so the others can be computed using the group action.
There are a total of 80 Yoshida functions (40 up to sign.)
Since the Yoshida functions are products of roots, they are invertible on $\P(\fh_6^{\reg})$.
Equivalently, by~\eqref{eq:uniform}, the corresponding Coble covariants are invertible on $M_{m,3}^\circ$.
\end{remark}
The Yoshida functions span the 10-dimensional space of Coble covariants from~\autoref{prop:CobleVSpace}.
Since the linear system defined by the Coble covariants is very ample on $M_{m,3}^*$, we can recover
\begin{equation}\label{eq:naruki_embedding}
  M_{m,3}^* \subset \P^{39}
\end{equation}
as the closure of image of the map
$\P(\fh_6) \dashrightarrow \P^{39}$
defined by the 40 Yoshida functions (up to sign.) Our choice of signs is indicated in~\autoref{tab:yoshida}.

\smallskip

The second set of Coble covariants are the \emph{Cross functions.}
Let $S \subset \fh_6^*$ be a root subsystem of type $A_2^{\oplus 3}$ and let $\alpha \in \fh_6^*$ be a root not orthogonal to any of the summands of $S$.
Let $S^+$ be a set of positive roots for $S$ and denote by $s_\alpha$ the reflection in the plane orthogonal to $\alpha$.
The Cross associated to the pair $(S^+, \alpha)$ is the difference
\begin{equation}\label{eqn:Cross_def}
  \Cross{}(S^+,\alpha) =  \Delta(S^+) - s_\alpha (\Delta(S^+)). 
\end{equation}

\begin{remark}\label{rm:CrossDifferenceYoshidas}
  Note that the Cross is a difference of two Yoshidas.
Furthermore, due to the linear relations between the Yoshida functions, each Cross function can be expressed in four distinct ways as a difference of Yoshidas.
The data of $(S, \alpha)$ (without the choice of positive roots) determines the Cross function up to a sign; we denote it by $\Cross{}(S, \alpha)$.
\end{remark}

By~\cite[Lemma 4.2]{col.van-gee.loo:09}, there are three mutually orthogonal roots in $S^+$ orthogonal to $\alpha$ (say $\alpha_1$, $\alpha_2$, and $\alpha_3$), each of them lying  in a different copy of $A_2$ in $S$.
Furthermore, the four roots divide $\Cross{}(S, \alpha)$, thus inducing the factorization:
\begin{equation}\label{eq:Cross}
  \Cross{}(S, \alpha) = \alpha \,\alpha_1\, \alpha_2\, \alpha_3\, Q,
\end{equation}
where $Q$ is a quintic (irreducible) polynomial (see~\cite[Lemma 4.4]{col.van-gee.loo:09}.) 

\begin{example}\label{ex:cross}
Let $\alpha = d_2 + d_4 + d_5$ and consider the root subsystem  $S$ from~\eqref{eq:yoshida1inD}, 
where the nine positive roots are the listed ones.  It follows that 
\begin{equation*}
  \begin{aligned}
    Q :=
    & d_{1}^{2} d_{2}^{2} d_{3} + d_{1}^{2} d_{2} d_{3}^{2} -  d_{2}^{2} d_{3} d_{4}^{2} -  d_{2} d_{3}^{2} d_{4}^{2} -  d_{1}^{2} d_{2}^{2} d_{5} -  d_{1}^{2} d_{2} d_{3} d_{5} + d_{1} d_{2}^{2} d_{3} d_{5} -  d_{1}^{2} d_{3}^{2} d_{5} 
   + d_{1} d_{2} d_{3}^{2} d_{5} \\ & + d_{1}^{2} d_{4}^{2} d_{5} -  d_{1} d_{2}^{2} d_{5}^{2} -  d_{1} d_{2} d_{3} d_{5}^{2} + d_{2}^{2} d_{3} d_{5}^{2} -  d_{1} d_{3}^{2} d_{5}^{2} + d_{2} d_{3}^{2} d_{5}^{2} + d_{1} d_{4}^{2} d_{5}^{2}
   -  d_{2}^{2} d_{3} d_{4} d_{6}  -  d_{2} d_{3}^{2} d_{4} d_{6} \\ &-  d_{1}^{2} d_{4}^{2} d_{6} + d_{2}^{2} d_{4}^{2} d_{6} + d_{2} d_{3} d_{4}^{2} d_{6} + d_{3}^{2} d_{4}^{2} d_{6} + d_{1}^{2} d_{4} d_{5}d_{6} 
  - d_{1} d_{4}^{2} d_{5} d_{6} + d_{1} d_{4} d_{5}^{2} d_{6} -  d_{4}^{2} d_{5}^{2} d_{6} -  d_{2}^{2} d_{3} d_{6}^{2} \\& -  d_{2} d_{3}^{2} d_{6}^{2} -  d_{1}^{2} d_{4} d_{6}^{2} + d_{2}^{2} d_{4} d_{6}^{2}   + d_{2} d_{3} d_{4} d_{6}^{2} + d_{3}^{2} d_{4} d_{6}^{2} + d_{1}^{2} d_{5} d_{6}^{2} -  d_{1} d_{4} d_{5} d_{6}^{2} + d_{1} d_{5}^{2} d_{6}^{2} -  d_{4} d_{5}^{2} d_{6}^{2}.
  \end{aligned}
\end{equation*}
and $\{\alpha_1,\alpha_2,\alpha_3\}= \{(d_{1} + d_{2} + d_{6}
), (d_{1} + d_{3} + d_{4}), (d_{3} + d_{5} + d_{6})\}$.
\end{example}

\begin{remark}\label{rm:a14tocross}
  By construction, the four roots $\{\alpha, \alpha_1,\alpha_2, \alpha_3\}$  form a root subsystem $R$ of type $\operatorname{A_1}^{\oplus 4}$. Choosing one of these roots determines two root subsytems of type $\operatorname{A_2}^{\oplus 3}$ containing the other three. Furthermore, they are related by the simple reflection induced by our chosen root (see~\cite[Lemma 4.2]{col.van-gee.loo:09}.) Thus, $R$ would induce eight Cross functions (four up to sign), but~\cite[Corollary 4.5]{col.van-gee.loo:09} reveals that this operation produces just one Cross function up to sign, henceforth denoted by  $\Cross{}(R^+)$.  This fact explains the four ways of writing each Cross function as a difference of two Yoshida functions discussed in~\autoref{rm:CrossDifferenceYoshidas}.
\end{remark}

The group $\Wgp$ acts transitively on the Cross functions, so they can all be computed by acting on the one from \autoref{ex:cross}. The action is explicitly described in~\autoref{tab:action}.
There are a total of 270 Cross functions (135 up to sign, listed in~\autoref{tab:crosses}.) Their vanishing loci has the following geometric interpretation.
Recall that an \emph{Eckardt point} on a del Pezzo cubic surface is the point of concurrency of three exceptional curves.
The locus of del Pezzo cubic surfaces with an Eckardt point forms a divisor in $M_{m,3}^\circ$, called the \emph{Eckardt divisor}.
\begin{proposition}\label{prop:cross_eckhard}
  The vanishing locus of the product of all Cross functions in $M_{m,3}^\circ$  is the Eckardt divisor.
\end{proposition}
\begin{proof}
  It suffices to prove the statement on $\P(\fh_6^{\reg})$, where we can do a direct computation.
  Let $p_i = [1:d_i:d_i^3]$ for $i = 1, \dots, 6$ and distinct $d_i$ and let $X$ be the blow-up of $\P^2$ at $p_1, \dots, p_6$.
  Consider the triple of exceptional curves on the cubic surface $X$ given by the proper transforms of the lines $L_{ij} = \overline{p_i p_j}$ for $(i,j) = (1,5)$, $(2,3)$, and $(4,6)$.
  The equation of $L_{ij}$ is
  \[ d_id_j(d_i+d_j) X - (d_i^2+d_id_j+d_j^2)Y + Z  = 0.\]
  The three lines $L_{15}$, $L_{23}$, and $L_{46}$ are concurrent if and only if the determinant of the matrix
  \begin{equation}\label{eq:detM}
    M = 
    \begin{pmatrix}
      d_1d_5(d_1+d_5) & d_1^2+d_1d_5+d_5^2 & 1 \\
      d_2d_3(d_2+d_3) & d_2^2+d_2d_3+d_3^2 & 1 \\
      d_4d_6(d_4+d_6) & d_4^2+d_4d_6+d_6^2 & 1 \\
    \end{pmatrix}.
  \end{equation}
  vanishes.
  Note that $\det M$ is an irreducible homogeneous quintic polynomial. Its negative is the quintic polynomial from~\autoref{ex:cross}, so it  is a factor of the corresponding Cross function.
\end{proof}

\subsection{Markings, exceptional curves and anticanonical triangles}\label{sec:except-curv-antic}
Let $(X, \E1, \dots, \E6)$ be a marked del Pezzo cubic surface.
Express $X$ as the blow-up of $\P^2$ at $p_1, \dots, p_6$ such that the $\E{i}$ is the exceptional divisor over $p_i$. The marking on $X$ yields a decomposition of the  27 exceptional curves on $X$ into three groups~\cite{cle:1866}:
  \begin{itemize}
\item $\E{i}$: the exceptional divisor over $p_i$, for $1 \leq i \leq 6$;
\item $\F{ij}$: the proper transform of the line through $p_i$ and $p_j$, for $1 \leq i \neq j \leq 6$;
\item $\Gc{j}$: the proper transform of the conic through $\{p_1, \dots, p_6 \} \setminus \{p_j\}$, for $1\leq j\leq 6$.
\end{itemize}
Note that the two indices in $\F{ij}$ are \emph{unordered}, namely $\F{ij} = \F{ji}$. For concrete computations, we always choose indices satisfying $i<j$. 

\begin{definition}
  \label{def:anticanonical triangle}
  An \emph{anticanonical triangle} in $X$ is a triple of exceptional curves whose pairwise intersection numbers are one. 
\end{definition}
\begin{remark} \label{rm:TriangleNames}
  Alternative names include \emph{Eckardt triangles} or \emph{tritangent trios}~\cite[Chapter 9]{dol:12}. The terminology provided above is rooted in a simple fact.   The sum of the triple giving an anticanonical triangle is an anticanonical divisor, i.e.\ the zero locus of a section of the anticanonical bundle.
\end{remark}
  There are 45 anticanonical triangles on a cubic surface. 
On a del Pezzo cubic surface marked as above, they come in two flavors:
\begin{itemize}
\item $x_{ij} = \{\E{i}, \F{ij}, \Gc{j}\}$, for $1 \leq i \neq j \leq 6$.
\item $y_{ijklmn} = \{\F{ij}, \F{kl}, \F{mn} \}$, for a tripartition $\{\{i,j\},\{k,l\},\{m,n\}\}$ of   $\{1, \dots, 6\}$.
\end{itemize}
Note that in the first group of 30 triangles, the indices are \emph{ordered},  namely $x_{ij} \neq x_{ji}$. While doing computations involving the second group of 15 triangles, we choose the indices of the tripartition $\{\{i,j\}, \{k,l\}, \{m,n\} \}$ so that  $i < j$, $k < l$, $m < n$, and $i < k < m$.

\begin{remark}\label{rm:markings}
  We let $\sed$ be the set consisting of the 27 symbols $\E{i}$, $\F{ij}$, and $\Gc{j}$, and $\ted$ be the set consisting of the 45 symbols $x_{ij}$  and $y_{ijklmn}$. Our earlier correspondences give two bijections: one between $\sed$ and the set of 27 exceptional curves on $X$, and the second one between $\ted$ and the set of 45 anticanonical triangles on $X$. The action of $\Wgp$ on the markings of $X$ induces an action on both  $\sed$ and $\ted$.
\end{remark}

The computation of  pairwise intersections described above confirms that each exceptional curve on a marked del Pezzo cubic curve with no Eckardt points meets ten others. Furthermore, such curves come in five pairs. The dual intersection complex of the arrangement of 27 exceptional curves is encoded by the 10-regular \emph{Schl\"afli graph} consisting of 27 vertices, 135 edges and 45 hollow triangles. The action of $\Wgp$ on the set of exceptional curves descends to a transitive action on the graph.

\begin{remark}\label{rm:triangleQuintics}
  The proof of \autoref{prop:cross_eckhard} gives a bijection between the 45 anticanonical triangles and the 45 quintic factors of the Cross functions (up to sign.) This correspondence is equivariant with respect to the  action of $\Wgp$.
  As such, it is generated by the identification $y_{152346} \longleftrightarrow \det(M)$, where $M$ is the matrix from~\eqref{eq:detM}.
More intrinsically, the bijection is characterized by the property that the vanishing locus of the quintic associated to an anticanonical triangle is the locus of marked cubic surfaces where the triangle degenerates to a concurrency point of the  three lines. It is for this reason that we call the 45 quintics in the $\Wgp$-orbit of  $\det M$ the \emph{Eckardt quintics}.

A direct computation reveals that each Eckardt quintic appears as a factor of precisely three Cross functions. Our labeling in~\autoref{tab:crosses} describes each such triple as $\{\Cross{3n}, \Cross{3n+1},\Cross{3n+2}\}$ for $n=0,\ldots, 44$, together with the anticanonical triangle indexing the quintic.
\end{remark}

\section{The anticanonical embedding}\label{sec:anti-canon-emb}

In this  section, we give an explicit $\Wgp$-equivariant description of the anticanonical map of the universal cubic surface, and use this to characterize the tropicalized anticanonically embedded universal cubic surface. The following two field extensions of $\KK$ will play a prominent role:
\begin{equation}\label{eq:fields}
  L := \KK( \Yos{i}\colon 0\leq i \leq 39) \subset  \widehat{L} := \KK(\fh_6) =   \KK(d_1,\ldots, d_6). 
\end{equation}
Here, the parameters $d_1,\ldots, d_6$ are algebraically independent over $\KK$ and form a basis of the $\KK$-vector space $\fh_6^*$, as discussed in~\autoref{sec:equiv-uniform}. Note that $L$ is the fraction field of $M_{m,3}^\circ$ and  $\widehat{L}$ is the fraction field of $\fh_6$. The  inclusion $L \hookrightarrow \widehat{L}$ is  induced by the uniformization map $\P(\fh_6^{\reg}) \rightarrow M_{m,3}^{\circ}$ from~\eqref{eq:uniform}.

The key input in our description is the explicit presentation of the Cox ring of the universal del Pezzo cubic surface given in \cite{ren.sha.stu:16}, which we now recall.
Let $X_L \to \spec L$ be the universal marked cubic surface.
Let $\E{i}$, $\F{ij}$, and $\Gc{j}$ be the 27 exceptional curves of $X_L$ as defined in \autoref{sec:yosh-cross-funct}.
We have an isomorphism
\begin{equation}\label{eq:Pic}
  \Pic(X_L) = \Z \langle H, \E{1}, \dots, \E{6} \rangle,
\end{equation}
where $H$ is the pullback of $\O_{\P^2}(1)$ under the map $X_L \to \P^2_{L}$ that blows down $\E{1}, \dots, \E{6}$.

\begin{definition}\label{def:CoxR} The \emph{Cox ring} of $X_L$ is the $\Z^7$-graded $L$-algebra
\begin{equation}\label{eqn:cox}
  \Cox (X_L) = \bigoplus_{(n_0, \dots, n_6) \in \Z^7} H^0(X_L, n_0K_{X_L} + n_1 \E1 + \dots + n_6 \E6).
\end{equation}
\end{definition}\noindent
Each effective divisor in $X_L$ gives an element of $\Cox(X_L)$, well-defined up to scaling.
For any field extension $L\subset L'$ (such as~\eqref{eq:fields}), we define the Cox ring $\Cox(X_{L'})$ as in~\eqref{eqn:cox},  replacing $L$ by $L'$. We have a natural isomorphism
\begin{equation*}\label{eq:CoxL'} \Cox(X_{L'}) = \Cox(X_L) \otimes_L L'.
\end{equation*}

As in~\autoref{rm:markings} we let  $\sed$ be the 27 element set consisting of symbols $\E{i}$ and $\Gc{i}$ for $1 \leq i \leq 6$ and $\F{ij}$ for $1 \leq i < j \leq 6$ giving a marking on the 27 exceptional curves on $X_L$. 
 The polynomial ring 
 $L[\sed]$  has two natural gradings.
The first one is $\Z$-valued, where each variable has degree one.
The other one is $\Z^7$-valued, and it is induced from the isomorphism  $\Pic(X_L) \cong \Z^7$ from~\eqref{eq:Pic}.
Explicitly, under this grading, we have
\begin{equation}\label{eq:Z^7grading}
  \deg \E{i}\! := e_i\; ; \quad \deg \Gc{i}\! := 2e_0 +e_i-\! \sum_{l=1}^6\! e_l
    \;\; (1\leq i \leq 6)\;;\;\; \deg \F{ij}\!:= e_0 -\! e_i -\! e_j \quad (1\leq i<j\leq 6),
\end{equation}
where $e_0,\ldots, e_6$ are the standard basis elements of $\Z^7$.

The next result give an explicit presentation of $\Cox(X_{\widehat L})$:

\begin{theorem}[{\cite[Proposition~2.2]{ren.sha.stu:16}}]
  \label{thm:universalCoxL}
  We have a $\Wgp$-equivariant surjection
  \[ \widehat L[\sed] \to \Cox(X_{\widehat L})\]
  that sends a variable $E \in \sed$ to a generator of $H^0(X_{\widehat L}, \O(E))$.
  The kernel is generated by 270 quadratic trinomials, all of which are $\Wgp$-conjugates (up to sign) of the following one:
  \[ (d_3-d_4)(d_1+d_3+d_4) \E{2}\F{12} - (d_2-d_4)(d_1+d_2+d_4)\E{3}\F{13} + (d_2-d_3)(d_1+d_2+d_3)\E{4}\F{14}.\]
\end{theorem}

\noindent
Note that the surjection $\widehat L[\sed] \to \Cox(X_{\widehat L})$ is compatible with the $\Z^7$-gradings and the $\Wgp$-actions on both sides.

\begin{remark}\label{rm:Saturation}
  Let $\widehat R \subset \widehat L$ be the subring obtained by inverting the 36 roots of $\Phi^+$ in the polynomial ring $\KK[d_1, \dots, d_6]$.
  We have a universal cubic del Pezzo surface $X_{\widehat R} \to \spec \widehat R$, and we can form the Cox ring
  \[ 
    \Cox (X_{\widehat R}) = \bigoplus_{(n_0, \dots, n_6) \in \Z^7} H^0(X_{\widehat R}, n_0K_{X_{\widehat R}} + n_1 \E1 + \dots + n_6 \E6).
  \]
  In~\cite[Proposition~2.2]{ren.sha.stu:16}, the authors show that we have a surjection
  \[
    \widehat R [\sed] \to  \Cox (X_{\widehat R}),
  \]
  whose ideal is generated by the 270 trinomials above, up to saturation by the 27 variables in $\sed$. However, no such saturation is needed over the field $\widehat L$. Indeed, by specializing the variables $d_i$, a \sage~computation certifies the ideal generated by the trinomials is already saturated.
\end{remark}

\begin{definition}\label{def:anticRing}
  The \emph{anticanonical} ring of $X_L$ is the $\Z$-graded $L$-algebra
\[ A(X_L) = \sum_{n \in \Z} H^0(X_L, -n K_{X_L}).\]
Since $K_X$ is anti-ample, the nonzero graded components of $A(X_L)$ are in non-negative degrees.
\end{definition}

As in~\autoref{rm:markings}. we let $\ted$ be the 45 element set consisting of 30 variables $x_{ij}$ for $1 \leq i \neq j \leq 6$ and 15 variables $y_{ijklmn}$ for distinct tripartitions $\{\{i,j\},\{k,l\},\{m,n\}\}$ of $\{1, \dots, 6\}$. We view $\ted$ as the set of markings of the 45 anticanonical triangles on  $X_L$. The group  $\Wgp$ acts on $\ted$.

The next result gives  a presentation of the anticanonical ring of $X_{\widehat L}$ by analogy with~\autoref{thm:universalCoxL}:

\begin{theorem}
  \label{thm:anticanL}
  We have an $\Wgp$-equivariant surjection
  \[\widehat{L}[\ted] \to A(X_{\widehat L}),\]
whose kernel is generated by an $\Wgp$-equivariant (up to sign) set of 270 linear trinomials and 120 cubic binomials.
  Explicitly, the linear trinomials are the $\Wgp$-conjugates of the following one
  \[ (d_3-d_4)(d_1+d_3+d_4) x_{21} - (d_2-d_4)(d_1+d_2+d_4)x_{31} + (d_2-d_3)(d_1+d_2+d_3)x_{41},\]
  and the cubic binomials are the $\Wgp$-conjugates of the following
  \[
    x_{12}x_{23}x_{31} - x_{13}x_{32}x_{21}.
  \]
\end{theorem}
\begin{proof}
    The anticanonical map embeds $X_{\widehat{L}}$ as a cubic hypersurface in $\P^3_{\widehat{L}}$, thus giving an isomorphism
  \begin{equation}\label{eq:isoAcubic}
    A(X_{\widehat L}) \simeq \widehat{L}[X,Y,Z,W] / \langle Q\rangle,
  \end{equation}
  for some cubic polynomial $Q$.     In particular, the anticanonical ring is generated in degree one.

  By  \autoref{thm:universalCoxL}, we have a surjection $\widehat{L}[\sed]_a \to \Cox(X_{\widehat{L}})_a$ on each graded component of degree $a \in   \Z^7$. We are interested in the case $a = (3,-1,-1,-1,-1,-1,-1)$ since $\widehat{L}[\sed]_a = \widehat{L}[\ted]_1$.   For this choice, it follows that  $\Cox(X_{\widehat{L}})_a = A(X_{\widehat{L}})_1$ and
\[ \widehat{L}[\sed]_a = \widehat{L} \langle \E{i}\F{ij}\Gc{j}, \F{ij}\F{kl}\F{mn} \rangle.\]
We let $e_i, f_{ij}, g_j$ in $\Cox(X_{\widehat{L}})_1$ be the images of $\E{i}, \F{ij}, \Gc{j}$, respectively, under the above identifications and maps.  The following commutative diagram 
  \begin{equation}\label{eq:coxtoAnticanonical}
  \begin{CD}
    \widehat{L}[\ted]_1 @= \widehat{L}[\sed]_a \\
    @VVV @VVV \\
    A(X_{\widehat{L}})_1 @= \Cox(X_{\widehat{L}})_a. \\
  \end{CD}
  \end{equation}
ensures that  the products  $e_if_{ij}g_j$ and $f_{ij}f_{kl}f_{mn}$ generate $A(X_{\widehat{L}})$.  We use this to define the surjective map  $\widehat{L}[\ted] \to A(X_{\widehat{L}})$, sending  $x_{ij}$ to $e_if_{ij}g_j$, and $y_{ijklmn}$ to $f_{ij}f_{kl}f_{mn}$.  By construction, the map is $\Wgp$-equivariant and  extends the left vertical arrow in~\eqref{eq:coxtoAnticanonical}.

We let $I$ be the kernel of $\widehat{L}[\ted] \to A(X_{\widehat{L}})$ and $J$ be the kernel of $\widehat{L}[\sed] \to \Cox(X_{\widehat{L}})$. The diagram \eqref{eq:coxtoAnticanonical} yields  $I_1 = J_a$. A degree computation reveals that   for each of the 270 quadric generators $q$ of $J$, there is a unique variable $v \in \sed$ such that $vq$ lies in $J_a$.
For example, for the quadric generator listed in \autoref{thm:universalCoxL}, it is the variable $G_1$.
Therefore, the component $J_a$ is spanned by the $\Wgp$-conjugates of 
\[ (d_3-d_4)(d_1+d_3+d_4) \E{2}\F{12}\Gc{1} - (d_2-d_4)(d_1+d_2+d_4)\E{3}\F{13}\Gc{1} + (d_2-d_3)(d_1+d_2+d_3)\E{4}\F{14}\Gc{1}, \]
which we identify with the linear trinomial in $x_{21}$, $x_{31}$ and $x_{41}$ from  the statement. A \sage~computation, available in the Supplementary material, reveals a total of 270 linear equations in this conjugacy class (up to sign.) These forms generate all linear relations among the chosen global sections of the anticanonical bundle.

From the  map $\widehat{L}[\ted] \to A(X_{\widehat{L}})$, it is easy to check that the cubic $x_{12}x_{23}x_{31} - x_{13}x_{21}x_{32}$ lies in the anticanonical ideal $I$. Therefore, so are its 120 $\Wgp$-conjugates (up to sign.) To conclude, we must show that the 120 cubics and the 270 linear forms generate $I$.

By~\eqref{eq:isoAcubic} we know that  $I$ is principal modulo the 270 linear forms. Thus, any cubic that is nonzero modulo the linear polynomials generates the quotient.
By evaluating at a generic choice of $(d_1,\ldots, d_6)$ in $\P(\fh_6^{\reg})$, we check that this is indeed the case for $x_{12}x_{23}x_{31} - x_{13}x_{21}x_{32}$. 
\end{proof}

\begin{remark}\label{rm:gradingAndAntiCanMap}
A simple inspection of the defining equations confirms that the equivariant map from~\autoref{thm:anticanL} is compatible with the $\Z$- and $\Z^7$-grading on $L[\ted]$ and $A(X_L)$, respectively.
\end{remark}

\textcolor{blue}{Theorems}~\ref{thm:universalCoxL} and~\ref{thm:anticanL} describe the universal Cox ring and the universal anticanonical ring after a base change from $L$ to $\widehat L$. A simple change of variables corresponding to a choice of global sections will allow us to describe the anticanonical ring over $L$, as we now explain.
\autoref{rm:triangleQuintics} gives a  $\Wgp$-equivariant bijection between the set $\ted$ and the set of 45 Eckardt quintics. We fix a choice of signs and denote the quintic corresponding to $x_{ij}$ by $Q_{ij}$ and the one corresponding to $y_{ijklmn}$ by $Q_{ijklmn}$.
Starting from $\ted$, we build a new set $T$ of 45 symbols $X_{ij}$ and $Y_{ijklmn}$, along with a map $T \to A(X_L)$ defined by
\begin{equation}\label{eq:YoshidaAdaptedCoords}
  X_{ij} \mapsto x_{ij} / Q_{ij} \text{ and } Y_{ijklmn} \mapsto y_{ijklmn}/Q_{ijklmn}.
\end{equation}

\begin{remark}\label{rm:actionOnT}
  Recall that $\Wgp$ acts on $\ted$ by permutations and it acts on the 45 quintics $\{Q_{ij}, Q_{ijklmn}\}$ by signed permutations.
  Therefore, $\Wgp$ acts on $x_{ij}/Q_{ij}$ and $y_{ijklmn}/Q_{ijklmn}$ by signed permutations.
  We let $\Wgp$ act on $T$  so that the map \label{eq:YoshidaAdaptedCoords} is equivariant.
  The signs depend on the signs chosen in the bijection betwen $\ted$ and the set of quintics.
  \autoref{tab:action} shows the action explicitly for our choice of signs.
\end{remark}

We use the new variables in $T$ to describe the universal anticanonical ring on the moduli space $M_{m,3}^{\circ}$ without using the uniformization map from~\eqref{eq:uniform}. Coefficients for the equations will involve Cross functions associated to root subsystems of type $A_1^{\oplus 4}$ (see~\autoref{rm:a14tocross}), after fixing a choice of signs that is compatible with the one we picked for $T$ in~\eqref{eq:YoshidaAdaptedCoords}.
\begin{theorem}
  \label{thm:anticanK}
  We have an $\Wgp$ equivariant surjection
  \[ L[T] \to A(X_L)\]
  whose ideal is generated by a $\Wgp$-equivariant set of 270 linear trinomials (up to sign) and 120 cubic binomials (up to sign.) The first group is generated by all $\Wgp$-conjugates of
  \begin{equation*}\label{eq:new_linear}
  \Cross{}(S_1) X_{21} - \Cross{}(S_2) X_{31} + \Cross{}(S_3)X_{41},
\end{equation*}
 whereas the second one is obtained as $\Wgp$-conjugates of the following
  \[
    \Cross{}(S_1')\Cross{}(S_2')\Cross{}(S_3') X_{12}X_{23}X_{31} -   \Cross{}({S_1}'')\Cross{}({S_2}'')\Cross{}({S_3}'') X_{13}X_{21}X_{32}.
  \]
Each $S_i, S_i'$ and ${S_i}''$ is the set of positive roots of an $A_1^{\oplus 4}$ root subsystem, specified in~\eqref{eq:allrootsCross} below.
\end{theorem}

\begin{proof} The statement follows by a direct computation after pre-composing the map from~\autoref{thm:anticanL} with the change of coordinates $L[T]\to \widehat L[\ted]$ from~\eqref{eq:YoshidaAdaptedCoords}. The kernel is $\Wgp$-invariant by construction. The remainder of this proof describes  its generators.
  
  Consider the following nine collections of positive roots of $A_1^{\oplus 4}$ subroot systems in $\EGp{6}$:
\begin{equation}\label{eq:allrootsCross}
  \begin{minipage}{0.53\textwidth}
\[    \begin{aligned}
    &S_1 := \{ d_3-d_4, d_1+d_3+d_4, d_5-d_6, d_1+d_5+d_6\}, \\
    &S_1' :=\{d_3-d_4, d_2+d_3+d_4, d_2+d_5+d_6, d_5-d_6\},\\
    &{S_1}'' \!\!:=\{d_2-d_4, d_2+d_3+d_4, d_3+d_5+d_6, d_5-d_6\},
  \end{aligned}
\]  \end{minipage}
  \begin{minipage}{0.19\textwidth}
  \[  \begin{aligned}
      &S_2 : = (2 3) (S_1),\\
      &S_2' :=(2 1 3)(S_1'), \\
      &{S_2}''\!\!:= (2 3 1)({S_1}''),
  \end{aligned}\]
    \end{minipage}
  \begin{minipage}{0.18\textwidth}
\[    \begin{aligned}
      &S_3 \! :=(2 4)(S_1),\\
      &S_3' \!:=(2 3 1)(S_1'), \\
      &{S_3}''\!\!:= (2 1 3)({S_1}''),
    \end{aligned}\]
  \end{minipage}
 \end{equation}
  where the second and third column subsystems are obtained from the first column ones by applying suitable permutations in $\Sn{6}$ (for an explicit description, see~\autoref{tab:action}.)

  The linear polynomial in~\autoref{thm:anticanL} written in the new $T$-coordinates   becomes:
 \[ (d_3-d_4)(d_1+d_3+d_4)Q_{21} X_{21} - (d_2-d_4)(d_1+d_2+d_4)Q_{31} X_{31} + (d_2-d_3)(d_1+d_2+d_3)Q_{41} X_{41}.\]
 After multiplying throughout by $(d_5-d_6)(d_1+d_5+d_6)$, these three coefficients become the Cross functions $\Cross{}(S_1)$, $\Cross{}(S_2)$ and $\Cross{}(S_3)$. This gives the linear trinomial in the statement. The kernel of the map is generated by all its $\Wgp$-conjugates.
  
The cubic polynomial in~\autoref{thm:anticanL} undergoes a similar transformation.
After rescaling the variables in $\ted$, we obtain the following binomial cubic in $L[T]$:
\begin{equation}\label{eq:newCubic}
  Q_{12}\,Q_{23}\,Q_{31}\,x_{12}\,x_{23}\,x_{31} - Q_{13}\,Q_{21}\,Q_{32}\,x_{13}\,x_{21}\,x_{32}.
\end{equation}
We now multiply throughout by the following degree 12 monomials in the roots of $\EGp{6}$:
\[
  \begin{split}
    P=&(d_1 + d_2 + d_4)(d_1 + d_3 + d_4)(d_1 - d_4)(d_1 + d_5 + d_6)(d_2 + d_3 + d_4)\\
    &(d_2 - d_4)(d_2 + d_5 + d_6)(d_3 - d_4)(d_3 + d_5 + d_6)(d_5 - d_6)^3
  \end{split}
\]
It follows that $P$ can be written in two ways as a product of three quartics arrising from the root subsystems $S_1',S_2',S_3'$, and ${S_1}'', {S_2}'', {S_3}''$, respectively:
\[P = \prod_{i=1}^3(\prod_{\beta\in S_i'} \beta) = \prod_{i=1}^3(\prod_{\beta\in {S_i}''} \beta).
\]
This factorization has the additional property that when multiplied by the quintics in~\eqref{eq:newCubic}, in the given order, each factor produces the Cross function associated to the four roots in  ${S_i}'$ or ${S_i}''$, respectively. For example, $Q_{123456}\prod_{\beta\in S_1'} \beta  = \Cross{}(S_1')$. The binomial cubic in the statement arises in this way. The remaining 120 are obtained by the action of $\Wgp$. 
\end{proof}

\begin{remark}\label{rm:anticComputations}
  The Supplementary material provides the \sage\ computations of
 Yoshida and Cross functions,  the universal anticanonical ring, and the $\Wgp$ action on these objects. The results are collected in \autoref{sec:appendix1}.
In that notation, the linear polynomial in \autoref{thm:anticanK} is 
\[
  \Cross{116} X_{21} - \Cross{2} X_{31} + \Cross{19} X_{41},
\]
Written in terms of  the Yoshida functions it becomes
\begin{equation}\label{eq:YoshidaP3}
  (\Yos{3} - \Yos{37}) X_{21} - (\Yos{20} - \Yos{8}) X_{31} + (\Yos{3} - \Yos{5}) X_{41}.
\end{equation}
The cubic polynomial in \autoref{thm:anticanK} is
\[
  \Cross{9}\Cross{107}\Cross{2}X_{12}X_{23}X_{31} - \Cross{80}\Cross{116}\Cross{86}X_{13}X_{21}X_{32},
\]
and can be described by means of Yoshida functions as follows
\[
  (\Yos{8}-\Yos{11})(\Yos{5}-\Yos{17})(\Yos{20}-\Yos{8})X_{12}X_{23}X_{31} -  (\Yos{8}-\Yos{3})(\Yos{3}-\Yos{37})(\Yos{19}-\Yos{24})X_{13}X_{21}X_{32}.
\]
We generate the anticanonical ideal by applying the $\Wgp$-action described explicitly in~\autoref{tab:action} to these two polynomials. This computation will be essential to prove~\autoref{thm:anticanonicalNoLines} and to describe the arrangement of metric trees from~\autoref{thm:Naruki}, including the data in~\autoref{tab:treeLabeling}. 
\end{remark}

\begin{remark}\label{rm:basisForP3} Over the generic point $\spec \widehat{L}$, the solution set to the 270 linear equations admits a basis where the entries lie in $\hat{R}$. We record it as a $4\times 45$ matrix. A simple computation available in the Supplementary material allows us to re-express each entry as a Laurent monomial in Yoshida and Cross functions. The $4\times 4$-minor of this matrix with columns indexed by the anticanonical triangles $\{x_{12}, x_{13}, x_{21}, x_{23}\}$ has value:
\[
    \frac{\Yos{21}\,\Yos{22}^4\;\Yos{28}\,\Yos{29}\,\Yos{30}^2\,\Yos{31}}
      {\Cross{106}\Cross{11}\Cross{115}\Cross{78}\Yos{17}^2\,\Yos{19}\,\Yos{23}^2\,\Yos{27}^2}.
      \]
      Therefore, a basis for the $\P^3_{\KK}$ linear span of each $X\to\spec \KK$ can be obtained from that of $\P^3_{\hat{R}}$ by specialization of the Yoshidas and Cross functions as long as these do not vanish.
\end{remark}

By construction, if we pick generic $\KK$-values for the Yoshida functions $\Yos{}\in \P^{39}$, then the associated cubic surface $X_{\Yos{}}$ is embedded in $\P^{44}$ by specilization of the equations from~\autoref{thm:anticanK}. For tropicalization purposes discussed in~\autoref{sec:trop-moduli-space}, we must determine these genericity conditions. The following result provides a partial answer to this question:
      \begin{lemma}\label{lm:alwaysGeneric}
 On the set $U = \{\Yos{} \in \GG{40}/\GG{} : \Cross{}\in \GG{135}/\GG{}\}$, each $X_{\Yos{}} \subset \P^{44}_{\KK}$ is determined by specializing the equations from~\autoref{thm:anticanK}.
\end{lemma}
\begin{proof}  Note that $U$ arises from the ring $\hat{R}$ from~\autoref{rm:Saturation}. \autoref{rm:basisForP3} ensures the linear trinomials determine $\P_{\KK}^3$ for $\Yos{} \in U$. 
      To address the validity of the representing binomial cubic equation in the statement on $U$, it is enough to certify that this equation, when restricted to $\P^3_{\KK}$, it remains irreducible and does not vanish everywhere. We write each point in $\P^3_{\KK}$ as a linear combination of our basis $\{v_0,\ldots v_3\}$ with scalars $a_0,\ldots, a_3$. The restricted cubic becomes a cubic $f\in \KK[a_0,\ldots, a_3]$ with exactly eight monomials (all of whose are extremal):
      \[a_0a_1a_2, a_1^2a_2, a_1a_2^2, a_0^2a_3, a_0a_1a_3, a_0a_2a_3, a_1a_2a_3, a_0a_3^2.\]
      Its coefficients are Laurent monomials in Yoshida and Cross functions. Therefore, our original cubic does not vanish along $\P^3_{\KK}$, as we wanted to show.

      To show that the cubic $f \in \KK[a_0,\ldots, a_3]$ remains irreducible we argue by contradiction. A simple inspection of its support forces any factorization to be of the form
      \[f = A( a_0 + B\,a_1 + C\,a_2 + D\,a_3)\, (a_0a_3 + E\,a_1a_2).\]
      A comparison of the resulting coefficients on each side forces three binomial identities on the Yoshida and Cross functions. A simple computation with \sage~allows to re-express one of them as a Laurent monomial in the Yoshida and Cross functions, so this identity never holds when $X_{\Yos{}}$ is smooth and has no Eckardt points. 
\end{proof}

\subsection{The 27 lines on the universal cubic del Pezzo} \label{sec:27-lines-universal} The characterization of the 27 lines on each fiber of $X_{L}\to \spec L$ extends to the universal cubic del Pezzo. In this section, we focus on two particular properties that will play a prominent role in~\autoref{sec:comb-types-tree}. We start by discussing the following classical statement:
\begin{lemma}\label{lm:coverLtoP1} Each  line on a smooth del Pezzo cubic surface admits a 2-to-1 map to $\P^1$.
\end{lemma}
\begin{proof} Consider the anticanonical embedding of $X$ in $\P^{44}$ and view its linear span as $\P^3$. Given a line $\ell$ in $X$, we define
  \[\mathcal{S} = \{\pi \colon \pi \text{ plane in }\P^3, \ell\subset \pi\} \subset \operatorname{Gr}(3,45).
  \]
  By construction, $\mathcal{S}$ is a two-dimensional vector space so  $ \P(\mathcal{S}) \simeq \P^1 $.  Each element $\pi$ of $\mathcal{S}$ produces a curve in $X$, namely the residual plane quadric $C_{\pi} := \overline{(\pi\cap X)\smallsetminus \ell} \subset \pi$. The intersection $C_{\pi}\cap \ell$ consists of two points in $\ell$ (up to multiplicity.)

  The smoothness of $X$ yields a map $X\to \operatorname{Gr}(3,45)$ sending each point $p$ in $X$ to its (translated) tangent plane $T_p(X)$. By  construction, $T_p(X)\in \mathcal{S}$, so it defines a 2-to-1 cover $\varphi\colon \ell\mapsto \P^1$ via $\varphi(p) = [T_p(X)]$.   Each of the five tritangent planes in $X$ containing $\ell$ lies in $\mathcal{S}$ and produce five marked points in $\P^1$. The fiber over each marked point $[\pi]$ is the pair of points in $C_{\pi}\cap \ell$. These points are distinct if $X$ has no Eckardt points.
\end{proof}

The 270 linear trinomials described in~\autoref{thm:anticanK} determine a $\P^3_L$ in $\P^{44}_L$, namely, the linear span of $X_L$ in $\P^{44}_L$. The exceptional curves on $X_L$ become the 27 lines in $X_L$.
Their defining equations can be determined as  follows. Each exceptional curve is contained in exactly five anticanonical triangles. Since each triangle corresponds to a variable of the polynomial ring $L[T]$, each exceptional curve is the vanishing locus of precise five variables in $T$.
For example, the curve $\E{1}$ is the vanishing locus of $X_{12}, \ldots, X_{16}$. Furthermore, the node  $\E{1}\cap \F{12}$ lies in the intersection of the nine hyperplanes indexed by $\{X_{1k}\colon k\neq 1\}\cup \{ X_{21}\}\cup \{Y_{123456}, Y_{123546}, Y_{123645}\}$.   The action of $\Wgp$ allows us to extend this characterization from $\E{1}$ to the remaining 26 lines. We conclude:
\begin{corollary}\label{cor:BoundaryPoints}
  In the absence of Eckardt points, any point on a line on an anticanonically embedded cubic surface $X\subset \P^{44}$ lies in exactly nine coordinate hyperplanes if it is an intersection point of two lines, and on five coordinate hyperplanes otherwise.
  \end{corollary}

\subsection{The boundary of an anticanonical tropical cubic surface}
\label{sec:bound-antic-trop} The proof of~\autoref{thm:anticanonicalNoLines} relies heavily on the rigid combinatorics of the boundary  of each anticanonical tropical cubic surface. By design, the universal embedding from~\autoref{thm:anticanK} ensures that the boundary of $X_L$ consists of precisely 27 lines. Tropicalization  will turn this into an arrangement of tropical lines in the boundary of $\Trop X \subset \TPr^{44}$:

\begin{lemma}\label{lm:boundaryTX} Given an anticanonical triangle $t$ in $\ted$, the  intersection of  $X\subset \P^{44}$ and the hyperplane associated to $t$ is the union of the lines in $\sed$ contained in $t$. The same holds for the tropical surface $\Trop X \subset \TPr^{44}$. 
\end{lemma}
\begin{proof}
  The result follows by the definition of the markings $\ted$ and $\sed$. For example,
\[  X \cap \{X_{12}= 0\} =  \E{1} \cup  \F{12} \cup  \Gc{2}\; \text{ and } \;  \Trop X \cap \{X_{12}=\infty\} = \Trop \E{1} \cup \Trop \F{12} \cup \Trop \Gc{2}.
\]
  The action of $\Wgp$ gives a similar identity for the other anticanonical coordinate hyperplanes.
\end{proof}

Each of the 27 tropical lines at infinity is a metric balanced tree with prescribed directions for its leaf edges (see~\autoref{def:tropLines} for a more precise statement.) Thus, the boundary of $\Trop X$ is an \emph{arrangement of metric trees}.
Our next result shows that the combinatorics of this tree arrangement matches that of the intersection complex of the 27 lines in $X$. For this reason, we refer to an intersection point of two boundary tropical lines as a \emph{nodal point} of the boundary of $\Trop X$. 
\begin{lemma}\label{lm:IntersectionBoundaryLines} Let $X$ be a smooth cubic del Pezzo surface without Eckardt points viewed in $\P^{44}$ via the anticanonical embedding. Then,   the 27 classical lines  in $X$ tropicalize to distinct trees in $\TPr^{44}$. Furthermore, two such trees intersect if and only if their classical counterparts do.
\end{lemma}
\begin{proof}
By~\autoref{cor:BoundaryPoints}, each classical line in $X$ lies in the intersection of the five hyperplanes determined by the anticanonical triangles containing the corresponding line. The same holds for their tropicalization. These 27 quintuples of hyperplanes are all distinct. Hence, so are the 27 trees.

  The statement regarding the pairwise intersection of all trees follows from the fact that if two classical lines do not meet, then the set of anticanonical triangles containing each one of them is disjoint. Our previous discussion characterizing a tree in terms of the five hyperplanes containing it implies that any intersection point between the tropicalization of two disjoint lines will have at least ten  coordinates with value $\infty$. This contradicts the description of the boundary from~\autoref{lm:boundaryTX}. We conclude that the intersection complex of the tropical and classical boundaries agree.
\end{proof}

\section{The Bergman fan of $\EGp{6}$ and the tropical Naruki space}\label{sec:bergman-fan-Naruki-fan}

Tropical projective varieties are obtained from closed subvarieties of projective space via coordinatewise valuations. Our object of interest in this section is the Naruki space $M_{m,3}^*$, which gives a compactification of the moduli space of marked smooth cubic surfaces, and its tropical counterpart: the tropical Naruki space $\Naruki$ from~\autoref{thm:extraLinesApex}. The later was introduced in \cite{hac.kee.tev:09}, and computed explicitly in~\cite[Section 6]{ren.sam.stu:14} and~\cite[Section 3]{ren.sha.stu:16}. In what follows, we describe the construction of $\Naruki$ from~\cite{ren.sam.stu:14}, and endow it with a fan structure suitable for determining valuations of Yoshida and (enough) Cross functions. This information is crucial to determine the combinatorial structure of tropical stable cubic del Pezzo surfaces in $\TPr^{44}$ by means of~\autoref{thm:anticanK}. It plays a central role in the proof of~\autoref{thm:anticanonicalNoLines}.

By~\eqref{eq:naruki_embedding}, $M_{m,3}^*$ admits a closed embedding in $\P^{39}$ via a (signed) choice of 40 Yoshida functions $\Yos{0}, \dots, \Yos{39}$.
The \emph{tropical Naruki space} $\mathcal N$ equals
\begin{equation}\label{eq:tropNS}
  \mathcal N  = \trop \left(M_{m,3}^*\right) \subset \TPr^{39}.
\end{equation}
 We  realized the above closed embedding as the image of the map 
\begin{equation}\label{eq:lm}
  \P(\fh_6) \overset{\ell}\dashrightarrow   \P^{35} \overset{m}\dashrightarrow \P^{39},
\end{equation}
where the map $\ell$ is linear and the map $m$ is monomial.
The 36 coordinates defining $\ell$ are the 36 positive roots $\Phi^+$ of $\fh_6$ (up to sign.) In turn, the $40$ coordinates defining $m$ are (signed) square-free degree nine monomonials in these roots (see~\autoref{sec:yosh-cross-funct}.) Our choice of signs is given in Tables~\ref{tab:roots} and \ref{tab:yoshida}. For tropicalization purposes, we represent $m$ via its $40\times 36$-matrix of exponents  (with $0/1$ entries.) We refer to it as the \emph{Yoshida matrix}.  The map~\eqref{eq:lm} is well-defined on $\P(\fh_6^{\reg})$, i.e., in the complement of the $\EGp{6}$-reflection arrangement in $\P^5$.

\begin{remark}\label{rm:index3oFYoshidasInDs} A simple computation in \sage, available in the Supplementary material, confirms that the  Yoshida matrix  has rank 16.
  Furthermore, its rows span a sublattice of $\Z^{36}$ of index three. This implies that, when analyzing if products of roots yield a monomial in the Yoshida functions, we will often need to use cube-roots of Yoshidas. 
\end{remark}

The linear-monomial factorization of the map in~\eqref{eq:lm} is ideally suited for tropicalization~\cite[Theorem 3.1]{dic.fei.stu:07} when we restrict our map to $\P(\fh_6^{\reg})$. The image of the tropical map $\trop(\ell)$  becomes the Bergman fan  $\cB$ of the $6\times 36$ matrix encoding the $\EGp{6}$-arrangement. It is a five-dimensional simplicial polyhedral fan in $\R^{36}/{\rspanone}$~\cite{ard.kli:06}. In turn, the monomial map $m$ becomes right-multiplication by the Yoshida matrix under tropicalization. Since all Yoshida functions are non-zero in the $\EGp{6}$-arrangement complement, we conclude
\begin{equation}\label{eq:narukiFan}
  \Naruki = \trop(m)(\cB)\subset \R^{40}/\rspanone \subset \TPr^{39}.
\end{equation}

\begin{remark}\label{rm:Modularinterpretation}
  Since the 40 coordinates in $\TPr^{39}$ correspond to the 40 Yoshida functions listed in~\autoref{tab:yoshida}. Each point in $\Naruki\subset \TPr^{39}$ records  the valuations of the Yoshida functions associated to a point in the  classical Naruki space $M_{m,3}^* \subset \P^{39}$ over the valued field $\KK$. It is in this sense that it plays the role of a tropical moduli space. In~\autoref{sec:comb-types} we characterize it as the moduli space of stable anticanonically embedded tropical smooth cubic surfaces with no Eckardt points.
\end{remark}

The Bergman fan $\cB$ of the matroid $\EGp{6}$ was explicitly computed in~\cite[Lemma 3.1]{ren.sha.stu:16} as a projection of the Bergman fan of the $\EGp{7}$-reflection arrangement in $\P^6$ via determining all its 100\,662\,348 circuits. We choose an alternative approach, realizing $\cB$ as the cone over the \emph{nested set complex} $N(\EGp{6},\Irr)$ associated to the minimal building set $\Irr$ of all proper irreducible flats of the matroid $\EGp{6}$ (see~\cite[Theorem 1.2]{ard.rei.wil:05/07} and~\cite[Theorem 4.1]{fei.stu:05}.) 

The abstract complex $N(\EGp{6},\Irr)$ consists of nested collections
of flats in the geometric lattice of partially ordered flats of the
matroid, ordered by inclusion. The vertices correspond to the
\emph{irreducible flats} (those which cannot be decomposed as a
product.)  Higher-dimensional cells are determined by nested families
of flats. A collection $\cF\subset \Irr$ is \emph{nested} if for every
antichain $\{\cF_1,\ldots, \cF_r\}$ in $\cF$ with $r\geq 2$, the
\emph{join flat} $\cF_1\vee \ldots \vee \cF_r$ does not belong to
$\Irr$ nor it equals $\Phi^+$. The join flat represents the subspace
obtained by intersecting the given subspaces in the
$\EGp{6}$-arrangement.

By \cite[Theorem~3.1]{bar.ihr:99}, there is a one-to-one correspondence between the lattice of flats of the $\EGp{6}$ reflection arrangement (ordered by reverse inclusion) and the poset of parabolic subgroups of $\Wgp$. 
In turn, parabolic subgroups of $\Wgp$ are determined by the root subsystems of $\EGp{6}$. There are precisely $15$ isomorphism classes of root subsystems, and each class consists of a single $\Wgp$-orbit. 
The proper irreducible flats come from the seven connected proper root subsystems of $\EGp{6}$, namely  $\operatorname{A}_i$ for $i = 1, \dots, 5$, $\operatorname{D}_4$, and $\operatorname{D}_5$, as shown in~\autoref{fig:subsystems}. 
Our labeling  of the 15  representatives is compatible with \cite[Table~4]{ren.sha.stu:16}. In total, there are $750$ proper irreducible flats.

To determine $\cB$ we must embed $N(\EGp{6},\Irr)$ in $\R^{36}/\rspanone$, we start by fixing an ordering $\{r_1,\ldots, r_{36}\}$ for the
positive roots $\Phi^+$ of $\EGp{6}$. We realize the vertices of $N(\EGp{6},\Irr)$ as the 0-1 incidence vector $\sum_{i\in \cF} e_i \in \R^{36}/\rspanone$ of each irreducible flat $\cF$, where $e_i$ denotes the $i$th.\ canonical basis element.  The  simplices of $\cB$ are realized as the positive span of the corresponding vertices in $N(\EGp{6},\Irr)$.

\begin{figure}[htb]
  \includegraphics[scale=0.25]{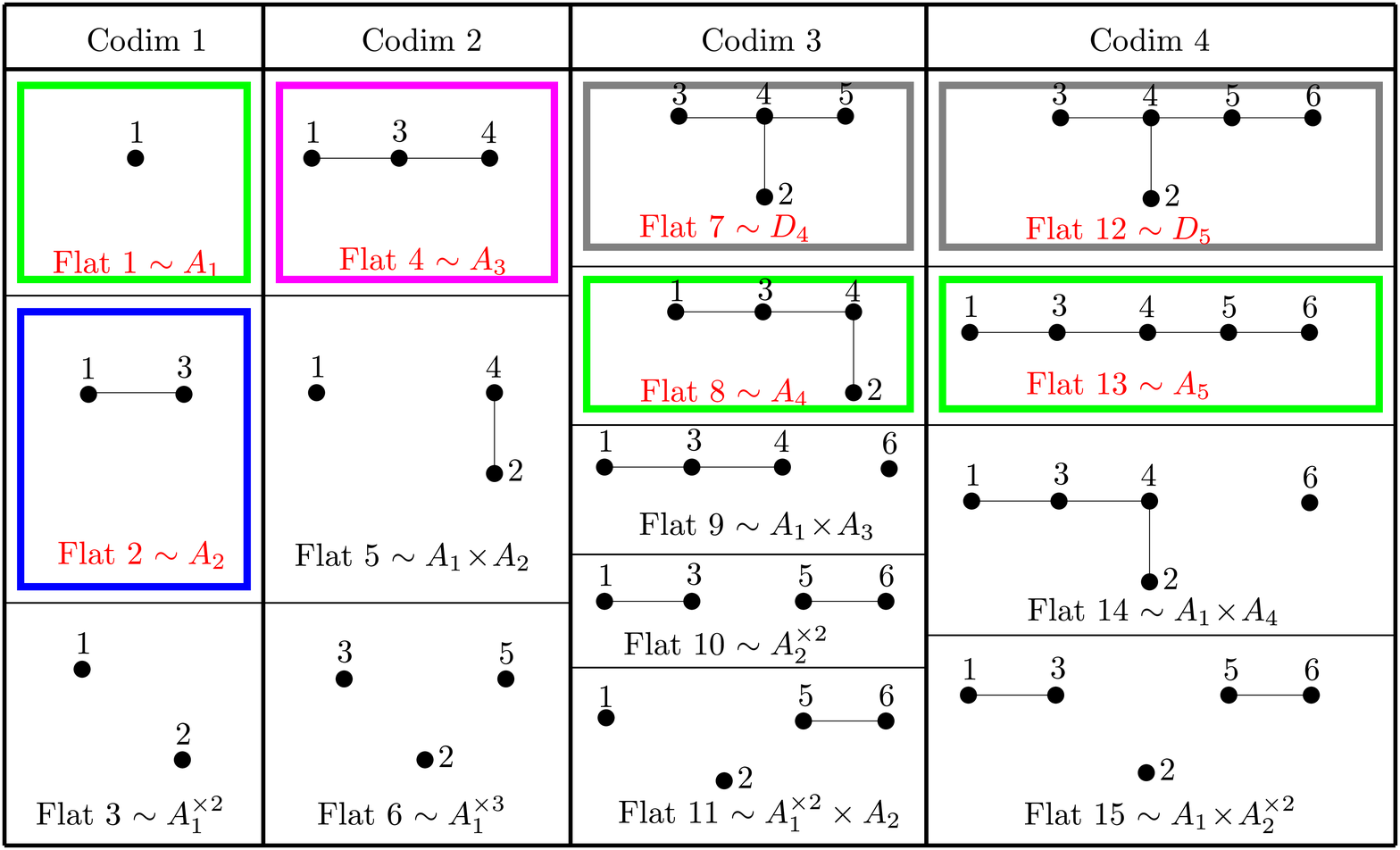}
  \caption{All $\Wgp$-representatives of flats and their associated root subsystems (following the labeling of~\autoref{fig:E6label}.) The boxed
  ones are proper irreducible flats. Their images under the Yoshida matrix are: $\Fl_1, \Fl_8, \Fl_{13}\mapsto \operatorname{(a)}$, $\Fl_2\mapsto \operatorname{(b)}$, $\Fl_4\mapsto$ ($\operatorname{a_2}$), and $\Fl_7, \Fl_{12}\mapsto \mathbf{0}\in \TP\P^{39}$.\label{fig:subsystems}}
\end{figure}

We use the above information to construct the Bergman fan $\cB$ inductively, exploiting the $\Wgp$-action at each step, starting from the $750$ vertices divided into seven orbits. In each iteration, we input the orbit representatives of cells of dimension $k$ and produce  cells of dimension $k + 1$ by testing whether adding a vertex to  a given cell produces an irreducible flat. We then use the $\Wgp$-action to produce all cells of dimension $k +1$ and output a list of its orbit representatives.
We optimize the process for $k = 2$ and above exploiting our knowledge of the  edge-vertex adjacency graph of the complex (computed in the first iteration.) Rather than testing all vertices of the complex against each low-dimensional cell, we restrict our search to vertices that form an edge with every vertex in the input cell. This significantly speeds up the computation. The Bergman fan is obtained after four iteration. The number of orbits in each dimension is given by the vector $(7, 24, 47, 49, 21)$. We refer to the Supplementary material for implementation details, running times  and a list of orbit sizes in each dimension.

\smallskip

Having constructed the Bergman fan, it is straightforward to determine the support of the  Naruki fan $\Naruki$: we simply apply the Yoshida matrix to each cone in $\cB$ as in~\eqref{eq:narukiFan}. A fan structure for this set was already computed in \cite{ren.sha.stu:16}. The $\Wgp$ symmetry of both $\cB$ and $\Naruki$ allows us reduce our computations to orbit representatives on both sides. 

We build $\Naruki$ starting with the rays. The image of the seven orbit representatives of rays in $\cB$ gives three orbits of rays in $\Naruki$, which we call $\operatorname{(a)}$, $\operatorname{(b)}$ and $\operatorname{(a_2)}$, following~\cite[Lemma 3.3]{ren.sha.stu:16}. By construction, $\operatorname{(a_2)}$ rays are sums of two $\operatorname{(a)}$ rays each. There are a total of 36 $\operatorname{(a)}$ rays, 40 $\operatorname{(b)}$ rays and 270 $\operatorname{(a_2)}$ rays. \autoref{fig:subsystems} shows that $\operatorname{(a)}$ rays arise from three root subsystems (of types $\operatorname{A}_1$, $\operatorname{A}_4$ and $\operatorname{A}_5$) whereas the $\operatorname{(b)}$ and $\operatorname{(a_2)}$ rays each come from a single root system (of types $\operatorname{A}_2$ and $\operatorname{A}_3$, respectively.) The remaining two orbits of rays in $\cB$ map to $\mathbf{0}$ under the Yoshida matrix.

To determine the support of $\Naruki$, it suffices to consider the
images of top-dimensional cells $\sigma$ of $\cB$ that are maximal
with respect to inclusion.  A direct computation shows that the
collection $\trop(m)(\sigma)$ of such cells is the union of two
$\Wgp$-orbits, confirming the results of~\cite[Section
  3]{ren.sha.stu:16}.  The first orbit corresponds to the image of a
maximal cone spanned by four rays of type $\operatorname{(a)}$.  The
second orbit corresponds to the image of a maximal cone spanned by
three rays of type $\operatorname{(a)}$ and one ray of type
$\operatorname{(b)}$. All these cones have dimension four, as we
expected from~\eqref{eq:tropNS}. We let $\scrC$ be the collection of cells in these two orbits. Cones on the first orbit will be refered to as cones \emph{of the first type}. We use the terminology  cones \emph{of the second type} for those in the second orbit.

Although the support of $\Naruki$ is the union of all cones in $\scrC$, this collection does not provide a fan structure on it---many pairs of cones intersect along non-faces.
To remedy the situation, for each  cone $\sigma \in \scrC$, we
consider the set $\scrC(\sigma)$ of all cones in $\scrC$  whose intersection with $\sigma$ is not a proper face of $\sigma$. A direct computation shows that $\scrC(\sigma)=\{\sigma\}$  for all cones $\sigma$ in $\scrC$ of the second type. In turn, if $\sigma$ is of the first type, then $\scrC(\sigma)$ consists of ten cones, all of which are of the first type as well.
Furthermore, we observe that the union of all cones in $\scrC(\sigma)$ is itself a simplicial cone, say $\widehat{\sigma}$, spanned by the images of four rays of type $\operatorname{(a)}$.
These cones $\widehat{\sigma}$ along with the cones of the second type give a fan structure to the tropical Naruki space.
Consistent with our previous notation and that of~\cite{ren.sha.stu:16}, we say the cones $\widehat{\sigma}$  are of type $\operatorname{(aaaa)}$. Similarly, we refer to the cones of the second type as $\operatorname{(aaab)}$ cones.

We endow $\Naruki$ with a refinement of the fan structure on the tropical Naruki space described above, and refer to it henceforth as the \emph{Naruki fan}.
It is obtained by taking the barycentric subdivision of the cones of type $\operatorname{(aaaa)}$ and the induced subdivision of the cones of type $\operatorname{(aaab)}$.
It is easy to check that the 24 resulting subcones of a cone of type $\operatorname{(aaaa)}$ lie in a single $\Wgp$-orbit, and so do the six resulting subcones of a cone of type $\operatorname{(aaab)}$.
Thus, the maximal cones of the Naruki fan are again divided into two $\Wgp$-orbits.
One orbit consists of cones of type $\operatorname{(aa_2a_3a_4)}$, that is, spanned by four rays of type $\operatorname{(a)}$, $\operatorname{(a_2)}$, $\operatorname{(a_3)}$, and $\operatorname{(a_4)}$, where a ray of type $\operatorname{(a_i)}$ is the sum of $i$ rays of type $\operatorname{(a)}$. The second orbit consists of  type $\operatorname{(aa_2a_3b)}$ cones. There are a total of $77\,760$ cones of type $\operatorname{(aa_2a_3a_4)}$, and $38\, 880$ of type $\operatorname{(aa_2a_3b)}$.
This structure agrees with the one described in~\cite[Lemma 3.3]{ren.sha.stu:16}. A suitable refinement of $\cB$ turns  $\trop\,(m)\colon \cB\to \Naruki$ into a map of fans.

\begin{remark}\label{rm:BergmanReps} For concrete computations, including those undertaken in~\autoref{sec:comb-types-tree}, it is useful to fix  primitive rays $\operatorname{(a)}$, $\operatorname{(a_2)}$, $\operatorname{(a_3)}$, $\operatorname{(a_4)}$ and $\operatorname{(b)}$ spanning adjacent maximal cone representatives in the Naruki fan. We choose:
  \begin{equation*}\label{eq:NarukiReps}
    \begin{aligned}
      \operatorname{a}\,&\!:= (0, 0, 0, 0, 0, 0, 1, 0, 0, 0, 0, 0, 0, 0, 0, 0, 0, 0, 0, 0, 1, 0, 1, 1, 0, 0, 1, 0, 0, 0, 1, 0, 0, 0, 0, 1, 1, 1, 1, 0),\\
      \operatorname{a_2}&\!:= (0, 1, 0, 0, 1, 0, 1, 0, 0, 0, 0, 1, 0, 0, 0, 0, 0, 0, 1, 0, 1, 1, 2, 1, 0, 0, 2, 0, 0, 0, 1, 0, 0, 0, 0, 1, 1, 2, 2, 1),\\
      \operatorname{a_3}&\!:= (0, 2, 1, 0, 1, 0, 1, 0, 0, 0, 0, 1, 0, 0, 0, 0, 1, 0, 2, 0, 1, 1, 3, 1, 0, 1, 2, 0, 0, 0, 2, 1, 0, 0, 0, 2, 1, 2, 3, 1),\\
      \operatorname{a_4}&\!:= (1, 4, 2, 1, 2, 1, 2, 1, 1, 1, 1, 2, 1, 1, 1, 1, 2, 1, 4, 2, 2, 2, 4, 2, 1, 2, 4, 2, 2, 2, 4, 2, 1, 1, 1, 4, 2, 4, 4, 2),\\
      \operatorname{b}\,&\!:= (1, 1, 1, 1, 1, 1, 1, 1, 1, 0, 1, 1, 1, 0, 0, 0, 1, 1, 1, 0, 1, 1, 3, 1, 0, 1, 1, 0, 0, 0, 1, 1, 0, 0, 0, 1, 1, 1, 1, 1).
    \end{aligned}
\end{equation*}
We use them to build two neighboring cones representing the two orbits of maximal cones in the tropical Naruki space:

\begin{equation}\label{eq:coneReps}
\operatorname{(aa_2a_3a_4)} = \R_{\geq 0}\langle a,a_2,a_3,a_4\rangle \quad \text{ and }\quad \operatorname{(aa_2a_3b)} = \R_{\geq 0}\langle a,a_2,a_3,b\rangle.
\end{equation}

The following vectors in the Bergman fan lie in the fibers of the Yoshida matrix over each of the five rays $\operatorname{a}, \ldots, \operatorname{a_4}$ and $\operatorname{b}$, respectively:

\begin{equation*}\label{eq:BergmanReps}
v_{\operatorname{a}}\!:= e_0, \; v_{\operatorname{a_2}}\!:= e_0 + e_1, \; v_{\operatorname{a_3}}\!:= e_0 + e_1 + e_4, \;
v_{\operatorname{a_4}}\!:= 2(e_0 + e_1) + e_{18} + e_{23} + e_{24} + e_{25}, \;
v_{\operatorname{b}}\!:= e_0 + e_2 + e_9.
\end{equation*}
\end{remark}

Recall from~\autoref{prop:cross_eckhard} that the 135 Cross functions characterize the Eckardt divisor in $M_{m,3}^{\circ}$. These functions also appear as coefficients of generators of the anticanonical ideal defining the universal cubic surface. Thus,  any fan structure on the tropical Naruki space parameterizing tropical cubic surfaces in $\TPr^{44}$ must be compatible with the tropicalization of these Cross function. Given a Cross function $C$, we compute $\trop C$ as follows. We write $C$ as a signed difference of two Yoshida functions  $ C = \pm(\Yos{i }- \Yos{j})$, as described in~\eqref{eqn:Cross_def}. The corner locus of $\trop(C)$ in $\Naruki$ is the closed subset defined by  $\val(\Yos{i}) = \val(\Yos{j})$. This corresponds to the equality of the $i$th and $j$th coordinates in  $\TPr^{39}$.

Our next result result confirms that the Naruki fan  $\Naruki$ described above gives a fan structure to the tropical Naruki space that is well-adapted to the Cross functions, i.e. their valuations will be piecewise linear on all maximal cones of $\Naruki$:

\begin{proposition}\label{prop:naruki_fan_cross}
  Let $C$ be a Cross function.  The corner locus of $\trop(C)$ in the fan $\Naruki$ is the support of a subfan of the Naruki fan. 
\end{proposition}
\begin{proof}
  Following earlier conventions, for each $k=0,\ldots, 39$ we let $y_{k} = \val(\Yos{k})$ be the $k$th coordinate of $\TPr^{39}$. Let $\sigma$ be a cone of the Naruki fan.
  We must to show that $\sigma \cap (y_i = y_j)$ is a subcone of $\sigma$.
  By induction on $\dim(\sigma)$, it suffices to show that either $\sigma \subset (y_i = y_j)$ or that the relative interior of $\sigma$ is disjoint from the hyperplane $(y_i = y_j)$.   We check this by direct computation with \sage, available in the Supplementary material.
  For every cone $\sigma$ of the Naruki fan, we iterate over all four expressions $\pm(\Yos{i} - \Yos{j})$ giving the Cross function $C$, and verify that the intersection of the  hyperplane $(y_i = y_j)$ with $C$ is either $\sigma$ itself or a smaller dimensional cone.
\end{proof}

\begin{remark}\label{rm:expValsCross}
  The previous result highlights some interesting features of Cross functions. Given a cone $\sigma$ in $\Naruki$, we denote its relative interior by $\sigma^{\circ}$. Since there are four ways of writing a Cross function as differences of Yoshidas, it is possible that on a given cone $\sigma$ we get $\sigma \subset (y_i=y_j)$ for one of these expressions, whereas $\sigma^{\circ}\cap (y_k=y_l) = \emptyset$ for some other expression. Indeed, whenever $\Cross{i}=\pm(\Yos{k} - \Yos{l})$ and $\val(\Yos{k}) \neq \val(\Yos{l})$ for a single point in $\sigma^{\circ}$, we have $\val(\Cross{i})=\min\{\val(\Yos{k}), \val(\Yos{l})\}$ for all points in $\sigma^{\circ}$. When all four expressions yield ties along $\sigma^{\circ}$, we have
  \begin{equation}\label{eq:expval}
  \val(\Cross{i})_{|\sigma^{\circ}} \geq \max\{ \min\{\val(\Yos{k}), \val(\Yos{l})\}: \Cross{i}= \pm(\Yos{k} - \Yos{l}), \val(\Yos{k}) = \val(\Yos{l}) \text{ on }\sigma^{\circ}\}.  
    \end{equation}
    We refer to the  right-hand side of~\eqref{eq:expval} as the \emph{expected valuation} of $\Cross{i}$.
\end{remark}

A direct computation available in the Supplementary material shows that the valuation of most Cross functions can be determined over a fixed positive-dimensional cone in the Naruki fan. But the behavior varies as we travers the fan. No ambiguities arise on  the relative interior of the  $\operatorname{(aa_2a_3b)}$ cones. However, on the relative interior of any  $\operatorname{(aa_2a_3a_4)}$ cone, the valuations of all Cross functions can be determined, with  three exceptions. The precise triple depends on the input cone. For example, for our chosen $\operatorname{(aa_2a_3a_4)}$ representative from~\autoref{rm:BergmanReps} we can predict all but the valuations of $\Cross{36}, \Cross{37}, \Cross{38}$.

The explicit computation of the fiber over $\operatorname{(aa_2a_3a_4)}^{\circ}$ on $\cB$  will allow us to show that the valuation of  $\Cross{37} = \Yos{34}-\Yos{8}$ (see~\autoref{tab:crosses}) agrees with the expected one, i.e.\ $\Trop(\Cross{37}) = \min\{\val(\Yos{34}), \val(\Yos{8})\}$.
 As we move to the boundary of both cones, more Cross functions will have undertermined valuations and they may or may not agree with the expected valuation. The extremal case is given by the apex of $\Naruki$ where no Cross function valuations can be established. We return to this fact in~\textcolor{blue}{Sections}~\ref{sec:trop-lines-trop} and~\ref{sec:comb-types-tree}.

 \smallskip
 
 We end this section by discussing the fibers of the Yoshida map $\trop(m)\colon \cB\to \Naruki$ on smooth points of $\Naruki$, that is points in the relative interior of the two maximal cone representatives from~\eqref{eq:coneReps}. Here is the precise statement. It  will play a central role in~\autoref{lm:Cr37}. 
 \begin{proposition}\label{pr:fibersOfYoshida}
   Each fiber of the Yoshida matrix over a smooth point of the Naruki fan is a unions of 66 cones. Each component is contained in a maximal cone of the Bergman fan and need not be open nor closed in the quotient topology of $\R^{36}/\rspanone$. Their closures have between five and seven extremal rays. 
   \end{proposition}
 \begin{proof} The result follows by a direct computation, available  in the Supplementary material. Next, we describe the process for points in the relative interior of $\sigma := \operatorname{(aa_2a_3a_4)}$ since this is the relevant cone for the proof of~\autoref{lm:Cr37}. The method works verbatim for the $\operatorname{(aa_2a_3b)}$ cone representative. Throughout, we fix the following convention to pick a canonical representative for any vector in  $\R^{40}/\rspanone$: all its coordinates are non-negative and at least one of them must vanish. We use the same convention for representatives in $\R^{36}/\rspanone$.

   Given any maximal cone $\tau$ in $\cB$, we determine whether  $\Ccurve_{\tau} := \trop(m)(\tau) \cap \sigma^{\circ}$ is non-empty by checking if  the baricenter of $\trop(m)(\tau) \cap \sigma$ lies in $\sigma^{\circ}$. This yields a total of 66 valid cones $\tau$, involving 53 rays of the Bergman fan $\cB$. The computation for $\operatorname{(aa2a3b)}$ gives the same number of cones, but only 52 rays.

   For each such $\tau$ we wish to compute $\Ccurve_{\tau}$. To this end, we determine which positive linear combinations of the five rays  of $\tau$ map to $\sigma^{\circ}$. Since $\sigma$ is simplicial, each ray $\rho$ yields a unique  expression 
   \begin{equation}\label{eq:imagerho}
     \trop(m)(\rho) = r_1^{\rho} \operatorname{a} + r_2^{\rho} \operatorname{a_2} + r_3^{\rho}\operatorname{a_3} + r_4^{\rho} \operatorname{a_4}.
   \end{equation}
   Notice that $\trop(m)(\rho)$ lies in the vector spanned by $\sigma$ but need not be in the cone $\sigma$. Thus, some scalars $r_i^{\rho}$ might be negative.

   The preimage of $\sigma^{\circ}$ under $\trop(m)$ restricted to $\tau$ is characterized as those linear combinations $\sum_{j=1}^5z_j\rho_j$ with non-negative scalars $z_1,\ldots, z_5\in \R$ subject to the following constraints:
   \begin{equation}\label{eq:constraintFibers}
     \sum_{j =1}^5 r_i^{\rho_j}z_j >0 \qquad \text{ for all } i=1,\ldots, 4.
   \end{equation}
   The closure of the space of solutions to~\eqref{eq:constraintFibers} is a polyhedron in $\R^5$. Using \sage~we determine its spanning rays. In turn, this data generates the extremal rays of the the closure of the cone $\tau \cap \trop(m)^{-1}(\sigma^{\circ})$ in $\R^{36}/\rspanone$. Their number varies between five and seven. The same calculation certifies that   $\Ccurve_{\tau}\!=\!\sigma^{\circ}$ for all 66 possible cones $\tau$. Furthermore, the  fiber over $\sigma^{\circ}$ always meets $\tau^{\circ}$. 
 \end{proof}
 \begin{remark}\label{rm:sampleTausForCr37} We carry out the computations in the proof of~\autoref{pr:fibersOfYoshida} for two of the 66 cones associated to the fiber over $\operatorname{(aa_2a_3a_4)}^{\circ}$, which we call $\tau_0$ and $\tau_1$. The output of these calculations will be used in the proof of~\autoref{lm:Cr37}.   We start by listing the rays spanning the closure of each cone, following our notation for the rays in the Bergman fan:
  \begin{equation}\label{eq:taus}
    \tau_0 = \R_{\geq 0}\langle \rho_{0}, \rho_{1}, \rho_{4}, \rho_{204}, \rho_{540}\rangle, \qquad \tau_0 = \R_{\geq 0}\langle \rho_{0}, \rho_{1}, \rho_{4}, \rho_{204}, \rho_{543}\rangle.
  \end{equation}
  Here, $\rho_i = e_i$ for $i=0,1,4$, whereas $\rho_{204}  :=e_1 + e_3+e_4+e_7+e_{10}+e_{13}$ and 
  \begin{equation*}\label{eq:raysTaus}
    \begin{aligned}
            \rho_{540} &:= e_1 + e_3 +e_4 + e_5 + e_7 + e_8 +e_{10} + e_{13} + e_{14} +e_{17}\,    , \\ \rho_{543} &:=   e_{1} +  e_{3} +  e_{4} +  e_{7} +  e_{10} +  e_{13} +  e_{32} +  e_{33} +  e_{34} +  e_{35}.
    \end{aligned}
  \end{equation*}
In the notation of~\autoref{fig:subsystems}, the rays $\rho_0, \rho_1$ and $\rho_4$ correspond to the flat $\cF_{13}$, whereas the ray $\rho_{204}$ comes from the flat $\cF_4$. The remaining two rays  are associated to the flat $\cF_8$.

Following~\eqref{eq:imagerho}, we write the image of each ray under $\trop(m)$ in the basis $\{\operatorname{a}, \operatorname{a_2}, \operatorname{a_3}, \operatorname{a_4}\}$:
\begin{equation*}\label{eq:images}
  \begin{aligned}
    \trop(m)(\rho_0) & = (1,0,0,0), \;\qquad \trop(m)(\rho_1) = (-1,1,0,0),\quad  \trop(m)(\rho_4) = (0,-1,1,0),\\
      \trop(m)(\rho_{204}) & = (1,0,-1,1), \!\quad \trop(m)(\rho_{540}) = \trop(m)(\rho_{543}) = (2,0,0,0).
  \end{aligned}
  \end{equation*}
The inequalities in the variables $z_1,\ldots, z_5$ listed in~\eqref{eq:constraintFibers} yield a polyhedron  in $\R^5$ with seven rays:
\begin{equation}\label{eq:raysR5}
  (0,1,1,1,0), (1,0,0,0,0), (1,1,0,0,0), (1,1,1,0,0), (0,2,2,0,1), (0,0,0,0,1) \,\text{and}\, (0,2,0,0,1).
\end{equation}
The seven extremal rays of the closure of  $\tau_i':=\tau_i\cap \trop(m)^{-1}(\operatorname{(aa_2a_3a_4)}^{\circ})$ in $\R^{36}/\rspanone$ are obtained by multiplying these seven vectors by the five spanning rays of each $\tau_i$. We write them in the same order as those in~\eqref{eq:raysR5}. To ensure the vector lies in the cone  $\tau_i'$ the scalars $p_0,\ldots, p_6$ associated to a point in the linear span of $\tau_i'$ must satisfy:
\begin{equation}\label{eq:ineqstausprimes}
p_1 + 2p_5>0, \quad p_2 + 2\,p_6 >0, \quad p_2+ 3\,p_4>0,\quad  p_0> 0 \;\; \text{ and } \; \; p_i\geq 0 \text{ for all } i=1,\ldots, 6.
\end{equation}
 \end{remark}

\section{Tropical convexity and collinearity  in $\TPr^{n-1}$}\label{sec:tropical-lines-TPn}

In this section we describe the convex structure on tropical lines in tropical projective space. Our main result characterizes collinearity of points in $\TPr^{n-1}$ in terms of vanishing of tropical $3\times 3$ minors. This result extends previous work of Develin-Santos-Sturmfels~\cite{dev.san.stu:05} from the tropical projective torus to its compactification. Furthermore, it yields an algorithm for reconstructing a tropical line from its points at infinity. These techniques will be central to proving~\autoref{thm:anticanonicalNoLines} and the last claim in ~\autoref{thm:Naruki}, describing the metric structure of the tropical lines in the boundary of each tropical cubic surface $\Trop X\subset \TPr^{44}$, as we traverse the Naruki fan $\Naruki$.

We start by recalling the definition of a tropical line in $\TP^{n-1}$. They are all obtained as tropicalizations of classical lines in $\P^{n-1}$ (see~\cite[Theorem 3.8]{spe.stu:04}.) Equivalently, they arise from (realizable) rank-two valuated matroids on $n$-elements.
 By working with the toric structure of $\P^{n-1}$ it suffices to only consider tropical lines meeting the interior of $\TP^{n-1}$. As usual, for each $i=0,1,\ldots, n-1$,  $e_i$ denotes the $i$th.\ canonical basis vector in $\R^{n}$.
\begin{definition}\label{def:tropLines} A \emph{tropical line} meeting the interior of $\TPr^{n-1}$ is a balanced metric tree with at most $n$ leaves attached to unbounded edges or legs of prescribed directions determined by a partition of $\{0,\ldots, n-1\}$. More precisely, if the tree has $m$ leaves, the legs have directions $e_{B_j}\!:=\sum_{i\in B_j} e_i$ for $j=1,\ldots, m$, where each $B_j$ is non-empty and the sets $B_1,\ldots, B_m$ partition $\{0,\ldots, n-1\}$.
  A tropical line will be \emph{generic} if it has exactly $n$ leaves.
\end{definition}
\noindent

The collections $B_1,\ldots, B_m$ record the coordinates of each leaf having value $\infty$. Thus, all leaves lie in the relative interior of distinct cells in $\TPr^{n-1}$.   \autoref{fig:tropLine} gives an example of two tropical lines in $\TPr^3$ meeting its interior. Tropical lines in the boundary of $\TPr^{n-1}$ will be viewed as meeting the interior of a suitable boundary cell $\TPr^{s-1}\subset \TPr^{n-1}$.

\begin{figure}[htb]
  \centering
  \begin{minipage}[c]{0.3\linewidth}
    \includegraphics[scale=0.33]{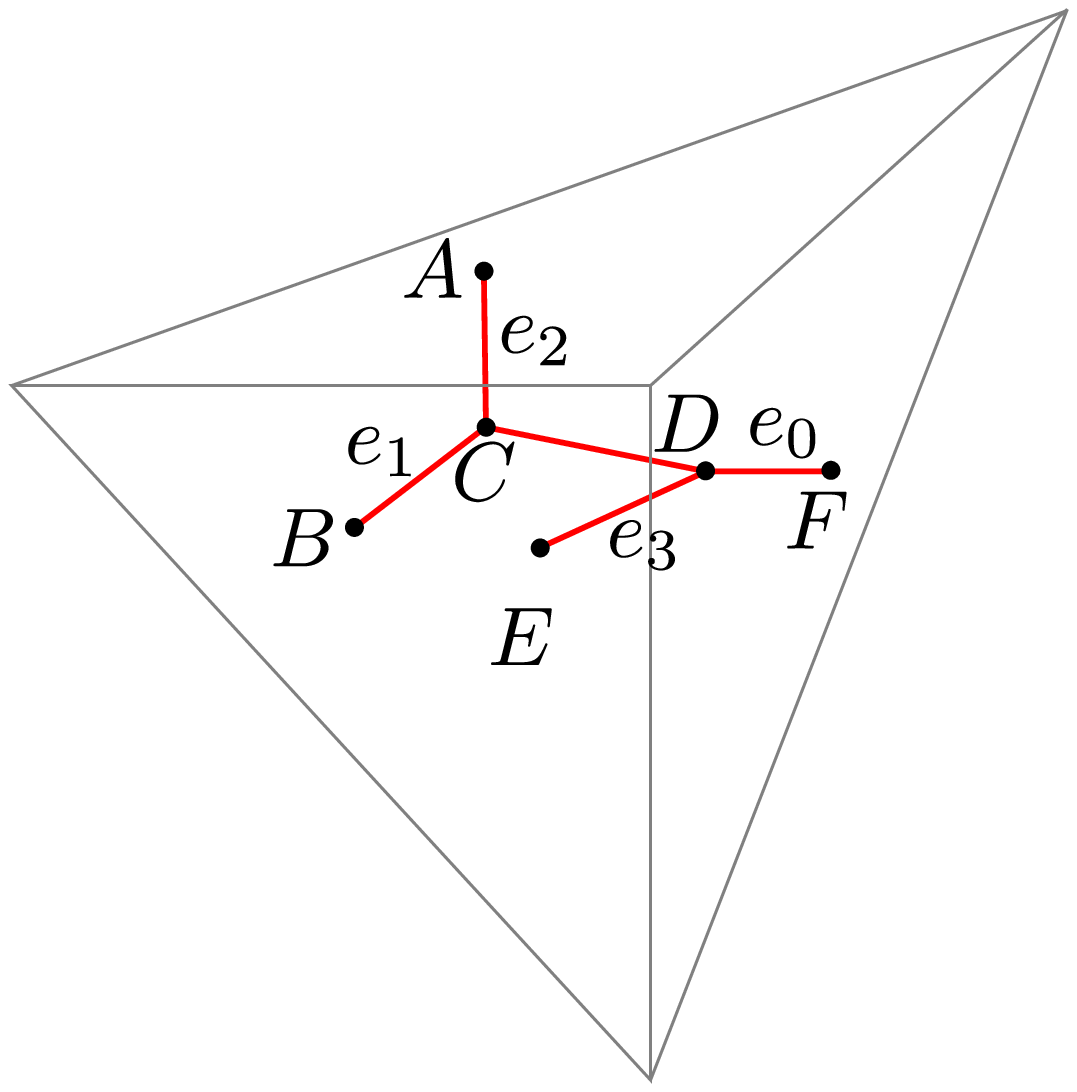} \qquad 
  \end{minipage}
  \quad
  \begin{minipage}[c]{0.15\linewidth}
   \small{    \[ \begin{aligned}
      A & =(0,1,\infty, 0)\\
      B &=(0,\infty, 2, 0)\\
      C &=(0,1,2,0)\\
      E &=(0,0,1,\infty)\\
      F &=(\infty, 0, 1, 0)\\
      D &= (0,0,1,0)\\
      D'& =(\infty, 1, 2, \infty)
    \end{aligned}
    \]}\normalsize
\end{minipage}
  \qquad
  \begin{minipage}[c]{0.3\linewidth}
    \includegraphics[scale=0.33]{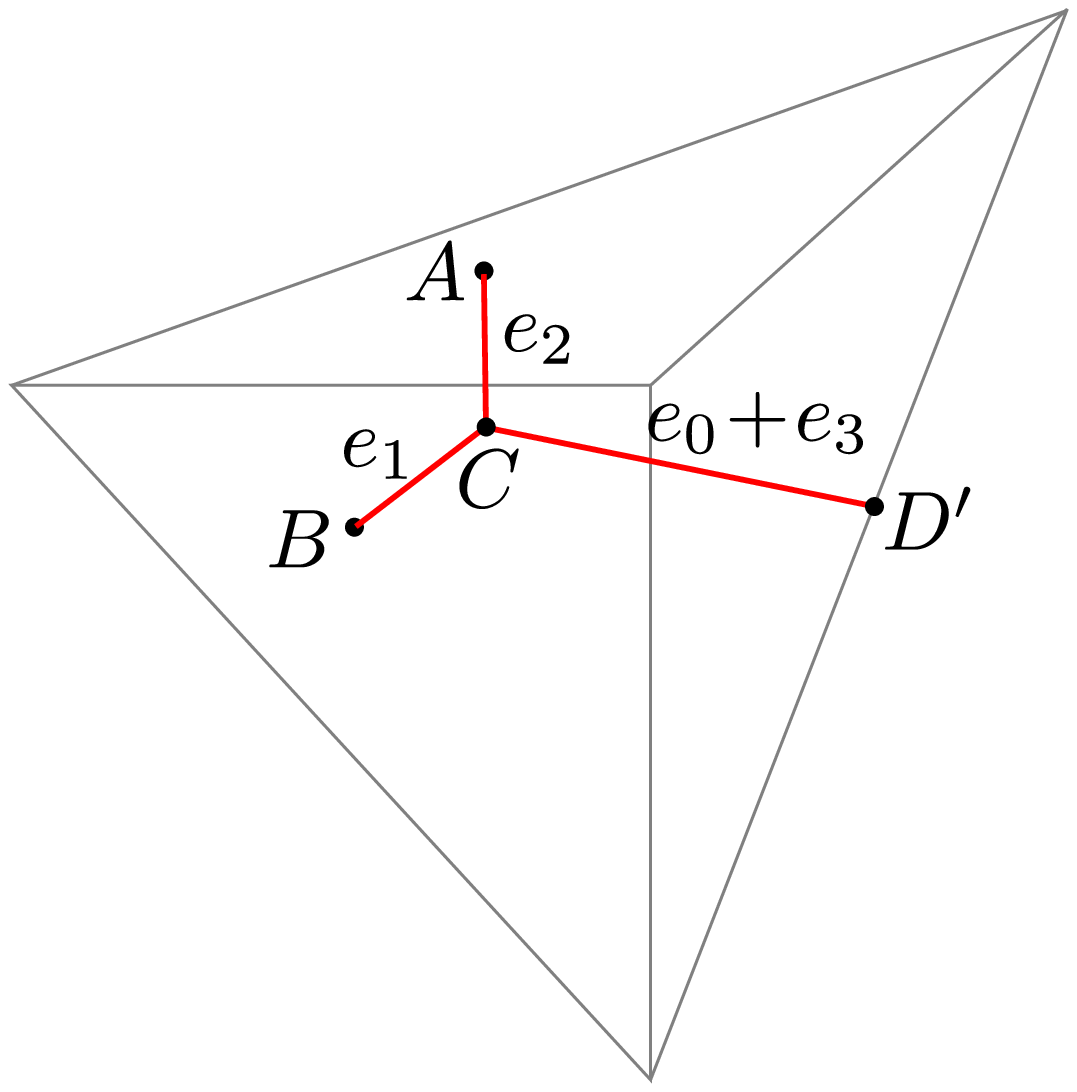} \qquad 
  \end{minipage}
  \caption{A generic tropical line in $\TPr^3$ and a non-generic one associated to the partition $ \{0,3\}\sqcup\{1\}\sqcup \{2\}$ of $\{0,1,2,3\}$. 
  \label{fig:tropLine}  }
\end{figure}

Classically, collinearity of $r\geq 3$ distinct points in $\P^{n-1}$ has a simple characterization: the associated $r\times n$-matrix must have rank two. Equivalently,  all its $3\times 3$ minors vanish. Foundational work on tropical linear algebra~\cite{dev.san.stu:05} yields the analogous statement for deciding tropical collinearity in the tropical projective torus. The determinant of each minor is replaced by its tropical permanent: 

\begin{definition}\label{def:tropPermanent}
  The \emph{tropical permanent} of a matrix $S\in \Rbar^{d\times d}$ is defined by
  \begin{equation}\label{eq:troPerm} \tperm(S) = \trop(det)(S) = \min_{\sigma \in \Sn{d}}\{s_{1\sigma(1)} +\ldots + s_{d\sigma(d) }\},
    \end{equation}
where $\Sn{d}$ denotes the set of permutations of $[d]$. The matrix $S$ (or its permanent) is  \emph{tropically singular} if the minimum in~\eqref{eq:troPerm} is achieved  twice. Otherwise,  we say it is \emph{tropically non-singular}. 
\end{definition}
Our next result extends the above characterization of tropical collinearity  from  the tropical projective torus to the compact setting.

\begin{proposition}\label{pr:tropicalLinesbyMinors} Fix a collection $\Ccurve$ of $r$ distinct points in $\TPr^{n-1}$ with pairwise disjoint $\infty$-entries. 
Then, the collection $\Ccurve$ is tropically collinear if and only if all $3\times 3$-minors of the associated $r\times n$-matrix with entries in $\Rbar$ are tropically singular.
\end{proposition}
\begin{proof} We let $M$ be the tropical $r\times n$-matrix in the statement.
 If $\Ccurve$ lie in the tropical projective torus $\R^n/\rspanone$, the matrix $M$ has entries in $\R$, and the statement follows from~\cite[Corollary 3.8 and Theorem 6.5]{dev.san.stu:05}.
 If we allow some of the points to lie in the boundary of $\TPr^{n-1}$  the argument needs to be slightly modified.

 We fix  $\Ccurve:=\{p_1,\ldots, p_r\}$.   Our hypothesis on the $\infty$-entries of each $p_i$ ensures that each coordinate hyperplane contains at most one point of $\Ccurve$. In particular, any tropical line containing them must meet the interior of $\TPr^{n-1}$.

Fist, we suppose the collection $\Ccurve$ is tropically collinear, and let $\Trop \ell$ be the tropical line through its points. In order to use the collinearity criterion over $\R^n/\rspanone$,  we replace $\Ccurve$  by a new collection $\Ccurve':=\{p_1',\ldots, p_r'\}$ of $r$ points with no $\infty$-coordinates. By construction, every point $p_i$ in the boundary of $\TPr^{n-1}$ will be a leaf of  $\Trop \ell$.  We replace each leaf $p_i$ by a point $p_i'$ in the leg adjacent to $p_i$. If a point  $p_i$ has only  real coefficients  we set $p_i'=p_i$.  The collection  $\Ccurve'$ is contained in $\Trop \ell$ and the corresponding tropical matrix $M'$ has only real entries. Therefore, collinearity in $\R^n/\rspanone$ ensures that all $3\times 3$-minors of $M'$ are tropically singular. As the points in $\Ccurve'$ approach those in $\Ccurve$, continuity  ensures that the corresponding tropical permanents of $M$ are also tropically singular.

  For the converse, assume that all $3\times 3$-minors of $M$ are singular. As before, we will approximate our collection $\Ccurve$  by a collection $\Ccurve'$ in $\R^n/\rspanone$ whose matrix $M'$ has the same property. We use~\autoref{lm:chooseProjection} below to build a coordinate projection $\TPr^{n-1} \to \TPr^{m-1}$ to ensure that each $p_i$ has at most one $\infty$ coordinate. \autoref{lm:recoverLineFromProjection} will allow us to lift any line through $\pi(\Ccurve)$ in $\TPr^{m-1}$ to a line in $\TPr^{n-1}$ through $\Ccurve$. Thus, we may assume $m=n$ and $\pi$ is the identity map.

  To produce a collection of points in $\R^{n}/\rspanone$ from $\Ccurve$, we fix a positive integer $N$ and we let $M'$ be  the matrix  obtained by replacing every $\infty$-entry of $M$ by $N$. We let $\Ccurve'_N:=\{p_i'\colon 1\leq i\leq r\}$ be the collection of  rows of $M'$. By construction, the non-boundary points of $\Ccurve$ are also in $\Ccurve'_N$. Assume they are the last $s$ points of each collection.

  If $N$ is large enough, the distribution of $\infty$ entries in $M$ ensures that every real term of any $3\times 3$-minor of $M$ has the same value as the corresponding minor on $M'$. Thus, all $3\times 3$-minors of $M'$ are tropically singular, so $\Ccurve'_N$ is tropically collinear. We let $\cT \ell_N$ be a tropical line through $\Ccurve'_N$.

  We claim that whenever $N$ is large enough we can pick $\cT \ell_N$ so that it contains $\Ccurve$ as well. Indeed,  given any $\mu\geq 0$,  we let  $\Ccurve_{\mu,N}$  be the configuration of $2r-s$ points in $\TPr^{n-1}$ obtained by adding  to $\Ccurve'_N$ all $r-s$ points ${p_i}{''}:=p_i' + \mu e_i$ for $i=0,\ldots, r-s-1$.   For $N$ large enough, we conclude  that all the $3\times 3$-minors of the $(2r-s)\times n$ matrix associated to $\Ccurve_{\mu,N}$ are tropically singular.  Hence, the expanded configuration will remain collinear. We let $\cT \ell_{N,\mu}$ be any tropical line through  $\Ccurve_{\mu,N}$.

  For $\mu$ large enough, it follows that all points ${p_i}''$  lie on (distinct) legs of $\cT \ell_{N,\mu}$.  Thus, the sequence $\{\cT \ell_{N,\mu}\}_{\mu}$ can be taken to be ultimately constant. Set this limiting value to be the tropical line $\cT \ell$.
  Since $\Ccurve'_{N+\mu}\subset \Ccurve_{N,\mu}$, it follows that we can set $\cT \ell_{N+\mu}$ equal to $\cT \ell$ as well. Since the points ${p_i}''$ converge to $p_i$ in $\Ccurve$ we conclude that $\Ccurve$ lies in $\cT \ell$. This concludes our proof. 
\end{proof}

\begin{lemma}\label{lm:chooseProjection}
  Let $\Ccurve$ be a collection of $r$ points in $\TPr^{n-1}$ with pairwise disjoint $\infty$-entries, and let $M$ be its associated $r\times n$ matrix. Assume that $\Ccurve$ has $k$ points in the boundary of $\TPr^{n-1}$. Then:
  \begin{enumerate}[(i)]
  \item Two columns of $M$ with a common $\infty$ entry represent the same point in $\TPr^{r-1}$.
    \item The coordinate projection $\pi\colon \TPr^{n-1}\dashrightarrow \TPr^{m-1}$ where $m$ is the number of distinct columns of $M$ viewed in $\TPr^{r-1}$ is well-defined and injective on $\Ccurve$.
    \end{enumerate}
  \end{lemma}
    \begin{proof}
After permutation, we may assume that $M$ has the form:
  \[ M=
  \left(  \begin{array}{c|c|c|c|c}
  \begin{array}{ccc}
    \infty & \ldots & \infty
  \end{array} &
  \begin{array}{lcc}\!\ast & \;\ldots & \;\ast
  \end{array} & \ldots &   \begin{array}{lcc}\!\ast & \;\ldots & \;\ast
  \end{array}  &   \begin{array}{ccc}\ast & \ldots & \ast
  \end{array}
\\\hhline{-|-|~~}
\begin{array}{lcc}\!\ast & \;\ldots & \;\ast
  \end{array} &
  \begin{array}{ccc}
    \infty & \ldots & \infty
 \end{array} &   \begin{array}{ccc}\ast & \ldots & \ast
  \end{array}  &   \begin{array}{lcc}\!\ast & \;\ldots & \;\ast
  \end{array} &   \begin{array}{ccc}\ast & \ldots & \ast
  \end{array}\\\hhline{~|-|~~~}
  \begin{array}{lcc}\;\ast & \;\ldots & \ldots
  \end{array} &
  \begin{array}{ccc} & \ldots & 
  \end{array}  &   \begin{array}{ccc} & \ddots & 
  \end{array} &   \begin{array}
    {lcc}\!\ast & \;\ldots & \;\ast
 \end{array} &  \begin{array}{ccc}\ast & \ldots & \ast
  \end{array}\\\hhline{~~~|-|~}
  \begin{array}{lcc}\;\ast & \;\ldots & \ldots 
  \end{array} &
  \begin{array}{ccc} & \ldots & 
  \end{array}  &   \begin{array}{ccc}\ast & \ldots & \ast
  \end{array} &   \begin{array}{ccc}
    \infty & \ldots & \infty
 \end{array} &  \begin{array}{ccc}\ast & \ldots & \ast
  \end{array}\\\hhline{---|-|-}
   \begin{array}{lcc}\!\ast & \;\ldots & \;\ast
  \end{array} &
 \begin{array}{lcc}\!\ast & \;\ldots & \;\ast
      \end{array}
  &   \begin{array}{lcc}\ast & \ldots & \ast
      \end{array} &
    \begin{array}{lcc}\!\ast & \;\ldots & \;\ast
    \end{array} &\begin{array}{ccc}\ast & \ldots & \ast
    \end{array}
    \\
    \begin{array}{ccc}\vdots\quad & 
       & \quad \vdots
    \end{array} &
        \begin{array}{ccc}\vdots\quad & 
       & \quad \vdots
        \end{array} &
            \begin{array}{ccc}\vdots & 
       &\quad \;\;\vdots
            \end{array} &
                \begin{array}{ccc}\vdots\quad & 
       & \quad \vdots
                \end{array} &                
                \begin{array}{ccc}\vdots & 
       & \quad \;\;\vdots
                \end{array}  \\
   \begin{array}{lcc}\!\ast & \;\ldots & \;\ast
  \end{array} &
 \begin{array}{lcc}\!\ast & \;\ldots & \;\ast
      \end{array}
  &   \begin{array}{lcc}\ast & \ldots & \ast
      \end{array} &
    \begin{array}{lcc}\!\ast & \;\ldots & \;\ast
    \end{array} &\begin{array}{ccc}\ast & \ldots & \ast
    \end{array}
  \end{array}
  \right),
 \]   
  where the $\ast$ indicate the corresponding entry has a real value.   For each $i=0,\ldots, r-1$, we let $B_i$ be the columns corresponding to $\infty$-entries on the $i$-th row of $M$. In particular, if the  bottom right block matrix has size $s\times k$, with  $s,k\geq 0$, it follows that $B_{r-k} = \ldots = B_{r-1} = \emptyset$.
  Furthermore,  since every row lies in  $\TPr^{n-1}$, we can find an $s_i$ in $\{0,\ldots, r-k-1\}$ with $M_{is_i}\in \R$.

  First, assume $|B_i|\geq 2$. Working with the $3\times 3$-tropical permanents involving 2 columns of $B_i$ and the column $s_i$, we conclude that  the $(r-1)\times |B_i|$-submatrix of $M$ with rows in $[n]\smallsetminus \{i\}$ and columns in $B_i$ has tropical rank 2, that is, all its $2\times 2$-minors are tropically singular. In particular, the difference of any two columns in this submatrix is a multiple of the all-ones vector so all columns in $B_i$ represent the same point in $\TPr^{r-1}$. This proves the first statement.

  The coordinate projection $\pi$ is determined by a choice of $m$ columns from $M$. We pick the first column of each non-empty $B_i$ and take the remaining columns in $[n]\smallsetminus \bigsqcup_{i=0}^{r-k-1} B_i$. By construction, the columns in the latter set lie in $\R^r/\rspanone$, so  $\pi$ is well defined on $\Ccurve$. Injectivity follows from (i). 
    \end{proof}

Tropical linear spaces correspond to valuated matroids~\cite{spe:08}. In this language, the set of leaves of a tropical line corresponds precisely to the \emph{cocircuits} of the underlying rank-two (loopless) valuated matroid on $n$-elements. In turn, the sets $B_1,\ldots, B_m$ encode the parallel elements of this matroid. This interpretation and the proof of~\autoref{pr:tropicalLinesbyMinors} both have the following consequence:

\begin{corollary}\label{cor:LineFromBoundary}
  Any tropical line in $\TPr^{n-1}$ is uniquely determined by its set of leaves.
  \end{corollary}

Let $\cT\ell$ be a tropical line in $\TPr^{n-1}$ with $m$ leaves associated to the partition $\bigsqcup_{i=1}^m B_i$ of $[n]$. Consider any subset $J$ of $[n]$ of size $m$ containing exactly one element of each set $B_i$. We define the canonical projection:
\begin{equation}\label{eq:projLJ}
 \pi_{J}\colon \TPr^{n-1}\dashrightarrow \TPr^{|J|-1}.
\end{equation}

\begin{lemma}\label{lm:recoverLineFromProjection}  The projection  $\pi_{J}$ from~\eqref{eq:projLJ}  is well defined on $\cT\ell$ and its image is a generic tropical line in $\TPr^{m-1}$. Furthermore, $\cT\ell$ can be uniquely recovered from its image together with the data of $J$ and all its leaves.
    \end{lemma}
\begin{proof} By~\autoref{lm:chooseProjection}, the projection $\pi_{J}$ is well-defined on the leaves of $\cT\ell$. Since all remaining points of $\cT \ell$ have all real coordinates, it is well-defined on the whole tropical line. By construction, the set $\pi_J(\cT \ell)$ is a balanced metric tree with $m$ leaves. All its legs have directions $e_0,\ldots, e_{m-1}$, so the result is the tropical generic line in $\TPr^{m-1}$.

  To reconstruct $\cT\ell$ from its leaves and its image under $\pi_J$ we must determine the missing $n-m$ coordinates of each point in the target line. We let $p_0,\ldots,p_{m-1}$ be the leaves of $\ell$, and we let $B_i$ be the set of $\infty$-coordinates of $p_i$. After applying a permutation in $\Sn{n}$, we may assume all $B_i$'s are intervals.
  Given $i=0, \ldots, m-1$ we let $b_i$ be the single element of $B_i\cap J$. For each  $j \in B_i$ we set
  \[
  \lambda_j := (p_s)_{j}-(p_s)_{b_i} \text{ for any} s\neq i.
  \]
  By~\autoref{lm:chooseProjection}, the left-hand expression is independent of our choice of $s$. It follows that any point $q\in \pi_J(\ell)\subset \TP^{m-1}$ can be lifted uniquely to a point $\overline{q}$ in $\ell$ setting  $\overline{q}_j := \lambda_j + q_i$ for each $j\in B_i$.  
    \end{proof}

The proof of~\autoref{pr:tropicalLinesbyMinors} yields an algorithm to reconstruct any tropical line from its set of leaves. This will play  a central role in~\autoref{sec:comb-types-tree}. By~\autoref{lm:recoverLineFromProjection}, it suffices to only consider the case  when the tropical line is generic and meets the interior of $\TPr^{n-1}$. This is the content of~\autoref{alg:treesFromLeaves}.

\begin{remark}\label{rm:NotationEdges} Since all primitive edge directions of metric trees in $\TPr^{n-1}$ are of the form $-e_{I}:=-\sum_{i \in I} e_i$, we write $(v,I)$ to indicate a vertex or leaf $v$ of a given tree  and a directed edge with direction $-e_I$ adjacent $v$.  
\end{remark}

\begin{algorithm}[htb]   \KwIn{A list $\Ccurve:=\{p_0,\ldots, p_n\}$ of tropically collinear points in $\TPr^{n}$.}
  \textbf{Assumption:} Each $p_i$ has exactly one $\infty$ coordinate, namely $(p_i)_i = \infty$ for all $i$.\\
         \KwOut{The vertices and edges of the unique tropical line in $\TPr^{n}$ through $\Ccurve$.}
         \texttt{Vertices} $\leftarrow \{(p_i,\{i\})\}$ ; 
         \texttt{Edges} $\leftarrow \emptyset$ ;\ifhmode\\\fi
         \texttt{StopTest} $\leftarrow \{((v,I), (v',I'))\in  \texttt{Vertices}^2 \colon \, I\cup I' = \{1,\ldots, n\}\}$;\ifhmode\\\fi 
         \While{$\emph{\texttt{StopTest}} = \emptyset$}
               {\texttt{PairsVertices} $\leftarrow \{(v,I), (v', I') \in \texttt{Vertices}^2 \colon v\neq v' \text{ in }\TP^n 
                 \}$\\
                 \While{$\emph{\texttt{PairsVertices}} \neq \emptyset$}
                       {$((v,I),(v',I')) \leftarrow \text{last element from }\texttt{PairsVertices}$\;  \texttt{PairsVertices} $\leftarrow \texttt{PairsVertices} \smallsetminus \{((v,I),(v',I'))\}$ \;
                         \If{the system~\eqref{eq:vertexadjacency} admits a solution $(\lambda, \lambda', \mu)$ giving a vertex $v''$}{
                           $(v'', I'') \leftarrow (v'', I\cup I')$\;
                           \texttt{Vertices} $\leftarrow \texttt{Vertices}  \cup \{(v'', I'')\}$;     $\texttt{Edges} \leftarrow$ $\texttt{Edges}\cup \{(v, v''), (v', v'')\}$\;
                         }}
                   Merge pairs $(v,I)$ in \texttt{Vertices} giving the same point in $\TPr^n$ and adjust \texttt{Edges} as follows:\ifhmode\\\fi
                                          1. $\texttt{AllVertices} \leftarrow \{v \colon \exists I \text{ with } (v,I) \in \texttt{Vertices}\}$ viewed in $\TPr^n$\;
                                            2. \For{$v \in \emph{\texttt{AllVertices}}$}{
                             $\texttt{AllDirections} \leftarrow  \{I' \colon I' \subset [n], \text{ and } \exists  (v',I') \in \texttt{Vertices} \text{ with } v' = v \text{ in }\TPr^n\}$\;
                                              $I_v \leftarrow \bigcup_{I' \in \texttt{AllDirections}} I'$                             \;                 }
                                            3. $\texttt{Vertices}\leftarrow \{(v,I_v) \colon v \in \texttt{AllVertices}\}$\;
                 4. $\texttt{Edges} \leftarrow \{(v,v') \colon v, v'\in \texttt{Vertices},\  \exists (\ww, \ww') \in \texttt{Edges} \text{ with } \ww = v \text{ and } \ww' = v' \in \TPr^n\}$\;
                 $\texttt{StopTest} \leftarrow
\{((v,I), (v',I'))\in  \texttt{Vertices}^2 \colon \, I\cup I' = \{1,\ldots, n\}\}$
}
               $\texttt{Edges} \leftarrow \texttt{Edges} \cup \{(v,v') \colon ((v,I), (v',I')) \in \texttt{StopTest}\}$;\ifhmode\\\fi
                             Merge pairs in \texttt{Vertices} giving the same points in $\TPr^n$ and adjust \texttt{Edges} (as in 1--4 above);\ifhmode\\\fi
               \KwRet{\emph{(\texttt{Vertices}, \texttt{Edges})}.}

  \caption{Reconstructing a generic tropical line in $\TPr^{n}$ from its set of $n+1$ leaves.\label{alg:treesFromLeaves}}
\end{algorithm}

The next  technical lemma is be at the hearth of the construction. It determines when a given pair of vertices of a tropical line are connected through a third vertex with prescribed directions for the pair of  adjacent edges or legs.

\begin{lemma}\label{lm:connectVertices} Let $(v,I)$ and $(v',I')$ be a pair of vertices of a generic tropical line in $\TPr^{n}$ together with an edge (or leg) adjacent to each. Then, we can find a vertex $v''$ adjacent to both $v$ and $v'$ via the prescribed edges if and only if the following system 
  \begin{equation}\label{eq:vertexadjacency}
     v_k - \lambda (e_I)_k =       v'_k - \lambda (e_{I'})_k  + \mu \qquad \text{ for all } k \in \{j\in \{1,\ldots,n\}  \text{ with } v_j,v'_j \neq \infty\}
   \end{equation}
has a solution $(\lambda, \lambda',\mu)$ with $\lambda,\lambda' \geq 0$ and  $\mu \in \R$.
  Furthermore, we can recover the vertex $v''$ in $\TPr^{n}$  as $v''_k = v_k -\lambda\;e_I $ if $v_k\neq \infty$, whereas $v''_k = v'_k -\lambda'\,e_{I'} + \mu \,\mathbf{1}$ if $v_k =  \infty$. 
  \end{lemma}
\begin{proof} The genericity assumption ensures that $v$ and $v'$ have disjoint sets of $\infty$-coordinates. The result follows by a simple linear algebra computation. 
  \end{proof}

\begin{proof}[Proof of~\autoref{alg:treesFromLeaves}] The algorithm constructs a generic tropical line in $\TPr^n$ by generating new vertices and edges from old ones. The production starts from the leaves and produces non-leaf vertices in a level structured fashion, moving towards the center of the tree. The set $\texttt{Vertices}$ will collect all pairs $(v,I)$ where $v \in \TPr^n$ is a leaf or vertex of the tree and $I$ indicates the  inward direction $-e_I$, i.e.\ the unique edge adjacent to $v$ pointing towards the center of the tree (see~\autoref{rm:NotationEdges}.)
  The set \texttt{Edges} records all pairs of adjacent vertices.
  Our stopping criterion is given by the variable \texttt{StopTest} that checks if two vertices have complementary inward directions.
  
  Each iteration will produce new vertices from pairs $(v,I)$, $(v',I')$ of  oldones with prescribed inward directions by analyzing the solvability of the corresponding system~\eqref{eq:vertexadjacency}. After running through all such pairs, and producing new vertices, we must ``merge'' distinct elements in $\texttt{Vertices}$ to record only the  vertex representatives in $\TPr^n$ together with their true inward directions.

  We do so as follows. Assume  $(v, I)$ is a vertex produced at a given iteration, but we had already constructed $v$ from a different pair of vertices. In this case, we must modify the set $I$. To this end, we collect all sets $I'$ where $(v',I') \in \texttt{Vertices}$ and  $v' = v \in \TPr^n$ in the variable \texttt{AllDirections}. By the balancing condition, the new inward direction emanating from $v$ will be encoded by the set $I_v$ obtained as the union of all elements in \texttt{AllDirections}.
  Once the set \texttt{Vertices} is adjusted, we produce the set \texttt{Edges} by replacing an unordered pair $(\ww,\ww')$ in the old collection by the pair of the representatives. 

  Once the stopping criterion \texttt{StopTest} is reached, we conclude that all vertices of the generic tropical line in $\TPr^n$ with leaves $\Ccurve$ have been covered: we have a path from each leaf to one of the two vertices in each pair $((v,I),(v',I'))$ in \texttt{StopTest}. It follows by construction that $(v,v')$ is an edge of the tree, so we must add all such pairs to our set \texttt{Edges}, if they were not recorded earlier.

  Finally, we perform a merging step of all vertices, together with the corresponding adjustment of the set of edges, and then output the pair $(\texttt{Vertices}, \texttt{Edges})$.
\end{proof}

\autoref{alg:treesFromLeaves} outputs each edge of a generic tropical line as its pair of adjacent vertices. Since each vertex comes with the information of its adjacent edge's direction pointing towards the center of the line, we can easily determine all edge lengths from this data, as we now explain:

\begin{lemma}\label{lm:treemetric}
  Consider a pair $(v,I)$ and $(v',I')$ of adjacent vertices of a generic tropical line in $\TP^n$. Then, either $I\sqcup I' = [n]$, $I\subset I'$ or $I'\subset I$. In the first two cases, the lattice length of the edge joining $v$ and $v'$ equals $|\lambda|$, where $(\lambda, \mu)$ is the unique solution to the system
  \begin{equation}\label{eq:systemMetric} v' - v = \lambda\,\sum_{i\in I} e_i  + \mu \,\mathbf{1}\quad  \text{ in }\R^{n+1}.
  \end{equation}
  In the third situation, the length is obtained by exchanging the roles of $(v,I)$ and $(v',I')$.
\end{lemma}
\begin{proof} Since $v$ and $v'$ are adjacent, we know that the direction of the edge joining them is one of the two inward directions. \autoref{alg:treesFromLeaves} provides two possible scenarios: either $I$ and $I'$ are disjoint (since the stopping criterion  was reached), or one of the vertices was obtained before the end of the \texttt{While} cycle. In the latter case, either $I\subset I'$ or viceversa. In both situations, up to a scalar multiple of $\mathbf{1}$, the primitive direction of the edge joining $v$ and $v'$ is one of the two inward directions: the smaller among $I$ and $I'$, or any of them if they are disjoint. The equation~\eqref{eq:systemMetric} follows.
  \end{proof}
\section{Combinatorics of extra tropical lines on tropical cubic del Pezzo surfaces}\label{sec:comb-extra-trop}

In this section, we turn our attention to the central topic of this paper, namely, the number  of tropical lines on anticanonically embedded tropical cubic surfaces.  Our main technique to rule out extra lines on tropical cubic surfaces exploits the rigid structure of the boundary of the tropical surface. In this section, we take the first step and build candidate boundary points in any  potential tropical lines meeting the interior of the tropical surface.

As in the previous sections, we let $X$ be a smooth cubic del Pezzo surface without Eckardt points, embedded in $\P^{44}$ via the anticanonical map, and we consider its induced tropicalization  $\Trop X \subset \TPr^{44}$. Our discussion in~\textcolor{blue}{Sections}~\ref{sec:27-lines-universal} and \ref{sec:bound-antic-trop} reveals a key property of this embedding: the boundary of the surface $X$ determined by  the coordinate hyperplanes is supported on the 27 lines on $X$. Since the boundary of $\Trop X$ is the tropicalization of the boundary of $X$, we conclude that any tropical line in the boundary of $\Trop X$  must be supported on the arrangement of trees determined by the tropicalizations of the 27 classical lines.

  Recall that the 10-regular Schl\"afli graph on 27 vertices  is the  intersection complex of the boundary divisor of $X\subset \P^{44}$ in the absence of Eckardt points.   The following is the main result of this section. It shows that any potential tropical line meeting the interior of $\Trop X$ has exactly five boundary points. Furthermore, they correspond to the intersection points of each of the five pairs of boundary tropical lines meeting a common exceptional curve. We refer to them as \emph{nodal points}.

\begin{theorem}\label{thm:candidateTropLines} Let $X$ be a smooth cubic del Pezzo surface without Eckardt points. Consider its tropicalization with respect to the anticanonical embedding. Then, $\Trop X$ contains at most 27 extra tropical lines meeting its interior. Each such line is indexed by a given exceptional curve in $X$, and has precisely five boundary points arising from the tropicalization of the five nodes associated to the link of the indexing curve in the Schl\"afli graph.
\end{theorem}

\begin{remark}\label{rm:Conditionsi-ii} Since every tropical line in $\TPr^{44}$ is realizable, we will always write any potential tropical line as $\Trop \ell$  for some line  $\ell$  in $\P^{44}$. Since any extra line on $\Trop X$ meets its interior,  our earlier discussion ensures that $\ell$ meets the dense torus.
  The tropical line $\Trop \ell$ satisfies two key properties:
  \begin{enumerate} [(i)]
\item $\Trop \ell$ meets all 45 boundary hyperplanes in $\TPr^{44}$, each one of them indexed by an anticanonical triangle. Each such intersection consists of one point.
  \item $\Trop \ell$ intersects each boundary tropical line in at most one point.
  \end{enumerate}
  Since $\Trop \ell$ is a tree with at most 45 leaves and its legs have directions with disjoint support, there is at most one point in the intersection between $\Trop \ell$ and any given boundary hyperplane. Equality holds by construction.
\end{remark}

\begin{proof}[Proof of~\autoref{thm:candidateTropLines}] Following~\autoref{rm:Conditionsi-ii}, we write our potential extra tropical line as $\Trop \ell$ for some  line $\ell$ meeting the dense torus of $\P^{44}$. 
    Our first key observation is that the combinatorics of the boundary of $\Trop X$, make it impossible to simultaneously satisfy conditions (i) and (ii) in~\autoref{rm:Conditionsi-ii}. A careful case by case analysis, which is described by~\textcolor{blue}{Lemmas}~\ref{lm:no3}, \ref{lm:no2} and~\ref{lm:no1}, shows that $\Trop \ell$ meets every boundary tropical line of $\Trop X$ in at most one point, which must be nodal.

    By~\autoref{cor:BoundaryPoints} we know that each node lies in exactly nine boundary hyperplanes. Since  all boundary points of $\Trop \ell$ must have disjoint sets of $\infty$-entries by (i),  we conclude that $\Trop \ell$ has at most five boundary points. By the action of $\Wgp$ we may assume that one of them equals $\Trop \E{1}\cap \Trop \F{12}$. \autoref{lm:lineWith5BoundaryPoints} implies that the remaining four  boundary points of $\Trop \ell$  are $\Trop \E{i}\cap \Trop \F{2i}$ for  $i\neq 1,2$. Notice that these five points come from the pairwise intersections of the ten classical lines meeting $\Gc{2}$, so  we use $\Gc{2}$ to index $\Trop \ell$. By~\autoref{cor:LineFromBoundary}, $\Trop \ell$ is uniquely determined by its boundary points.   There are 27 such labels and they are all $\Wgp$-conjugate. This concludes our proof. 
\end{proof}

\begin{lemma}\label{lm:only1} If $\Trop \ell$ meets $\Trop E$ for some exceptional curve $E\in \sed$ at a non-nodal point, then $\Trop \ell$ avoids all ten lines $\Trop E'$ with $E'\in \sed$ and $E'\cap E \neq \emptyset$.
\end{lemma}
\begin{proof} By the action of the Weyl group $\Wgp$ we may assume that $E=\E{1}$. Then $E' = \F{1j}$ or $\Gc{j}$ for $j=2,\ldots, 6$. Condition (i) from~\autoref{rm:Conditionsi-ii} applied to the boundary hyperplane indexed by the symbol $x_{1j}$ in $\ted$ ensures that $\Trop \ell$ avoids $\Trop E'$. 
\end{proof}

  \begin{lemma}\label{lm:no3} $\Trop \ell$ cannot meet three tropical boundary lines at non-nodal points.    
  \end{lemma}
  \begin{proof} We argue by contradiction. \autoref{lm:only1} ensures that the tree lines in the statement come from disjoint exceptional curves.  Without loss of generality, by the action of the Weyl group $\Wgp$ we may assume they are $\Trop \E{1}$, $\Trop \E{2}$ and $\Trop \E{3}$.
    Again, \autoref{lm:only1} implies that $\Trop \ell$ avoids 
the three lines $\Trop \F{14}$, $\Trop \F{25}$ and $\Trop \F{36}$. By~\autoref{lm:boundaryTX} we conclude that $\Trop \ell$ does not intersect the boundary hyperplane indexed by  the symbol $y_{123456}$ in $\ted$, in violation of  condition (i) from~\autoref{rm:Conditionsi-ii}.
  \end{proof}
  
  \begin{lemma}\label{lm:no2}  $\Trop \ell$ cannot meet two boundary tropical lines at non-nodal points.
  \end{lemma}
  \begin{proof} Again, we argue by contradiction. By~\autoref{lm:only1}, we know that the two boundary lines in the statement cannot intersect. Without loss of generality,  we may assume they are $\Trop \E{1}$ and $\Trop \E{2}$. Again,~\autoref{lm:only1} ensures that $\Trop \ell$ avoids all lines $\Trop \Gc{k}$, $\Trop \F{1j}$ and $\Trop \F{2l}$ for $k=1,\ldots, 6$, $j=2,\ldots, 6$ and $l=3,4,5,6$.
However, since $\Trop \ell$ intersects all the hyperplanes indexed by the symbols $y_{1k2l**}$ in $\ted$, we conclude that $\Trop \ell$ meets all tropical lines $\F{ij}$ with $i,j\notin \{1,2\}$. In particular, $\Trop \ell$ avoids $\Trop \F{12}$ but meets both $\F{34}$ and $\F{56}$. ~\autoref{lm:only1} forces this intersection to be the node  $\Trop \F{34}\cap \Trop \F{56}$.

Since $\Trop \ell$ avoids $\Trop \F{13}$ and $\Trop \Gc{1}$ and meets the anticanonical triangle associated to $x_{31}$, it follows from~\autoref{lm:boundaryTX} that $\Trop \ell$ intersects $\Trop\E{3}$. Since $\Trop \F{34}$ also meets $\Trop \ell$, it follows that $\Trop \F{34}\cap \Trop \ell = \Trop \F{34}\cap \Trop \E{3}$. This contradicts condition (i) from~\autoref{rm:Conditionsi-ii} because
 $\Trop \F{34}\cap \Trop    \F{56}$ and $\Trop \E{3}\cap \Trop \F{34}$ are distinct points
    in $\Trop \F{34}$ by~\autoref{lm:IntersectionBoundaryLines}.
    \end{proof}

  \begin{lemma}\label{lm:no1}  $\Trop \ell$ cannot meet a boundary tropical line at a non-nodal point.
  \end{lemma}
  \begin{proof} As before, we argue by contradiction. Without loss of generality, we assume the line in the statement is $\Trop \E{1}$. \autoref{lm:only1} implies that
    \begin{equation}\label{eq:noF1i}
      \Trop \ell\cap \Trop \F{1i} = \emptyset  \qquad \text{ and } \qquad \Trop \ell \cap \Trop \Gc{i} =\emptyset \quad \text{ for all }i\neq 1.
    \end{equation}

    \begin{claim}\label{cl:1}
$\Trop \ell$  intersects $\Trop \E{i}$ for all $i=2,\ldots, 6$.
    \end{claim}
    On the contrary, assume $\Trop \ell$ avoids one of these lines, say $\Trop \E{6}$.     By considering the symbol $x_{61}$ in $\ted$,  condition (ii) and~\eqref{eq:noF1i} ensure  that $\Trop \ell\cap \Trop \Gc{1}\neq \emptyset$. Since $\Trop \ell$ meets $\Trop \E{1}$ at a non-nodal point by hypothesis, \autoref{lm:no2} forces $\Trop \ell \cap \Trop \Gc{1}$ to be a nodal point of $\Trop \Gc{1}$. Again,~\eqref{eq:noF1i} and the combinatorics of the anticanonical triangles restrict the nature of this node: it must lie in  $\Trop \E{j}$ for some $j=2,\ldots, 5$, which we may take as $\Trop \E{2}$.

    Since $\Trop \ell\cap \Trop \Gc{1}\cap \Trop \E{2}\neq \emptyset$, conditions (i) and (ii) from~\autoref{rm:Conditionsi-ii} and the absence of Eckardt points implies that $\Trop \ell$ avoids  $\Trop \F{2i}$ for all $i\neq 2$. By considering the hyperplane indexed by the symbol $x_{31}$, condition (ii) and~\eqref{eq:noF1i} ensure that $\Trop \ell$ does not meet $\Trop \E{3}$.

Finally, since $\Trop \ell\cap \Trop \E{3} = \Trop \ell\cap \Trop \Gc{2} = \Trop \ell\cap \Trop \F{23}=\emptyset$, we conclude that $\Trop \ell$ does not meet the hyperplane indexed by $x_{32}$, violating  condition (i). The claim follows from here.
\begin{claim}\label{cl:2} $\Trop \ell\cap \Trop \F{ij}= \emptyset$ for all $1<i<j\leq 6$.
\end{claim}
\noindent
The result follows from~\autoref{cl:1}. Indeed, if any of these intersections were non-empty,~\autoref{lm:only1} would force it to be both the nodes $\Trop \F{ij}\cap \Trop \E{i}$ and $\Trop \F{ij}\cap \Trop \E{j}$. This cannot happen by~\autoref{lm:boundaryTX}.

Finally,~\autoref{cl:2} implies that $\Trop \ell$ avoids the hyperplane indexed by $y_{123456}$, violating condition (i). This concludes our proof.
  \end{proof}

  The previous lemmas ensure that the boundary points of any extra tropical line in $\Trop X$ meeting the interior of $\TPr^{44}$ lie in the set of 135 boundary nodal points. Our next result implies that knowing one of them, determines the rest:
  \begin{lemma}\label{lm:lineWith5BoundaryPoints} 
    Let $X$ be a smooth cubic del Pezzo without Eckardt points and let $\Trop \ell$ be a tropical line in $\Trop X$ meeting the interior of $\TPr^{44}$. Assume  $\Trop \ell$ intersects the line $\Trop \E{1}$ at some node $\Trop \E{1}\cap \Trop \F{1k}$ with $k\neq 1$. Then, $\Trop \ell$ has exactly five points in its boundary, namely $\Trop \E{i}\cap \Trop \F{ij}$ with $i\neq j$.
  \end{lemma}
  \begin{proof} By the action of $\Sn{6}$, we may assume that $k=2$. Since $\Trop \ell$ intersects $\Trop \F{12}$ at $\Trop \E{1}\cap \Trop \F{12}$, conditions (i) and  (ii) from~\autoref{rm:Conditionsi-ii} ensure that $\Trop \ell$ avoids $\Trop \E{2}$, $\Trop \Gc{1}$, $\Trop \Gc{2}$ and all $\Trop \F{mn}$ with $m,n\notin \{1,2\}$. Likewise, since $\Trop \ell$ intersects $\Trop \E{1}$ at the same node $\Trop \E{1}\cap \Trop \F{12}$, we conclude that $\Trop \ell$ misses $\Trop \F{1j}$ and $\Trop \Gc{j}$ for all $j=3,4,5, 6$.

    Next, we analyze the intersection of $\Trop \ell$ with the eight remaining  boundary tropical lines, namely
    $\Trop \E{j}$ and $\Trop \F{2j}$ for $j=3, \ldots, 6$. By~\autoref{lm:no1}, any such intersection point must be nodal, and these eight lines come in four intersecting pairs. Thus, $\Trop \ell$ itself has at most 5 boundary points.

    Since $\Trop \ell$ avoids both $\Trop \E{2}$ and $\Trop \Gc{j}$ for all $j=3,\ldots, 6$, but meets the boundary hyperplane indexed by $x_{2j}$, we conclude that $\Trop \ell$ meets all $\Trop \F{2j}$. Similarly, analyzing the intersection between $\Trop \ell$ and the triangles $x_{ij}$ for a fix $i\neq 1,2$ and $j\notin \{1,2,i\}$, it follows that $\Trop \ell$ meets $\Trop \E{i}$.
    
    Finally, since $\Trop \ell$ meets both $\Trop \E{i}$ and $\Trop \F{2i}$ for $i\neq 2$, ~\autoref{lm:only1} ensures that these intersections are the node  $\Trop \E{i}\cap \Trop \F{2i}$.  Thus, the boundary points of $\Trop \ell$  have the desired description.
  \end{proof}

  \autoref{thm:candidateTropLines} imposes severe restrictions on any  potential tropical line on $\Trop X$ meeting the interior of $\TPr^{44}$ by characterizing its  boundary points. Each such combination arises  as the tropicalization of the five points in $\P^{44}$ associated to the link of a vertex in the Schl\"afli graph.  The anticanonical coordinates of each of these points lies in the function field assocaited to the 40 Yoshida functions from~\autoref{sec:yosh-cross-funct}. The next result realizes them as Laurent monomials in both Yoshidas and Cross functions. The proof provides an algorithm for determining these expressions. Implementations in \python\ and \sage\ are available in the Supplementary material.
  
  \begin{lemma}\label{lm:CoordinatesOfNodes}
    Let $X_L \to \spec L$ be the universal a smooth cubic del Pezzo surface anticanonically embedded in $\P^{44}$. Then, each of the 135 nodes of $X_L$ obtained as pairwise intersections between its 27 exceptional curves has exactly nine zero coordinates. The remaning ones are Laurent monomials in the Yoshida and Cross functions. Furthermore, each nonzero coordinate has at least one Cross function factor.
  \end{lemma}

  \begin{proof}  By the action of $\Wgp$, it is enough to show the validity of the statement for the five points associated to the link of $\Gc{2}$, namely $\E{1}\cap \F{12}$ and $\E{i}\cap \F{2i}$ for $i\neq 1, 2$.  By~\autoref{cor:BoundaryPoints}, these points are characterized by the vanishing of precisely nine anticanonical coordinates in the linear span of $X_L$ cut out by  the 270 linear equations~\eqref{eq:YoshidaP3}.

    Using \sage, we compute  a basis $\{v_0,\ldots, v_3\}$ of the four-dimensional solution set to this linear system over the function field associated to the 40 Yoshida functions. We encode it as a $4\times 45$ matrix $M$ all of whose entries are Laurent monomials in Yoshida and Cross functions. Since these functions are algebraically depedent, we choose to work instead with an equivalente matrix $M'$, whose entries are the  rational functions in the parameters $d_1,\ldots, d_6$ expressing the corresponding entries in $M$. The new matrix is obtained from the  $\Wgp$-orbits of~\eqref{eq:yoshida1inD} and~\eqref{eq:Cross}.

    We obtain the precise coordinates for the classical node associated to a given  boundary point of $\Trop \ell$ by finding a generator of the one-dimensional left-kernel of the corresponding $4\times 9$-submatrix of $M'$: this generator provides the scalars in the linear combination of the vectors $v_i$ giving the node. For example, the point $\E{1}\cap \F{12}$ coincides with the basis element $v_3$, but in general, the scalars  will be  rational functions in parameters $d_1, \ldots, d_6$. A \python~script allows us to re-express these coordinates as rational functions in the Yoshidas functions.
  
  By factorizing the 36 nonzero coordinates of $\E{1}\cap \F{12}$, we certify the claim in the statement (see the Supplementary material.) The factors are monomials and binomials in the Yoshidas. Furthermore, these binomial expressions yield Cross functions. By acting via the transpositions $(1\, i)$ in $\Sn{6}$ we see that the remaining four nodes associated to the link of $\Gc{2}$ the Schl\"afli graph have the desired monomial expressions. This concludes our proof.
  \end{proof}

        \begin{remark}\label{rm:cubicNotUsed}
          The parameterization of the 135 nodes described in~\autoref{lm:CoordinatesOfNodes} in terms of Yoshida and Cross functions is solely obtained from the coordinates of the $\P^3$ linearly spanned by $X\to\spec \KK$, and  makes no use of the binomial cubic equation cutting out $X_L$ in $\P_L^3$.     Explicit coordinates for the 27 quintuples of nodes are available in the Supplementary material.
   \end{remark}

   \begin{remark}\label{rm:genericImpliesTropical}
     Let $X_{\widehat R} \to \spec \widehat R$ be the universal family of marked smooth cubic del Pezzo surfaces (see \autoref{rm:Saturation}) and $X_{\widehat R} \subset \P^{44}_{\widehat R}$ the universal anticanonical embedding defined by the 45 sections $E_iF_{ij}G_j$.
     Each of the 135 nodes of the anticanonical triangles gives a section $\spec {\widehat R} \to X_{\widehat R} \subset \P^{44}_{\widehat R}$.
   By \autoref{lm:CoordinatesOfNodes}, over the generic point $\spec_{\widehat L}$, the homogeneous coordinates of this section are given by Laurent monomials in the Yoshida and Cross functions.
   By continuity, the expressions for the homogeneous coordinates remain valid over the open subset of $\spec {\widehat R}$ where the Cross functions are invertible. Note that the Yoshida functions are already invertible in $\widehat R$ as they are products of the roots in $\EGp{6}$.
   As a result, if $X \to \spec \KK$ is a marked smooth cubic del Pezzo surface without Eckardt points defined over the valued field $\KK$, then the homogeneous anticanonical coordinates of the 135 nodes are given by Laurent monomials in the Yoshida functions and the Cross functions, as specified in \autoref{lm:CoordinatesOfNodes}.
   \end{remark}

\section{Ruling out extra tropical lines 
for   non-apex Naruki cones}\label{sec:trop-lines-trop}

\autoref{pr:tropicalLinesbyMinors}  provides a powerful tool to check when a finite set of points in $\TPr^{n-1}$ is not collinear. Namely,  it suffices to find a tropically non-singular $3\times 3$-minor in the associated $r\times n$-matrix with entries in $\Rbar$. In this section we exploit this fact to rule out extra tropical lines for tropical cubic surfaces $\Trop X$ whenever the valuations of the 40 Yoshida functions determining $X$ lies in a non-apex cone in the Naruki fan. A separate analysis will be required for the apex. We do this in~\autoref{sec:tropical-lines-trivial}.

The following is the main result in this section and gives a precise version of~\autoref{thm:anticanonicalNoLines}. To simplify the exposition, its proof will be provided by a series of auxiliary technical lemmas and propositions.
\begin{theorem}\label{thm:nonTrivialValNoLifting} Let $X$ be a smooth del Pezzo cubic without Eckardt points with associated  Yoshida functions $\Yos{}\in \P^{39}_{\KK}$. Assume the coordinate-wise valuation of $\Yos{}$ lies outside $\rspanone$. Then, the anticanonically embedded tropical del Pezzo cubic $\Trop X$ has exactly 27 tropical lines.
\end{theorem}

\begin{proof}
  \autoref{lm:IntersectionBoundaryLines} ensures that the boundary of $\Trop X$ contains exactly 27 tropical lines. 
\autoref{thm:candidateTropLines}, \autoref{pr:noInteriorLinesAllCones} and~\autoref{lm:noInteriorLinesACone} imply that there are no tropical lines in $\Trop X$ meeting its interior. By~\autoref{rm:Conditionsi-ii}, this concludes our proof.  
  \end{proof}

\noindent
In the rest of the section, we describe the algorithmic approach to~\autoref{thm:nonTrivialValNoLifting} and the new aforementioned results involved in its proof.

Our computations in~\autoref{sec:comb-extra-trop} determined the existence of at most 27 tropical lines on $\Trop X$  beyond the tropicalization of the 27 classical lines in $X$, each of which has exactly five boundary points. Each extra line is indexed by a symbol in $\sed$ corresponding to the 27 exceptional curves in $X$ as in~\autoref{rm:markings}. \autoref{thm:candidateTropLines} expresses the five boundary points of any potential line  as the tropicalizations of the 135  nodes  in the boundary of the surface $X$. By~\autoref{lm:CoordinatesOfNodes}, the coordinates of these nodes come in two flavors: nine of them are zero, and the remaining 36 are Laurent monomials in the Yoshidas and Cross functions. Furthermore, Cross functions do appear in all nonzero coordinates.

 The Cross functions, expressed as linear binomials in the 40 Yoshida functions yield the  \emph{Cross tropical hyperplane arrangement} in $\R^{40}$  generated by the single hyperplane $\Yos{3}=\Yos{37}$ (associated to the $\Cross{116}$ function) and its 134  $\Wgp$-conjugates. \autoref{prop:naruki_fan_cross} confirms that the Cross arrangement is compatible with the fan structure of $\Naruki$ discussed in~\autoref{sec:bergman-fan-Naruki-fan}.  The valuation of each Cross function cannot be completely determined from a given point in $\Naruki$ if the point  lies in the tropical hyperplane determined by the Cross. Indeed, in the presence of ties between the summands, the valuation of the binomial expression can be higher than expected. Our previous observation allows us to check this condition uniformly on each non-apex cone of $\Naruki$ using any point in its relative interior. We choose the sum of the rays spanning each cone as our witness point.

 The uncertainty of the valuations of these 135 Cross functions makes it a priori impossible to determine all coordinates of boundary points of extra tropical lines on $\Trop X$ from the associated five classical nodes. To take this uncertainty into account, proceed as follows. Given a non-apex cone $\sigma$ in $\Naruki$ and a symbol in $\sed$ indexing a potential tropical line, we compute two $5\times 45$-matrices:
 \begin{enumerate}
 \item a matrix $\TMexp$ with the expected valuations of each of the five classical nodes in $\P^{44}$ (see~\autoref{rm:expValsCross}.) The entries are linear functions in the valuations of the Yoshida functions. For a given point in $\sigma^{\circ}$ it gives an element in $\Rbar$;
 \item a matrix $\TMnone$ whose $(i,j)$ entries come in two types. Given $\sigma$, we conside the collection $\Ccurve$ indexing all Cross hyperplanes intersecting its relative interior. If the $j$th coordinate of the $i$th.\ node in $\P^{44}$ does not contain any Cross function in $\Ccurve$, then the valuation of this coordinate is the expected one, so $\TMnone[i,j]=\TMexp[i,j]$. Otherwise, its valuation cannot be determined.  We record the entry of $\TMnone$ as a `\texttt{None}'.
 \end{enumerate}

 Rather than working with all cones in $\Naruki$, it suffices to pick $\Wgp$-orbit representatives of all maximal cones in the Naruki fan $\Naruki$ and their faces. This reduces our task to 23 non-apex cones. We choose the two top-dimensional  neighboring cones from~\eqref{eq:coneReps} as the representatives of the two orbits of maximal cones in $\Naruki$.

 The values of these 27 pairs of matrices is constant along the relative interior of each of the above 23 cones. Our implementation stores the values of these two matrices at the baricenter of each cone $\sigma$, so the entries of our matrices take values in $\Rbar\cup \{\texttt{None}\}$.  For each cone $\sigma$ in $\Naruki$, we record the 27 pairs of (evaluated) matrices as the family
 \begin{equation}\label{eq:pairsOfMatrices}
   \cF_{\sigma}:=\{(\TMexp[E,\sigma], \TMnone[E,\sigma])\colon E \text{ in  }\sed\}.
 \end{equation}

 We use  the families $\cF_{\sigma}$  to  rule out all the potential extra tropical lines in $\Trop X$ meeting the interior of $\TPr^{44}$  for all cones in $\Naruki$ except the apex.
 The process is described in~\autoref{alg:minorsWithNones} and its implementation in \python~can be found in the Supplementary material. Most of the non-apex cones of $\Naruki$ are covered by~\autoref{pr:noInteriorLinesAllCones}. The cone  of type $\operatorname{(a)}$ requires a different strategy for 15 of the indexing symbols in $\sed$. \autoref{lm:noInteriorLinesACone} discusses these special cases.  It is important to highlight that our method rules out extra lines both for stable and unstable tropical cubic surfaces. 

 \begin{algorithm}[t]   \KwIn{An ordered list $J$ in $\{0,\ldots, 4\}$ of size 3 and a pair of $5\times 45$ matrices $(\TMexp,\TMnone)$ in $\cF_{\sigma}$ of expected and certain coordinates of five  points in the boundary of $\Trop X\subset \TPr^{44}$, for a fixed cone $\sigma$ in the Naruki fan $\Naruki$.}
    \textbf{Assumption:} $X$ is a smooth del Pezzo cubic with no Eckardt points, embedding in $\P^{44}$ via the Yoshida-adapted anticanonical coordinates, where the valuation of its 40 associated Yoshida functions yield a point in the cone $\sigma$. \\
         \KwOut{A 3-element \emph{subset} in $\{0,\ldots, 44\}$ giving the  columns of a tropically non-singular $3\times 3$ minor with rows in $J$ for each pair $(\TMexp, \TMnone)$ in $\cF_{\sigma}$, or a \emph{list} of all 3-element subsets of $\{0,\ldots, 44\}$ giving  tropically non-singular minors of $\TMexp$ if no tropically non-singular minor without `\texttt{None}' entries can be detected.}
\texttt{NonSingMinors} $\leftarrow \emptyset$ ;\ifhmode\\\fi
      \For{$J'$ in $\Subsets([44],3)$}
              {$S_{\expec}(J,J') \leftarrow$  $3\times 3$-submatrix of $\TMexp$ with rows $J$ and columns $J'$\;
                \ifhmode\\\fi
                \If{$S_{\expec}(J,J')$ is tropically non-singular}
                   {\texttt{NonSingMinors} $\leftarrow $ \texttt{NonSingMinors} + [$J'$]\ifhmode\\\fi
                     $S_{\noNones}(J,J') \leftarrow$  $3\times 3$-submatrix of $\TMnone$ with rows $J$ and columns $J'$\;\ifhmode\\\fi
                     \If{$S_{\noNones}(J,J')=S_{\exp}(J,J')$} 
                        {
                          \KwRet{$J'$}
                        }
                   }
              }
              \KwRet{\emph{\texttt{NonSingMinors}}.}
  \caption{Ruling out a potential tropical line in $\TPr^{44}$ meeting the interior of $\Trop X$.\label{alg:minorsWithNones}}
\end{algorithm}

\begin{proposition}\label{pr:noInteriorLinesAllCones}
Assume the valuation of the 40 Yoshida functions associated to $X$ lies outside $\rspanone$. Given a symbol $E$ in $\sed$ associated to a vertex of the Schl\"afli graph, we consider the five nodes on $X$ corresponding to the link of $E$. Then, their tropicalizations in $\TPr^{44}$ are not tropically collinear.
\end{proposition}

\begin{proof}
  Since not all 40 Yoshida functions on $X$ have the same valuation, we know the associated point $p$ in $\Naruki$ is not the apex. We let $\sigma$ be the smallest cone in $\Naruki$ containing the point $p$ in its relative interior, and consider the family $\cF_{\sigma}$ in~\eqref{eq:pairsOfMatrices} encoding the expected and true coordinatewise valuations of the 27 quintuples of classical nodes in the link of $E$ (see~\autoref{lm:CoordinatesOfNodes}.) 

  Given the symbol $E$,~\autoref{alg:minorsWithNones} starts from the pair of matrices $(\TMexp[E,\sigma], \TMnone[E,\sigma])$ and a fixed  3-element set $J$ of $\{0,\ldots, 5\}$  and searches for a 3-element set $J'$ of $\{0,\ldots, 44\}$ giving a tropically non-singular minor of  $\TMnone[E,\sigma]$ with rows in $J$ and columns in $J'$ that agrees with the corresponding minor of $\TMexp[E,\sigma]$. The last condition is tested  by simply checking that the minor of  $\TMnone[C,\sigma]$ has  no `\texttt{None}' entries.  The existence of such a minor together with~\autoref{pr:tropicalLinesbyMinors} would imply that the tropicalization of the five nodes associated to $E$ are not tropically collinear in $\TPr^{44}$.

  If such a minor cannot be found, the Algorithm records the list \texttt{NonSingMinors} of all non-singular $3\times 3$ minors from $\TMexp[E,\sigma]$ with columns in $J$. Our~\python~implementation then repeats the search for the 
 the next set $J$ in the lexicographic order. We iterate this process until all triples of rows have been tested.

  ~\autoref{tab:allCones} gives all instances where~\autoref{alg:minorsWithNones} succeeds. We record the information of rows and columns giving a tropically non-singular $3\times 3$-minor of $\TMnone[C,\sigma]$. The method finds a tropically non-singular  minor for all non-apex cones in $\Naruki$ and all curves $E$ with one exception: the combination of the $\operatorname{(a)}$ cone and the 15 exceptional curves in the  collection 
\begin{equation}\label{eq:failingCurves}  \mathscr{F}:=\{
\E{3}, \E{4}, \E{5}, \E{6}, \F{12}, \F{34}, \F{35}, \F{36}, \F{45}, \F{46}, \F{56}, \Gc{3}, \Gc{4}, \Gc{5}, \Gc{6}\}.
\end{equation}
In these cases,~\autoref{lm:noInteriorLinesACone} and~\autoref{alg:specialACone} provides a tropically non-singular minor for the corresponding element in $\cF_{\operatorname{(a)}}$. This concludes our proof.  
\end{proof}

The failure of~\autoref{alg:minorsWithNones} in providing a certifying minor with  entries in $\Rbar$ for the pairs $(\TMexp[E,\operatorname{(a)}], \TMnone[E,\operatorname{(a)}])$ for a symbol $E$ in the set $\mathscr{F}$  from~\eqref{eq:failingCurves} need not imply the tropical collinearity of the corresponding five boundary points. An analysis of the entries of the list of $3\times 3$-minors of $\TMexp[E,\operatorname{(a)}]$ obtained as output for all $3$-element sets $J$ gives  many instances where these $3\times 3$ submatrices have two non-infinite terms that become `\texttt{None}' entries on the corresponding minor in $\TMnone[E,\operatorname{(a)}]$ both terms involve .

Furthermore, up to permutations of rows and columns, the associated minors in $\TMnone[E,\operatorname{(a)}]$ have the form
\begin{equation}\label{eq:shape}
\begin{pmatrix}
    \ast & \infty & \infty \\
    \infty & \ast & \texttt{None}\\
    \ast& \ast &\texttt{None}
                       \end{pmatrix},
  \end{equation}
where $\ast$ indicates an entry in $\R$. This shape reveals the nature of these three coordinates on the corresponding classical nodes in $X$. In particular, the third entries of the last two nodes involve  Cross functions whose associated tropical hyperplanes contain the cone $\operatorname{(a)}$.

The following lemma shows that, nonetheless, we can completely determine the valuation of their ratio.
Therefore, these two `\texttt{None}' entries will not change the number of terms realizing the minimum in the tropical permanent, showing that the $3\times 3$-minor of $\TMnone$ is  tropically non-singular and thus the quintuple of boundary points fails to be tropically collinear.

\begin{algorithm}[htb]
  \KwIn{A 3-element set $J\subset \{0,\ldots, 5\}$, a symbol $E$ in $\sed$, a $5\times 45$ matrix $N$ giving the five classical nodes associated to the link of $E$ in the Schl\"afli graph, a list $\scrS$ of 3-element sets in $\{0,\ldots, 44\}$ encoding all tropically non-singular $3\times 3$-minors in  $\TMexp[E,\sigma]$  for a fixed cone $\sigma$ in the Naruki fan $\Naruki$, and a list $\cI$ of row indices giving a $\Z$-basis of the row space of the Yoshida matrix}
    \textbf{Assumptions:} $X$ is a smooth del Pezzo cubic with no Eckardt points, embedding in $\P^{44}$ via~\autoref{thm:anticanK}.  The valuation of its 40 Yoshida functions is a point in the relative interior of  $\sigma$.   The output of~\autoref{alg:minorsWithNones} for $E$ and $J$ is the list $\scrS$ ($=$\texttt{NonSingMinors}.)\\
         \KwOut{The empty list or a pair consisting of: (1) an element  $J'$ in $\scrS$ giving the three columns of a tropically non-singular minor with rows  $J$ for  $(\TMexp[E,\sigma], \TMnone[E,\sigma])$, and (2) the expression of the third power of the ratio of the two relevant entries with unknown valuations in $N$ as a Laurent monomial in the Yoshidas indexed by $\cI$.  }
      \For{$J'$ in $\scrS$}
          {$S_{\expec}(J,J') \leftarrow$  $3\times 3$-submatrix of $\TMexp[E,\sigma]$ with rows $J$ and columns $J'$\;
                \ifhmode\\\fi
                \If{$S_{\expec}(J,J')$ is tropically non-singular}
                   {$S_{\noNones}(J,J') \leftarrow$  $3\times 3$-submatrix of $\TMnone[E,\sigma]$ with rows $J$ and columns $J'$\;\ifhmode\\\fi
                     \For{$(\tau_1,\tau_2)\in \Sn{3}\times \Sn{3}$}
                         {$(N^{\new}, S_{\noNones}^{\new}) \leftarrow ((\tau_1,\tau_2)\!\cdot\! N,(\tau_1,\tau_2)\!\cdot\! S_{\noNones}(J,J'))$\;\ifhmode\\\fi
                           
                     \If{$S_{\noNones}^{\new}$ has the entries-pattern of~\eqref{eq:shape} 
                       (where $\ast$ indicates a real number)}
                        {$(A,B) \leftarrow $ product of all Cross factors in each $(N^{\new}[2,3], N^{\new}[3,3])$, converted into a rational function in $d_1,\ldots, d_6$ \;\ifhmode\\\fi             
                          \If{$(A/B)^3$ is a monomial in the Yoshida functions indexed by $\cI$}
                          {  \KwRet{$(J',  (A/B)^3)$.}}
                        }
                         }
                   }
          }
              \KwRet{$[\;]$.}
  \caption{Finding tropically non-singular $3\times 3$-minors involving unknown valuations.\label{alg:specialACone}}
\end{algorithm}

\begin{lemma}\label{lm:noInteriorLinesACone}
  Let  $E$ be a curve in the collection $\mathscr{F}$ from~\eqref{eq:failingCurves} and let $N$ be the $5\times 45$ matrix giving the five classical nodes associated to the link of $E$ in the Schl\"afli graph. Then, there exists a pair of unordered 3-element sets $J$ in  $\{0,\ldots, 4\}$ and $J'$  in $\{0,\ldots, 44\}$ satisfying three conditions:
  \begin{enumerate}[(i)]
  \item    the submatrix of  $\TMexp[E,\operatorname{(a)}]$ with rows $J$ and columns $J'$ is tropically non-singular,
  \item the submatrix of $\TMnone[E,\operatorname{(a)}]$ with  rows $J$ and columns $J'$ has the shape~\eqref{eq:shape},
    \item the expression $(\frac{N[J[2],J'[3]]}{N[J[3],J'[3]]})^3$ is a Laurent monomial in the Yoshida functions.
  \end{enumerate}
\end{lemma}

\begin{proof} The discussion preceeding the statement shows why it is conceivable to find a pair $(J,J')$ satisfying (i) and (ii). It gives rise to~\autoref{alg:specialACone}, which enables us to prove (iii) by explicit computations. A direct factorization of the expression in (iii) is not possible since the Yoshida functions are not algebraically independent. Thus, we are forced to express all relevant functions in terms of the parameters $d_1,\ldots, d_6$.

    From the shape~\eqref{eq:shape} we deduce that  $N[J[2],J'[3]]$ and $N[J[3],J'[3]]$ are the only entries in the submatrix of $N$ with rows $J$ and columns $J'$ with undetermined valuations. This means that both entries contain Cross functions in their factorization whose valuations cannot be determined along the cone $\operatorname{(a)}$. We let $A$ (respectively B) be the product of all the Cross factors in $N[J[2],J'[3]]$ (resp.~$N[J[3],J'[3]]$.) A  simple calculation shows that the ratio $A/B$ is a Laurent monomial in the 36 positive roots of $\EGp{6}$ listed in~\autoref{tab:roots}.

    By~\autoref{rm:index3oFYoshidasInDs} we know that the exponent vectors of the 40 Yoshida functions span a rank 16 sublattice of $\Z^{36}$ of index 3. A basis of this lattice is provided by the Yoshida functions indexed by the set $\cI = \{5\}\cup\{17,\ldots, 31\}$ (see the Supplementary material.)  In particular, even though the ratio $A/B$ need not be a Laurent monomial in the Yoshida functions indexed by $\cI$, its cube will be.    We certify this last step as follows. We factor $(A/B)^3$ as a Laurent monomial in the positive roots of $\EGp{6}$ and let $v$ be its exponent vector in $\Z^{36}$. We solve the linear system of equations $M^{t} x = v^t$, where $M$ is the $16\times 36$-submatrix of the Yoshida matrix (seen in~\autoref{tab:yoshida}) with rows in $\cI$. If it exists, its unique solution will lie in $\Z^{16}$ and it will be  the exponent vector giving the desired factorization of $(A/B)^3$. Thus,  the cubed ratio in (iii) has the same desired property.

    For each choice of $J$ and $E$,  we run our \python~implementation of~\autoref{alg:specialACone} (available in the Supplementary material.) As the input set $\scrS$ we use the output \texttt{NonSingMinors} from~\autoref{alg:minorsWithNones} obtained from the input $J$ and $E$. 
    \autoref{tab:coneA} shows the choice of rows $J$ for each symbol $E$ in $\mathscr{F}$ from~\eqref{eq:failingCurves} where~\autoref{alg:minorsWithNones} failed to give the desired non-singular minor in $\TMnone$ without `\texttt{None}' entries but~\autoref{alg:specialACone} succeeded in providing a minor satisfying conditions (i) through (iii). 
\end{proof}

\begin{table}[htb]
  \centering
  \begin{tabular}{|c|c|c|c|c|}
 \hline Curve & Rows & Columns & Minor Shape & \small{$(A/B)^3$} \normalsize as a Laurent monomial in the Yoshidas \\ \hline
   $\E{3}$ & $[0,1,2]$ & $[8,9,24]$ &
\scalebox{0.7}{  $\begin{pmatrix}
    \ast & \infty & \infty \\
    \infty & \ast & {\texttt{None}}\\
    \ast & \ast & {\texttt{None}}
\end{pmatrix}$}\normalsize & -1\\ \hline
$\E{4}$ & $[0,1,2]$ & $[8,9,13]$ &
\scalebox{0.7}{  $\begin{pmatrix}
    \ast & \infty & {\texttt{None}} \\
    \infty & \ast & \infty\\
    \ast & \ast & {\texttt{None}}
\end{pmatrix}$}\normalsize &
$\Yos{5}\, \Yos{22}\, \Yos{26}^{2}\,  \Yos{30}/( \Yos{18}^{2}\, \Yos{20}\, \Yos{21} \, \Yos{27})$
\\\hline
  $\E{5}$ & $[0,1,2]$ & $[7,32,33]$ &
\scalebox{0.7}{$\begin{pmatrix}
     \ast & {\texttt{None}}& \infty\\
    \infty& \infty & \ast \\
    \ast &{\texttt{None}}& \ast
  \end{pmatrix}$}\normalsize &
$(-1) \, \Yos{17}^{3} \, \Yos{18}^{2} \, \Yos{20}  \, \Yos{26} \, \Yos{27} \, \Yos{30}^{2}/ (\Yos{5}  \, \Yos{19}^{3} \, \Yos{21}^{2} \, \Yos{22}  \, \Yos{29}^{3})$
\\\hline
  $\E{6}$ & $[0,1,2]$ & $[7,8,12]$ &
\scalebox{0.7}{  $\begin{pmatrix}
    \infty & \ast & \infty \\
    \ast & \infty & {\texttt{None}}\\
    \ast & \ast &{\texttt{None}}
  \end{pmatrix}$}\normalsize & \ 1\\\hline
   $\F{12}$ & $[0,2,4]$ & $[17,26,44]$ &
\scalebox{0.7}{  $\begin{pmatrix}
    {\texttt{None}} & \infty & \ast \\
    \infty & \ast & \infty\\
    {\texttt{None}} & \ast & \ast
  \end{pmatrix}$}\normalsize & -1\\\hline
   $\F{34}$ & $[0,2,4]$ & $[5,7,23]$ &
\scalebox{0.7}{  $\begin{pmatrix}
    \ast & \infty & \infty \\
    \infty & \ast & {\texttt{None}}\\
    \ast & \ast & {\texttt{None}}
  \end{pmatrix}$}\normalsize & \ 1\\\hline
   $\F{35}$ & $[0,1,2]$ & $[6,15,23]$ &
\scalebox{0.7}{  $\begin{pmatrix}
    \ast & \ast & {\texttt{None}} \\
    \infty & \ast & \infty\\
    \ast & \infty & {\texttt{None}}
\end{pmatrix}$}\normalsize &
 $\Yos{5}^{2}\, \Yos{21} \, \Yos{25}^{3} \, \Yos{26} \, \Yos{27}  \, \Yos{29}^{3}  /  ( \Yos{18} \, \Yos{20}^{2}  \, \Yos{22} \, \Yos{24}^{3}\, \Yos{28}^{3}\, \Yos{30})$
\\\hline
$\F{36}$ & $[0,1,2]$ & $[6,9,19]$ &
\scalebox{0.7}{ $\begin{pmatrix}
    \ast & \infty & \infty \\
    \infty & \ast & {\texttt{None}}\\
    \ast & \ast &{\texttt{None}}
  \end{pmatrix}$}\normalsize & -1\\\hline
$\F{45}$ & $[0,1,2]$ & $[7,10,12]$ &
\scalebox{0.7}{$\begin{pmatrix}
    \ast & \ast & {\texttt{None}} \\
    \ast & \infty & \infty\\
    \infty & \ast &{\texttt{None}}
  \end{pmatrix}$}\normalsize &   $\Yos{5}\, \Yos{18} \,\Yos{26}^{2}\,  \Yos{31}^{3}/( \Yos{20}\,\Yos{21}\, \Yos{22}^{2}\,\Yos{27}\, \Yos{30}^{2})$
\\\hline
 $\F{46}$ & $[0,1,2]$ & $[8,9,13]$ & \scalebox{0.7}{$\begin{pmatrix}
    \ast & \infty & {\texttt{None}} \\
    \infty & \ast & \infty\\
    \ast & \ast &{\texttt{None}}
  \end{pmatrix}$}\normalsize &
$ \Yos{18} \, \Yos{20}^{2} \, \Yos{22} \, \Yos{24}^{3} \, \Yos{28}^{3}  \, \Yos{30} / (\Yos{5}^{2} \, \Yos{21}  \, \Yos{25}^{3} \, \Yos{26} \, \Yos{27} \, \Yos{29}^{3})$
\\\hline
 $\F{56}$ & $[0,1,2]$ & $[8,9,13]$ & \scalebox{0.7}{$\begin{pmatrix}
    \ast & \infty & {\texttt{None}} \\
    \infty & \ast & \infty\\
    \ast & \ast &{\texttt{None}}
  \end{pmatrix}$}\normalsize & $\Yos5^2\,  \Yos{20}\, \Yos{22}^{2}\, \Yos{26}\,  \Yos{28}^{3}\,  \Yos{31}^{3} / (\Yos{17}^{3}\,\Yos{18}\,\Yos{21}^{2}\, \Yos{25}^{3}\,\Yos{27}^{2}\,\Yos{30})$
\\\hline
   $\Gc{3}$ & $[0,2,4]$ & $[30,35,44]$ &
\scalebox{0.7}{  $\begin{pmatrix}
    {\texttt{None}} & \infty & \ast \\
    \infty & \ast & \infty\\
    {\texttt{None}} & \ast & \ast
  \end{pmatrix}$}\normalsize & \ 1\\\hline
   $\Gc{4}$ & $[0,3,4]$ & $[20,32,38]$ &
\scalebox{0.7}{  $\begin{pmatrix}
    \ast & {\texttt{None}} & \infty \\
    \infty & \infty & \ast\\
    \ast & {\texttt{None}} & \ast
\end{pmatrix}$}\normalsize &
$(-1)\, \Yos{18} \, \Yos{20}^{2} \, \Yos{22} \, \Yos{30} /(  \Yos{5}^{2} \, \Yos{21}  \, \Yos{26} \, \Yos{27})$
\\\hline
   $\Gc{5}$ & $[0,2,4]$ & $[32,34,42]$ &
\scalebox{0.7}{  $\begin{pmatrix}
    {\texttt{None}} & \ast & \ast \\
    \infty & \ast & \infty\\
    {\texttt{None}} & \ast & \ast
  \end{pmatrix}$}\normalsize & \ 1\\\hline
   $\Gc{6}$ & $[0,1,3]$ & $[11,32,42]$ &
\scalebox{0.7}{  $\begin{pmatrix}
    \ast & {\texttt{None}} & \infty \\
    \infty & \infty & \ast\\
    \ast &  {\texttt{None}} & \ast
  \end{pmatrix}$}\normalsize & \ 1\\\hline
  \end{tabular}
  \caption{Ruling out the remaining 15 potential interior tropical lines for tropical surfaces with associated Naruki cone $\operatorname{\mathbf{(a)}}$  that are not covered by~\autoref{tab:allCones}. The entries $A$ and $B$ are products of all Cross factors in the coordinates of two classical nodes responsible for the `\texttt{None}' entries in the minor of $\TMnone$ (from top to bottom.)\label{tab:coneA}}
\end{table}

The proof of~\autoref{thm:anticanonicalNoLines} is now complete.

\afterpage{
  \clearpage
\thispagestyle{empty}
  \begin{landscape}
   \begin{table}[htb]
     \hspace{-15ex} \;\tiny{       \begin{tabular}{|c|c||c|c|c|c|c|c||c|c|c|c|c|c|c|c|c|c|c|c|c|c|c|}
         \hline
         \scriptsize{Cone} & \scriptsize{Rows} & \scriptsize{\E{1}}&\scriptsize{\E{2}}&\scriptsize{\E{3}}&\scriptsize{\E{4}}&\scriptsize{\E{5}}&\scriptsize{\E{6}}&\scriptsize{\F{12}}&\scriptsize{\F{13}}&\scriptsize{\F{14}}&\scriptsize{\F{15}}&\scriptsize{\F{16}}&\scriptsize{\F{23}}&\scriptsize{\F{24}}&\scriptsize{\F{25}}&\scriptsize{\F{26}}&\scriptsize{\F{34}}&\scriptsize{\F{35}}&\scriptsize{\F{36}}&\scriptsize{\F{45}}&\scriptsize{\F{46}}&\scriptsize{\F{56}}          \\\hline\hline
                  \tiny{\!\!\!\textbf{(aa2a3a4)}}\!\!\!& \!\!\scriptsize{012}\!\!\!& 
                   \!\!0,1,2\!\!&
           \!\!0,2,12\!\!&
\!\!0,1,2\!\!&
           \!\!0,1,4\!\!&
             \!\!0,1,4\!\!&
         \!\!0,1,15\!\!&
         \!\!0,8,36\!\!&
         \!\!0,2,5\!\!&
         \!\!0,1,6\!\!&
         \!\!0,1,5\!\!&
         \!\!0,1,5\!\!&
         \!\!0,1,2\!\!&
         \!\!0,1,2\!\!&
         \!\!0,1,3\!\!&
         \!\!0,1,2\!\!&
         \!\!0,2,12\!\!&
         \!\!0,1,2\!\!&
         \!\!0,1,3\!\!&
         \!\!0,1,7\!\!&
         \!\!0,2,3\!\!&
         \!\!0,2,3\!\!\\
\hline
                  \tiny{\!\!\!\textbf{(aa2a3b)}}\!\!\!& \!\!\scriptsize{012}\!\!\!& 
                   \!\!0,2,9\!\!&
           \!\!0,2,18\!\!&
\!\!1,2,7\!\!&
           \!\!0,1,7\!\!&
             \!\!0,1,7\!\!&
         \!\!0,2,7\!\!&
         \!\!0,2,7\!\!&
         \!\!0,5,11\!\!&
         \!\!0,1,6\!\!&
         \!\!0,5,10\!\!&
         \!\!0,1,7\!\!&
         \!\!0,1,2\!\!&
         \!\!0,1,2\!\!&
         \!\!0,1,2\!\!&
         \!\!0,1,10\!\!&
         \!\!0,2,18\!\!&
         \!\!0,1,2\!\!&
         \!\!0,1,2\!\!&
         \!\!0,1,7\!\!&
         \!\!0,2,4\!\!&
         \!\!0,2,4\!\!\\
\hline\hline
                  \scriptsize{\!\!\!\textbf{(aa2a3)}}\!\!\!& \!\!\scriptsize{012}\!\!\!& 
                   \!\!0,2,12\!\!&
           \!\!0,2,18\!\!&
\!\!1,2,5\!\!&
           \!\!0,1,7\!\!&
             \!\!0,1,7\!\!&
         \!\!0,1,15\!\!&
         \!\!0,8,36\!\!&
         \!\!0,5,11\!\!&
         \!\!0,1,6\!\!&
         \!\!0,5,10\!\!&
         \!\!0,1,7\!\!&
         \!\!0,1,2\!\!&
         \!\!0,1,2\!\!&
         \!\!0,1,3\!\!&
         \!\!0,1,2\!\!&
         \!\!0,2,18\!\!&
         \!\!0,1,2\!\!&
         \!\!0,1,4\!\!&
         \!\!0,1,7\!\!&
         \!\!0,2,4\!\!&
         \!\!0,2,4\!\!\\
\hline

                  \scriptsize{\!\!\!\textbf{(aa2a4)}}\!\!\!& \!\!\scriptsize{012}\!\!\!& 
                   \!\!0,1,2\!\!&
           \!\!0,2,12\!\!&
\!\!0,2,4\!\!&
           \!\!0,1,4\!\!&
             \!\!0,1,4\!\!&
         \!\!0,1,15\!\!&
         \!\!0,8,36\!\!&
         \!\!0,2,5\!\!&
         \!\!0,1,6\!\!&
         \!\!0,1,5\!\!&
         \!\!0,1,5\!\!&
         \!\!0,1,2\!\!&
         \!\!0,1,2\!\!&
         \!\!0,1,3\!\!&
         \!\!0,1,2\!\!&
         \!\!0,2,12\!\!&
         \!\!0,1,2\!\!&
         \!\!0,1,3\!\!&
         \!\!0,1,7\!\!&
         \!\!0,2,3\!\!&
         \!\!0,2,3\!\!\\
\hline

                  \scriptsize{\!\!\!\textbf{(aa3a4)}}\!\!\!& \!\!\scriptsize{012}\!\!\!& 
                   \!\!0,1,2\!\!&
           \!\!0,2,12\!\!&
\!\!0,1,2\!\!&
           \!\!0,1,4\!\!&
             \!\!0,1,4\!\!&
         \!\!0,1,15\!\!&
         \!\!0,8,36\!\!&
         \!\!0,2,5\!\!&
         \!\!0,1,6\!\!&
         \!\!0,1,5\!\!&
         \!\!0,1,5\!\!&
         \!\!0,1,2\!\!&
         \!\!0,1,2\!\!&
         \!\!0,1,3\!\!&
         \!\!0,1,2\!\!&
         \!\!0,2,12\!\!&
         \!\!0,1,2\!\!&
         \!\!0,1,3\!\!&
         \!\!0,1,20\!\!&
         \!\!0,2,3\!\!&
         \!\!0,2,3\!\!\\
\hline
                  \scriptsize{\!\!\!\textbf{(a2a3a4)}}\!\!\!& \!\!\scriptsize{012}\!\!\!& 
                   \!\!0,2,9\!\!&
           \!\!0,2,12\!\!&
\!\!0,2,4\!\!&
           \!\!0,2,4\!\!&
             \!\!0,2,4\!\!&
         \!\!0,2,12\!\!&
         \!\!0,8,36\!\!&
         \!\!0,2,5\!\!&
         \!\!0,2,8\!\!&
         \!\!0,2,5\!\!&
         \!\!0,2,5\!\!&
         \!\!0,2,3\!\!&
         \!\!0,2,4\!\!&
         \!\!0,2,3\!\!&
         \!\!0,2,5\!\!&
         \!\!0,2,12\!\!&
         \!\!0,2,5\!\!&
         \!\!0,2,3\!\!&
         \!\!0,2,7\!\!&
         \!\!0,2,3\!\!&
         \!\!0,2,3\!\!\\
\hline

                  \scriptsize{\!\!\!\textbf{(aa2b)}}\!\!\!& \!\!\scriptsize{012}\!\!\!& 
                   \!\!0,2,13\!\!&
           \!\!0,2,23\!\!&
\!\!1,3,7\!\!&
           \!\!0,2,8\!\!&
             \!\!0,2,7\!\!&
         \!\!0,2,7\!\!&
         \!\!0,2,7\!\!&
         \!\!0,5,11\!\!&
         \!\!0,1,6\!\!&
         \!\!0,5,10\!\!&
         \!\!0,1,7\!\!&
         \!\!0,1,2\!\!&
         \!\!0,1,2\!\!&
         \!\!0,1,2\!\!&
         \!\!0,1,10\!\!&
         \!\!0,2,25\!\!&
         \!\!0,1,2\!\!&
         \!\!0,1,2\!\!&
         \!\!0,1,7\!\!&
         \!\!0,2,7\!\!&
         \!\!0,2,8\!\!\\
\hline

                  \scriptsize{\!\!\!\textbf{(aa3b)}}\!\!\!& \!\!\scriptsize{012}\!\!\!& 
                   \!\!0,2,9\!\!&
           \!\!0,2,18\!\!&
\!\!1,2,7\!\!&
           \!\!0,1,7\!\!&
             \!\!0,1,7\!\!&
         \!\!0,2,7\!\!&
         \!\!0,2,7\!\!&
         \!\!0,5,11\!\!&
         \!\!0,1,6\!\!&
         \!\!0,5,10\!\!&
         \!\!0,1,7\!\!&
         \!\!0,1,2\!\!&
         \!\!0,1,2\!\!&
         \!\!0,1,2\!\!&
         \!\!0,1,10\!\!&
         \!\!0,2,18\!\!&
         \!\!0,1,2\!\!&
         \!\!0,1,2\!\!&
         \!\!0,1,33\!\!&
         \!\!0,2,4\!\!&
         \!\!0,2,4\!\!\\
\hline

                  \scriptsize{\!\!\!\textbf{(a2a3b)}}\!\!\!& \!\!\scriptsize{012}\!\!\!& 
                   \!\!0,2,9\!\!&
           \!\!0,2,18\!\!&
\!\!2,3,7\!\!&
           \!\!0,2,4\!\!&
             \!\!0,2,4\!\!&
         \!\!0,2,7\!\!&
         \!\!0,2,7\!\!&
         \!\!0,5,12\!\!&
         \!\!0,2,8\!\!&
         \!\!0,5,10\!\!&
         \!\!0,2,7\!\!&
         \!\!0,2,7\!\!&
         \!\!0,2,4\!\!&
         \!\!0,2,7\!\!&
         \!\!0,2,4\!\!&
         \!\!0,2,18\!\!&
         \!\!0,2,3\!\!&
         \!\!0,2,10\!\!&
         \!\!0,2,7\!\!&
         \!\!0,2,4\!\!&
         \!\!0,2,4\!\!\\
\hline
\hline

                 \scriptsize{\!\!\!\textbf{(aa2)}}\!\!\!& \!\!\scriptsize{012}\!\!\!& 
                   \!\!1,8,13\!\!&
           \!\!0,2,23\!\!&
\!\!1,3,5\!\!&
           \!\!0,6,8\!\!&
             \!\!0,6,7\!\!&
         \!\!0,6,7\!\!&
         \!\!\!{2,18,33}\!\!\!&
         \!\!0,5,11\!\!&
         \!\!0,1,6\!\!&
         \!\!0,5,10\!\!&
         \!\!0,1,7\!\!&
         \!\!0,1,2\!\!&
         \!\!0,1,2\!\!&
         \!\!0,1,3\!\!&
         \!\!0,1,2\!\!&
         \!\!0,2,25\!\!&
         \!\!0,1,2\!\!&
         \!\!0,1,4\!\!&
         \!\!0,1,7\!\!&
         \!\!0,2,7\!\!&
         \!\!0,3,8\!\!\\
\hline

                 \scriptsize{\!\!\!\textbf{(aa3)}}\!\!\!& \!\!\scriptsize{012}\!\!\!& 
                   \!\!0,2,12\!\!&
           \!\!0,2,18\!\!&
\!\!1,2,5\!\!&
           \!\!0,1,7\!\!&
             \!\!0,1,7\!\!&
         \!\!0,2,13\!\!&
         \!\!0,8,36\!\!&
         \!\!0,5,11\!\!&
         \!\!0,1,6\!\!&
         \!\!0,5,10\!\!&
         \!\!0,1,7\!\!&
         \!\!0,1,2\!\!&
         \!\!0,1,2\!\!&
         \!\!0,1,3\!\!&
         \!\!0,1,2\!\!&
         \!\!0,2,18\!\!&
         \!\!0,1,2\!\!&
         \!\!0,1,4\!\!&
         \!\!0,1,33\!\!&
         \!\!0,2,4\!\!&
         \!\!0,2,4\!\!\\
\hline

                 \scriptsize{\!\!\!\textbf{(aa4)}}\!\!\!& \!\!\scriptsize{012}\!\!\!& 
                   \!\!0,2,9\!\!&
           \!\!0,2,12\!\!&
\!\!0,2,4\!\!&
           \!\!0,2,4\!\!&
             \!\!0,2,4\!\!&
         \!\!0,2,15\!\!&
         \!\!0,8,36\!\!&
         \!\!0,2,5\!\!&
         \!\!0,2,8\!\!&
         \!\!0,2,5\!\!&
         \!\!0,2,5\!\!&
         \!\!0,2,3\!\!&
         \!\!0,2,4\!\!&
         \!\!0,2,3\!\!&
         \!\!0,2,7\!\!&
         \!\!0,2,12\!\!&
         \!\!0,2,5\!\!&
         \!\!0,2,3\!\!&
         \!\!0,2,7\!\!&
         \!\!0,2,3\!\!&
         \!\!0,2,3\!\!\\
\hline

                 \scriptsize{\!\!\!\textbf{(a2a3)}}\!\!\!& \!\!\scriptsize{012}\!\!\!& 
                   \!\!0,2,12\!\!&
           \!\!0,2,19\!\!&
\!\!2,3,5\!\!&
           \!\!0,2,4\!\!&
             \!\!0,2,4\!\!&
         \!\!0,2,13\!\!&
         \!\!0,8,36\!\!&
         \!\!0,5,14\!\!&
         \!\!0,2,8\!\!&
         \!\!0,5,10\!\!&
         \!\!0,2,7\!\!&
         \!\!0,2,7\!\!&
         \!\!0,2,10\!\!&
         \!\!0,2,7\!\!&
         \!\!0,2,7\!\!&
         \!\!0,2,18\!\!&
         \!\!0,2,5\!\!&
         \!\!0,2,16\!\!&
         \!\!0,2,7\!\!&
         \!\!0,2,4\!\!&
         \!\!0,2,4\!\!\\
\hline

                 \scriptsize{\!\!\!\textbf{(a2a4)}}\!\!\!& \!\!\scriptsize{012}\!\!\!& 
                   \!\!0,2,9\!\!&
           \!\!0,2,12\!\!&
\!\!0,2,4\!\!&
           \!\!0,2,4\!\!&
             \!\!0,2,4\!\!&
         \!\!0,2,12\!\!&
         \!\!0,8,36\!\!&
         \!\!0,2,5\!\!&
         \!\!0,2,8\!\!&
         \!\!0,2,5\!\!&
         \!\!0,2,5\!\!&
         \!\!0,2,3\!\!&
         \!\!0,2,4\!\!&
         \!\!0,2,3\!\!&
         \!\!0,2,5\!\!&
         \!\!0,2,12\!\!&
         \!\!0,2,5\!\!&
         \!\!0,2,3\!\!&
         \!\!0,2,7\!\!&
         \!\!0,2,3\!\!&
         \!\!0,2,3\!\!\\
\hline

                 \scriptsize{\!\!\!\textbf{(a3a4)}}\!\!\!& \!\!\scriptsize{012}\!\!\!& 
                   \!\!0,2,9\!\!&
           \!\!0,2,12\!\!&
\!\!0,2,4\!\!&
           \!\!0,2,4\!\!&
             \!\!0,2,4\!\!&
         \!\!0,2,12\!\!&
         \!\!0,8,36\!\!&
         \!\!0,2,5\!\!&
         \!\!0,2,8\!\!&
         \!\!0,2,5\!\!&
         \!\!0,2,5\!\!&
         \!\!0,2,3\!\!&
         \!\!0,2,4\!\!&
         \!\!0,2,3\!\!&
         \!\!0,2,5\!\!&
         \!\!0,2,12\!\!&
         \!\!0,2,5\!\!&
         \!\!0,2,3\!\!&
         \!\!0,2,7\!\!&
         \!\!0,2,3\!\!&
         \!\!0,2,3\!\!\\
\hline

\hline

                 \scriptsize{\!\!\!\textbf{(ab)}}\!\!\!& \!\!\scriptsize{012}\!\!\!& 
                    \!\!0,2,20\!\!&
           \!\!---\!\!&
\!\!1,3,12\!\!&
           \!\!1,2,8\!\!&
             \!\!1,2,7\!\!&
         \!\!1,2,7\!\!&
         \!\!0,2,7\!\!&
         \!\!0,5,11\!\!&
         \!\!0,5,16\!\!&
         \!\!0,5,11\!\!&
         \!\!0,5,10\!\!&
         \!\!0,1,2\!\!&
         \!\!0,1,2\!\!&
         \!\!0,2,7\!\!&
         \!\!0,1,10\!\!&
         \!\!---\!\!&
         \!\!0,1,2\!\!&
         \!\!0,1,3\!\!&
         \!\!0,1,37\!\!&
         \!\!0,2,7\!\!&
         \!\!0,2,8\!\! \\
& \!\!\scriptsize{013}\!\!\!& 
                    \!\!\!\!&
           \!\!\!2,3,20\!\!\!&
\!\!\!\!&
           \!\!\!\!&
             \!\!\!\!&
         \!\!\!\!&
         \!\!\!\!&
         \!\!\!\!&
         \!\!\!\!&
         \!\!\!\!&
         \!\!\!\!&
         \!\!\!\!&
         \!\!\!\!&
         \!\!\!\!&
         \!\!\!\!&
         \!\!0,2,3\!\!&
         \!\!\!\!&
         \!\!\!\!&
         \!\!\!\!&
         \!\!\!\!&
         \!\!\!\! \\
\hline

                 \scriptsize{\!\!\!\textbf{(a2b)}}\!\!\!& \!\!\scriptsize{012}\!\!\!& 
                   \!\!0,2,13\!\!&
           \!\!0,2,23\!\!&
\!\!2,3,7\!\!&
           \!\!0,2,8\!\!&
             \!\!0,2,7\!\!&
         \!\!0,2,7\!\!&
         \!\!0,2,7\!\!&
         \!\!0,5,12\!\!&
         \!\!0,2,8\!\!&
         \!\!0,5,10\!\!&
         \!\!0,2,7\!\!&
         \!\!0,2,7\!\!&
         \!\!0,2,10\!\!&
         \!\!0,2,7\!\!&
         \!\!0,2,4\!\!&
         \!\!0,2,25\!\!&
         \!\!0,2,3\!\!&
         \!\!0,2,10\!\!&
         \!\!0,2,7\!\!&
         \!\!0,2,7\!\!&
         \!\!0,2,8\!\! \\
\hline

                 \scriptsize{\!\!\!\textbf{(a3b)}}\!\!\!& \!\!\scriptsize{012}\!\!\!& 
                   \!\!0,2,9\!\!&
           \!\!0,2,18\!\!&
\!\!2,3,7\!\!&
           \!\!0,2,4\!\!&
             \!\!0,2,4\!\!&
         \!\!0,2,7\!\!&
         \!\!0,2,7\!\!&
         \!\!0,5,12\!\!&
         \!\!0,2,8\!\!&
         \!\!0,5,10\!\!&
         \!\!0,2,7\!\!&
         \!\!0,2,7\!\!&
         \!\!0,2,4\!\!&
         \!\!0,2,9\!\!&
         \!\!0,2,4\!\!&
         \!\!0,2,18\!\!&
         \!\!0,2,3\!\!&
         \!\!0,2,10\!\!&
         \!\!0,2,7\!\!&
         \!\!0,2,4\!\!&
         \!\!0,2,4\!\! \\
\hline \hline
        \scriptsize{\!\!\!\textbf{(a)}}\!\!\!& \!\!\scriptsize{012}\!\!\!& 
                   \!\!2,8,20\!\!&
           \!\!---\!\!&
\!\!---\!\!&
           \!\!---\!\!&
             \!\!---\!\!&
         \!\!---\!\!&
         \!\!---\!\!&
         \!\!0,5,11\!\!&
         \!\!0,5,16\!\!&
         \!\!0,5,11\!\!&
         \!\!0,5,10\!\!&
         \!\!0,1,10\!\!&
         \!\!0,1,6\!\!&
         \!\!0,2,7\!\!&
         \!\!0,1,10\!\!&
         \!\!---\!\!&
         \!\!---\!\!&
         \!\!---\!\!&
         \!\!---\!\!&
         \!\!---\!\!&
         \!\!---\!\! \\

& \!\!\scriptsize{013}\!\!\!& 
                   \!\!\!\!&
           \!\!\!3,4,20\!\!\!&
\!\!---\!\!&
           \!\!---\!\!&
             \!\!---\!\!&
         \!\!---\!\!&
         \!\!---\!\!&
         \!\!\!\!&
         \!\!\!\!&
         \!\!\!\!&
         \!\!\!\!&
         \!\!\!\!&
         \!\!\!\!&
         \!\!\!\!&
         \!\!\!\!&
         \!\!---\!\!&
         \!\!---\!\!&
         \!\!---\!\!&
         \!\!---\!\!&
         \!\!---\!\!&
         \!\!---\!\! \\
& \!\!\scriptsize{rest}\!\!\!& 
                   \!\!\!\!&
           \!\!\!\!&
\!\!---\!\!&
           \!\!---\!\!&
             \!\!---\!\!&
         \!\!---\!\!&
         \!\!---\!\!&
         \!\!\!\!&
         \!\!\!\!&
         \!\!\!\!&
         \!\!\!\!&
         \!\!\!\!&
         \!\!\!\!&
         \!\!\!\!&
         \!\!\!\!&
         \!\!---\!\!&
         \!\!---\!\!&
         \!\!---\!\!&
         \!\!---\!\!&
         \!\!---\!\!&
         \!\!---\!\!\\

\hline

                 \scriptsize{\!\!\!\textbf{(a2)}}\!\!\!& \!\!\scriptsize{012}\!\!\!& 
                   \!\!2,8,13\!\!&
           \!\!0,2,33\!\!&
\!\!2,3,5\!\!&
           \!\!0,8,13\!\!&
             \!\!0,7,12\!\!&
         \!\!0,7,13\!\!&
         \!\!2,18,36\!\!&
         \!\!0,5,14\!\!&
         \!\!0,3,8\!\!&
         \!\!0,5,10\!\!&
         \!\!0,2,7\!\!&
         \!\!0,2,7\!\!&
         \!\!0,2,10\!\!&
         \!\!0,2,7\!\!&
         \!\!0,2,7\!\!&
         \!\!0,2,25\!\!&
         \!\!0,2,15\!\!&
         \!\!0,2,26\!\!&
         \!\!0,2,7\!\!&
         \!\!0,2,7\!\!&
         \!\!0,3,8\!\!\\
\hline

                 \scriptsize{\!\!\!\textbf{(a3)}}\!\!\!& \!\!\scriptsize{012}\!\!\!& 
                   \!\!0,2,12\!\!&
           \!\!0,2,19\!\!&
\!\!2,3,7\!\!&
           \!\!0,2,4\!\!&
             \!\!0,2,4\!\!&
         \!\!0,2,13\!\!&
         \!\!0,8,36\!\!&
         \!\!0,5,14\!\!&
         \!\!0,2,8\!\!&
         \!\!0,5,10\!\!&
         \!\!0,2,7\!\!&
         \!\!0,2,7\!\!&
         \!\!0,2,10\!\!&
         \!\!0,3,9\!\!&
         \!\!0,2,7\!\!&
         \!\!0,2,18\!\!&
         \!\!0,2,5\!\!&
         \!\!0,2,16\!\!&
         \!\!0,2,7\!\!&
         \!\!0,2,4\!\!&
         \!\!0,2,4\!\!\\
\hline

                 \scriptsize{\!\!\!\textbf{(a4)}}\!\!\!& \!\!\scriptsize{012}\!\!\!& 
                   \!\!0,2,9\!\!&
           \!\!0,2,12\!\!&
\!\!0,2,4\!\!&
           \!\!0,2,4\!\!&
             \!\!0,2,4\!\!&
         \!\!0,2,15\!\!&
         \!\!0,8,36\!\!&
         \!\!0,2,5\!\!&
         \!\!0,2,8\!\!&
         \!\!0,2,5\!\!&
         \!\!0,2,5\!\!&
         \!\!0,2,3\!\!&
         \!\!0,2,4\!\!&
         \!\!0,2,3\!\!&
         \!\!0,2,7\!\!&
         \!\!0,2,12\!\!&
         \!\!0,2,5\!\!&
         \!\!0,2,3\!\!&
         \!\!0,2,7\!\!&
         \!\!0,2,3\!\!&
         \!\!0,2,3\!\!\\
\hline

                 \scriptsize{\!\!\!\textbf{(b)}}\!\!\!& \!\!\scriptsize{012}\!\!\!& 
                   \!\!2,8,21\!\!&
           \!\!---\!\!&
\!\!---\!\!&
           \!\!2,7,8\!\!&
             \!\!2,7,8\!\!&
         \!\!3,7,8\!\!&
         \!\!2,4,7\!\!&
         \!\!2,4,7\!\!&
         \!\!---\!\!&
         \!\!\!{2,15,16}\!\!\!&
         \!\!---\!\!&
         \!\!2,3,7\!\!&
         \!\!---\!\!&
         \!\!\!{2,15,16}\!\!\!&
         \!\!---\!\!&
         \!\!---\!\!&
         \!\!\!{2,15,16}\!\!\!&
         \!\!---\!\!&
         \!\!2,3,8\!\!&
         \!\!2,7,8\!\!&
         \!\!2,7,8\!\!\\

        & \!\!\scriptsize{013}\!\!\!& 
                   \!\!\!\!&
           \!\!\!2,3,20\!\!\!\!&
\!\!\!2,3,20\!\!\!\!&
           \!\!\!\!&
             \!\!\!\!&
         \!\!\!\!&
         \!\!\!\!&
         \!\!\!\!&
         \!\!\!2,\!16,21\!\!\!\!&
         \!\!\!\!&
         \!\!2,8,15\!\!&
         \!\!\!\!&
         \!\!2,3,15\!\!&
         \!\!\!\!&
         \!\!2,3,16\!\!&
         \!\!2,3,15\!\!&
         \!\!\!\!&
         \!\!2,3,16\!\!&
         \!\!\!\!&
         \!\!\!\!&
         \!\!\!\!\\
\hline
     \end{tabular}}
     
\begin{minipage}[l]{0.49\textwidth}
 \;\tiny{
         \begin{tabular}{|c|c||c|c|c|c|c|c|}
         \hline
         \scriptsize{Cone} & \scriptsize{Rows} &          \scriptsize{\Gc{1}}&\scriptsize{\Gc{2}}&\scriptsize{\Gc{3}}&\scriptsize{\Gc{4}}&\scriptsize{\Gc{5}}&\scriptsize{\Gc{6}}
         \\\hline\hline
                  \tiny{\!\!\!\textbf{(aa2a3a4)}}\!\!\!& \!\!\scriptsize{012}\!\!\!& 
         \!\!0,1,2\!\!&
         \!\!0,2,7\!\!&
         \!\!0,2,5\!\!&
         \!\!0,2,5\!\!&
         \!\!0,1,5\!\!&
         \!\!0,2,5\!\!\\
\hline
                  \tiny{\!\!\!\textbf{(aa2a3b)}}\!\!\!& \!\!\scriptsize{012}\!\!\!& 
         \!\!0,1,2\!\!&
         \!\!0,2,7\!\!&
         \!\!0,5,7\!\!&
         \!\!0,2,5\!\!&
         \!\!0,5,10\!\!&
         \!\!0,1,5\!\!\\
\hline\hline
                  \scriptsize{\!\!\!\textbf{(aa2a3)}}\!\!\!& \!\!\scriptsize{012}\!\!\!& 
         \!\!0,1,2\!\!&
         \!\!0,2,7\!\!&
         \!\!0,5,7\!\!&
         \!\!0,2,5\!\!&
         \!\!0,12,14\!\!&
         \!\!0,2,5\!\!\\
\hline

                  \scriptsize{\!\!\!\textbf{(aa2a4)}}\!\!\!& \!\!\scriptsize{012}\!\!\!& 
         \!\!0,1,2\!\!&
         \!\!0,2,7\!\!&
         \!\!0,5,20\!\!&
         \!\!0,2,5\!\!&
         \!\!0,1,5\!\!&
         \!\!0,2,5\!\!\\
\hline

                  \scriptsize{\!\!\!\textbf{(aa3a4)}}\!\!\!& \!\!\scriptsize{012}\!\!\!& 
         \!\!0,1,2\!\!&
         \!\!0,2,7\!\!&
         \!\!0,2,5\!\!&
         \!\!0,2,5\!\!&
         \!\!0,1,5\!\!&
         \!\!0,2,5\!\!\\
\hline
                  \scriptsize{\!\!\!\textbf{(a2a3a4)}}\!\!\!& \!\!\scriptsize{012}\!\!\!& 
         \!\!0,2,7\!\!&
         \!\!0,2,7\!\!&
         \!\!0,2,5\!\!&
         \!\!0,2,5\!\!&
         \!\!0,2,5\!\!&
         \!\!0,2,5\!\!\\
\hline

                  \scriptsize{\!\!\!\textbf{(aa2b)}}\!\!\!& \!\!\scriptsize{012}\!\!\!& 
         \!\!0,1,2\!\!&
         \!\!0,2,7\!\!&
         \!\!0,5,7\!\!&
         \!\!0,2,5\!\!&
         \!\!0,5,10\!\!&
         \!\!0,1,5\!\!\\
\hline

                  \scriptsize{\!\!\!\textbf{(aa3b)}}\!\!\!& \!\!\scriptsize{012}\!\!\!& 
         \!\!0,1,2\!\!&
         \!\!0,2,7\!\!&
         \!\!0,5,7\!\!&
         \!\!0,2,5\!\!&
         \!\!0,5,10\!\!&
         \!\!0,1,5\!\!\\
\hline

                  \scriptsize{\!\!\!\textbf{(a2a3b)}}\!\!\!& \!\!\scriptsize{012}\!\!\!& 
         \!\!0,10,12\!\!&
         \!\!0,2,7\!\!&
         \!\!0,5,7\!\!&
         \!\!0,2,5\!\!&
         \!\!0,5,10\!\!&
         \!\!0,2,5\!\!\\
\hline
\hline

                 \scriptsize{\!\!\!\textbf{(aa2)}}\!\!\!& \!\!\scriptsize{012}\!\!\!& 
         \!\!0,1,2\!\!&
         \!\!0,2,7\!\!&
         \!\!0,5,21\!\!&
         \!\!0,2,5\!\!&
         \!\!0,12,14\!\!&
         \!\!0,2,5\!\!\\
\hline

                 \scriptsize{\!\!\!\textbf{(aa3)}}\!\!\!& \!\!\scriptsize{012}\!\!\!& 
         \!\!0,1,2\!\!&
         \!\!0,2,7\!\!&
         \!\!0,5,7\!\!&
         \!\!0,2,5\!\!&
         \!\!0,12,14\!\!&
         \!\!0,2,5\!\!\\
\hline

                 \scriptsize{\!\!\!\textbf{(aa4)}}\!\!\!& \!\!\scriptsize{012}\!\!\!& 
         \!\!0,2,7\!\!&
         \!\!0,2,7\!\!&
         \!\!0,5,20\!\!&
         \!\!0,2,5\!\!&
         \!\!0,2,5\!\!&
         \!\!0,2,5\!\!\\
\hline

                 \scriptsize{\!\!\!\textbf{(a2a3)}}\!\!\!& \!\!\scriptsize{012}\!\!\!& 
         \!\!0,10,22\!\!&
         \!\!0,2,7\!\!&
         \!\!0,5,7\!\!&
         \!\!0,2,5\!\!&
         \!\!0,12,14\!\!&
         \!\!0,2,5\!\!\\
\hline

                 \scriptsize{\!\!\!\textbf{(a2a4)}}\!\!\!& \!\!\scriptsize{012}\!\!\!& 
         \!\!0,2,7\!\!&
         \!\!0,2,7\!\!&
         \!\!0,5,20\!\!&
         \!\!0,2,5\!\!&
         \!\!0,2,5\!\!&
         \!\!0,2,5\!\!\\
\hline

                 \scriptsize{\!\!\!\textbf{(a3a4)}}\!\!\!& \!\!\scriptsize{012}\!\!\!& 
         \!\!0,2,7\!\!&
         \!\!0,2,7\!\!&
         \!\!0,2,5\!\!&
         \!\!0,2,5\!\!&
         \!\!0,2,5\!\!&
         \!\!0,2,5\!\!\\
\hline
\hline

                 \scriptsize{\!\!\!\textbf{(ab)}}\!\!\!& \!\!\scriptsize{012}\!\!\!& 
         \!\!0,1,2\!\!&
         \!\!0,2,7\!\!&
         \!\!1,2,7\!\!&
         \!\!1,2,6\!\!&
         \!\!1,5,7\!\!&
         \!\!0,1,7\!\!\\
\hline

                 \scriptsize{\!\!\!\textbf{(a2b)}}\!\!\!& \!\!\scriptsize{012}\!\!\!& 
         \!\!0,10,12\!\!&
         \!\!0,2,7\!\!&
         \!\!0,5,7\!\!&
         \!\!0,2,5\!\!&
         \!\!0,5,10\!\!&
         \!\!0,2,5\!\!\\
\hline

                 \scriptsize{\!\!\!\textbf{(a3b)}}\!\!\!& \!\!\scriptsize{012}\!\!\!& 
         \!\!0,10,12\!\!&
         \!\!0,2,7\!\!&
         \!\!0,5,7\!\!&
         \!\!0,2,5\!\!&
         \!\!0,5,10\!\!&
         \!\!0,2,5\!\!\\
\hline \hline
\scriptsize{\!\!\!\textbf{(a)}}\!\!\!& \!\!\scriptsize{012}\!\!\!& 
         \!\!0,1,11\!\!&
         \!\!0,2,7\!\!&
         \!\!---\!\!&
         \!\!---\!\!&
         \!\!---\!\!&
         \!\!---\!\!\\

& \!\!\scriptsize{013}\!\!\!& 
         \!\!\!\!&
         \!\!\!\!&
         \!\!---\!\!&
         \!\!---\!\!&
         \!\!---\!\!&
         \!\!---\!\!\\
& \!\!\scriptsize{rest}\!\!\!& 
         \!\!\!\!&
         \!\!\!\!&
         \!\!---\!\!&
         \!\!---\!\!&
         \!\!---\!\!&
         \!\!---\!\!\\

\hline

                 \scriptsize{\!\!\!\textbf{(a2)}}\!\!\!& \!\!\scriptsize{012}\!\!\!& 
         \!\!0,10,12\!\!&
         \!\!0,2,7\!\!&
         \!\!0,5,21\!\!&
         \!\!0,2,5\!\!&
         \!\!0,12,14\!\!&
         \!\!0,10,36\!\!\\
\hline

                 \scriptsize{\!\!\!\textbf{(a3)}}\!\!\!& \!\!\scriptsize{012}\!\!\!& 
         \!\!0,11,12\!\!&
         \!\!0,2,7\!\!&
         \!\!0,5,7\!\!&
         \!\!0,2,5\!\!&
         \!\!0,12,14\!\!&
         \!\!0,2,5\!\!\\
\hline

                 \scriptsize{\!\!\!\textbf{(a4)}}\!\!\!& \!\!\scriptsize{012}\!\!\!& 
         \!\!0,2,7\!\!&
         \!\!0,2,7\!\!&
         \!\!0,5,20\!\!&
         \!\!0,2,5\!\!&
         \!\!0,2,5\!\!&
         \!\!0,2,5\!\!\\
\hline

                 \scriptsize{\!\!\!\textbf{(b)}}\!\!\!& \!\!\scriptsize{012}\!\!\!& 
         \!\!2,3,7\!\!&
         \!\!2,4,7\!\!&
         \!\!2,4,7\!\!&
         \!\!---\!\!&
         \!\!7,12,17\!\!&
         \!\!--- \!\!\\

        & \!\!\scriptsize{013}\!\!\!& 
         \!\!\!\!&
         \!\!\!\!&
         \!\!\!\!&
         \!\!3,8,20\!\!&
         \!\!\!\!&
         \!\!2,7,20\!\!\\
\hline
 \end{tabular}}
\end{minipage}\begin{minipage}[r]{0.69\textwidth}
  \caption{Ruling out all 27 potential non-boundary tropical lines for all non-apex cones in the Naruki fan. Each entry gives the three columns of a tropically non-singular $3\times 3$-minor of the pair of matrices $(\TMexp[E,\sigma],\TMnone[E,\sigma])$ in the family $\cF_{\sigma}$ from~\eqref{eq:pairsOfMatrices} for each cone $\sigma$.\\
    An absence of a triple for an extremal ray is indicated by `---' and it should be interpreted as~\autoref{alg:minorsWithNones} failing to find a non-singular minor with the prescribed rows. Whenever a choice of three rows does not rule out all extremal curves, we move on to the next $3$-element set of rows (in the lexicographic order) and only check the remaining extremal curves. \\
    All 27 extremal curves are covered by a suitable choice of rows with the exception of the  cell $\mathbf{(a)}$, for which the method only rules out 12 potential lines.  The remaining 15 cases are treated in~\autoref{tab:coneA}.\label{tab:allCones}}
\end{minipage}
   \end{table}\clearpage
\end{landscape}}
 \normalsize

\section{Finding Cross function representatives for Eckardt triangles} 
\label{sec:find-repr-cross}

The need to employ ratios of matrix entries in \autoref{lm:noInteriorLinesACone} to rule out extra tropical lines on the $\operatorname{(a)}$ cone is rooted in the redundancy of Cross functions for tropicalization purposes. Indeed, by~\autoref{rm:triangleQuintics} each Eckardt quintic is associated to a triple of Cross functions. \autoref{pr:ratiosOfCross} shows that the valuation of the ratio of any two of them  is a linear function on the valuations of all 40 Yoshida functions. Thus, the valuation of all 135 Cross functions can be explicitly computed from the Naruki fan and the valuation of suitable 45 Cross functions, each associated to a distinct Eckardt quintic. Our aim in this section is to find these 45 representative Cross functions. This choice will be crucial to determine the metric structure on the 27 boundary trees on $\Trop(X_L)$ and prove~\autoref{thm:Naruki}.

Following~\autoref{rm:markings}, we let $\ted$ be the collection of symbols encoding the 45 anticanonical triangles on $X$. The same set is used to index the 45 Eckardt quintics.
\begin{proposition}\label{pr:ratiosOfCross} For each symbol in $\ted$, consider the triple  $\{\Cross{i_1}, \Cross{i_2},\Cross{i_3}\}$ of Cross functions associated to the corresponding Eckardt quintic. Then, the  ratios  $(\Cross{i_j}/\Cross{i_k})^3$ can be expressed as Laurent monomials in Yoshidas functions. In particular, their valuations are linear function on the Naruki fan.

\end{proposition}
\begin{proof}
 Since the action of $\Wgp$ is transitive on all 135 Cross functions, it suffices to prove the result for the triple associated to a single Eckardt quintic. We choose that of~\autoref{ex:cross}, corresponding to the symbol $y_{152346}$ in $\ted$. In the notation of~\autoref{tab:crosses}, the triple becomes $\{\Cross{30}, \Cross{31}, \Cross{32}\}$.

 As in the proof of~\autoref{lm:noInteriorLinesACone}, we know that the exponent vectors of the 16 Yoshida functions $\{\Yos{5}, \Yos{17},\ldots, \Yos{31}\}$ span the index-three sublattice of $\Z^{36}$ associated to the rows of the Yoshida matrix.  A direct computation available in the Supplementary material reveals that
 \begin{equation}\label{eq:crRatios}
   \Big(\frac{\Cross{31}}{\Cross{30}}\Big)^3\!\!= \frac{\Yos{17}^6\,\Yos{18}^2\,\Yos{21}^4\,\Yos{23}^3\,\Yos{25}^3\,\Yos{26}\,\Yos{27}\,\Yos{30}^5}{\Yos{5}\,\Yos{19}^3\,\Yos{20}^2\,\Yos{22}^7\,\Yos{24}^3\,\Yos{28}^6\,\Yos{31}^3}
\; \text{ and}\; \Big(\frac{\Cross{32}}{\Cross{30}}\Big)^3 \!\!=  \frac{\Yos{17}^3\,\Yos{18}\,\Yos{21}^2\,\Yos{27}^2\,\Yos{30}^4}{\Yos{5}^2\,\Yos{20}\,\Yos{22}^2\,\Yos{26}\,\Yos{28}^3\,\Yos{31}^3}.\qedhere
 \end{equation}
   \end{proof}
In the remainder of this section we explain how to choose the 45 representing Cross functions, one per triple. Our criteria is based on the difference between the expected and actual valuations of Cross functions on the baricenter of each of the 24 orbit representatives of Naruki cones. Precise formulas for the expected valuation of a given Cross function in the relative interior of each cone were provided in~\autoref{rm:expValsCross}. The valuation of a given $\Cross{i}$, will only be undertermined  when each of the four expressions $\Cross{i} = \pm(\Yos{k} - \Yos{j})$ arising from~\autoref{rm:a14tocross} has $\val(\Yos{k}) = \val(\Yos{j})$.

~\autoref{pr:ratiosOfCross} can be used to give a better formula for the expected valuation of all 135 Cross functions based on a single representing member of each triple of Cross functions. Our next objective is to determine whether these new expected valuations are achieved generically by finding  suitable $d_1,\ldots, d_6$ in $\KK$ in the fibers of the Yoshida map.

\smallskip

Before doing this, we fix some standard notation. We let $\Gamma:=\val(\KK^*)$ be the value group of $\KK$. Our assumptions on the valued field $\KK$ ensure that $\Gamma$ is divisible and dense in $\R$. Furthermore, the valuation admits a splitting $\gamma\mapsto t^{\gamma}$, where we set $t^0 := 1$. After appropriate rescaling if necessary, we may assume $1\in \Gamma$, and thus $\Q\subset \Gamma$.

Given  a smooth cubic surface $X$ over $\KK$ with no Eckardt points, we let $\Yos{}$ be the point in $\P^{39}_{\KK}$ recording the value of the 40 Yoshida functions on $X$. We set $\underline{p} := \trop(\Yos{}) \in \Naruki$. The following definition will be relevant:
\begin{definition}\label{def:surfaceTypes}
  If the point $\underline{p}$ lies  in the relative interior of an $\operatorname{(aa_2a_3a_4)}$ cone in the Naruki fan, we say $\Trop X \subset \TP^{44}$ is an $\operatorname{(aa_2a_3a_4)}$ surface. We define $\operatorname{(aa_2a_3b)}$ surfaces analogously.
\end{definition}
\smallskip

The next two lemmas provide  uniform formulas to produce parameters $d_1,\ldots, d_6$ on the cones $\operatorname{(aa_2a_3a_4)}$ and $\operatorname{(aa_2a_3b)}$ from~\eqref{eq:coneReps} and all their faces. Both choices match on their overlap.
In what follows, we let $r_1,r_2,r_3, r_4$ be the scalars used to write points in the first cone, and $r_1,r_2,r_3, r'_4$ be those used for the second cone. For each $I\subset\{1,\ldots, 3\}$ we let $r_I := \sum_{i\in I} r_i$.

\begin{lemma}\label{lm:genExamplesaaaa} The following choice of parameters $d_1,\ldots, d_6\in \KK^*$ produces an $\operatorname{(aa_2a_3a_4)}$ tropical cubic surfaces in $\TPr^{44}$ associated to scalars  $r_1,\ldots,r_4 \in \Gamma_{\geq 0}$:
  \begin{equation*}\label{eq:aaaa_ds}
    \begin{minipage}[l]{0.5\linewidth}
    \[        \begin{aligned}
      d_1 &= (-{u_2}\,t^{r_{123} + 2\,{r_4}} + {u_1}\,t^{r_{23} +2\, r_4} \;- {u_3})/3,\\
    d_2 &=  (2\,{u_2}\,t^{r_{123} + 2\,{r_4}} + {u_1}\,t^{r_{23} +2\,r_4} \;- {u_3})/{3},\\
    d_3&= (-{u_2}\,t^{r_{123} + 2\,{r_4}} + {u_1}\,t^{r_{23} +2\,r_4} +2 {u_3})/{3},
    \end{aligned}
  \]
    \end{minipage}
    \begin{minipage}[l]{0.4\linewidth}
    \[\begin{aligned}
      d_4 &= (u_4 + u_5\,t^{r_3}) + d_1,\\
    d_5 &= (u_4 + u_6\,t^{r_3} + u_7\,t^{2\,r_{123}+2\, r_{4}}) +d_1,\\
    d_6&=  u_8\,t^{r_4} +d_1,
    \end{aligned}
    \]
    \end{minipage}
      \end{equation*}
    where $u_1,\ldots, u_8\in \KK^*$ satisfy $\val(u_i) = 0$ for all $i$ and the residue classes are generic with respect to the linear conditions~\eqref{eq:gen_aaaa}. If $r_1,\ldots,r_4>0$, these reduce to  $\val(u_5-u_6) = \val(u_4-u_3)=\val(2u_4-u_3) = 0$.
  \end{lemma}

\begin{lemma}\label{lm:genExamplesaaab}
  The following choice of parameters $d_1,\ldots, d_6\in \KK^*$ produces an $\operatorname{(aa_2a_3b)}$ tropical cubic surfaces in $\TPr^{44}$ associated to scalars   $r_1,\ldots,r_4' \in \Gamma_{\geq 0}$:
\begin{equation*}\label{eq:aaab_ds}
  \begin{minipage}[l]{0.5\linewidth}
    \[\begin{aligned}
      d_1 &= (-{u_2}\,t^{r_{123} + {r_4'}} + {u_1}\,t^{r_{23} }\; - {u_3}\, t^{r_4'})/3,\\
    d_2 &=  (2{u_2}\,t^{r_{123} + {r_4'}}\; + {u_1}\,t^{r_{23} } \;- {u_3}\, t^{r_4'})/{3},\\
    d_3&= (-{u_2}\,t^{r_{123} + {r_4'}} + {u_1}\,t^{r_{23} } +2 {u_3} t^{r_4'})/{3},
    \end{aligned}
    \]
    \end{minipage}
    \begin{minipage}[l]{0.4\linewidth}
    \[\begin{aligned}
      d_4 &= (u_4 + u_5\,t^{r_3}) + d_1,\\
    d_5 &= (u_4 + u_6\,t^{r_3} + u_7\,t^{2\,r_{123}+2 r_{4}'}) +d_1,\\
    d_6&=  u_8 +d_1,
    \end{aligned}
    \]
    \end{minipage}
\end{equation*}
    where $u_1,\ldots, u_8\in \KK^*$ have $\val(u_i) = 0$ for all $i$ and are generic relative to the constraints in~\eqref{eq:gen_aaaa} and \eqref{eq:gen_aaab}. If $r_1,\ldots,r_4'>0$, we require  $\val(u_5\!-\!u_6)\!=\!\val(u_8\!-\!u_4)\!=\!\val(u_8\!+\!u_4)\!=\!\val(2u_4\!+\!u_8)\!=\!0$.
  \end{lemma}

\begin{proof}[Proof of~\autoref{lm:genExamplesaaaa}]
    It suffices to show that our choice of parameters yields positive roots in $\Wgp$ with valuations given by the point $v := r_1 v_a + r_2 v_{a2} + r_3 v_{a3} + r_4 v_{a4} \in \R^{36}_{\geq 0}$, where $v_{a},\ldots, v_{a4}$ are as in~\autoref{rm:BergmanReps}. If all scalars are positive, $v$ has precisely seven non-zero coordinates, i.e., those associated to the roots $d_2-d_1$, $d_{5}-d_4$, $d_6-d_2$, $d_6-d_1$, $d_{1}+d_{2}+d_3$ and $d_{2} + d_{3}+d_{6}$. The remaining 29 roots must have valuation 0. A direct computation on the expected initial forms of all roots shows that their valuations are given by $v$ if and only if the three $\Z$-linear expessions in the statement have valuation 0.

  Allowing some of the scalars $r_i$ to vanish yields a precise genericity condition. Namely, each of the following 58 linear expressions in the $u_i$'s  must  have valuation zero:
  \begin{equation}\label{eq:gen_aaaa}
    \begin{aligned}
      &u_5\!-\!u_6, u_4\!+\!u_5, u_4\!+\!u_6,u_3\!-\!u_2, u_4\!-\!u_3,
      u_8\!-\!u_2, u_8\!-\!u_3, u_8\!-\!u_4, u_8\!+\!u_1,  u_4\!+\!u_8, \\
      & 2u_4\!-\!u_3,  u_8\!+\!u_1\!-\!u_2,   u_4\!+\!u_6\!+\!u_7, u_4\!+\!u_5\!-\!u_2, u_4\!+\!u_6\!+\!u_7\!-\!u_2,  u_4\!-\!u_3\!+\!u_5, \\ & u_4\!-\!u_3\!+\!u_6, u_4\!+\!u_6\!+\!u_8,  
      u_4\!+\!u_5\!-\!u_8, u_4\!+\!u_6\!-\!u_8, u_4\!+\!u_5\!+\!u_8,       u_4\!+\!u_8\!-\!u_3, u_5\!-\!u_6\!-\!u_7, \\
      &  u_1\!+\!u_4\!+\!u_5, 2u_4\!+\!u_5\!+\!u_6,  u_1\!+\!u_4\!+\!u_6, u_1\!+\!u_8\!-\!u_3, 
      u_4\!+\!u_6\!+\!u_7\!-\!u_8, 2u_4\!+\!u_8\!-\!u_3,  \\
      & u_4\!+\!u_6\!+\!u_7\!-\!u_3,   u_4\!+\!u_5\!+\!u_1\!-\!u_3, u_4\!+\!u_6\!+\!u_1\!-\!u_3, u_4\!+\!u_8\!+\!u_5\!-\!u_3, u_4\!+\!u_8\!+\!u_6\!-\!u_3,\\
      & u_1 \!+\!u_4\!+\!u_5\!+\!u_8,
       2u_4\!+\!u_5\!+\!u_6\!-\!u_3,
       u_4\!+\!u_5\!+\!u_1\!-\!u_2, u_4\!+\!u_6\!+\!u_1\!+\!u_7,u_4\!+\!u_8\!+\!u_6\!+\!u_1,\\
       &   u_1\!+\!u_4\!+\!u_6\!+\!u_7\!-\!u_2,        u_4\!+\!u_6\!+\!u_1\!+\!u_7\!-\!u_3, u_4\!+\!u_8\!+\!u_5\!+\!u_1\!-\!u_3, 2u_4\!+\!u_1\!+\!u_5\!+\!u_6\!-\!u_3,\\
       &    u_4\!+\!u_5\!+\!u_8\!+\!u_1\!-\!u_2, u_4\!+\!u_8\!+\!u_6\!+\!u_1\!-\!u_3, 2u_4\!+\!u_5\!+\!u_6\!+\!u_1,u_1\!+\!u_4\!+\!u_6\!+\!u_7\!+\!u_8\!-\!u_2,
       \\ & 2u_4\!+\!u_5\!+\!u_6\!+\!u_1\!+\!u_7\!-\!u_3\!-\!u_2, 
       u_4\!+\!u_8\!+\!u_5\!+\!u_1\!-\!u_2\!-\!u_3,  2u_4\!+\!u_5\!+\!u_6\!+\!u_8\!-\!u_3, \\
       & u_4\!+\!u_8\!+\!u_6\!+\!u_1\!+\!u_7\!-\!u_2\!-\!u_3,   2u_4\!+\!u_5\!+\!u_6\!+\!u_1\!+\!u_7\!-\!u_3,u_1\!+\!u_4\!+\!u_6\!+\!u_7\!+\!u_8\!-\!u_3,\\
       &  2u_4\!+\!u_5\!+\!u_6\!+\!u_7\!+\!u_1\!-\!u_2,    2u_4\!+\!u_5\!+\!u_6\!+\!u_8\!+\!u_1\!-\!u_3,  2u_4\!+\!2u_1\!+\!u_5\!+\!u_6\!+\!u_8\!-\!u_3,\\
       & 2u_4\!+\!u_5\!+\!u_6\!+\!u_8\!+\!u_1\!+\!u_7\!-\!u_2\!-\!u_3,         2u_4\!+\!2u_1\!+\!u_5\!+\!u_6\!+\!u_8\!+\!u_7\!-\!u_2\!-\!u_3. \qquad
       \qquad \qquad \quad \qedhere
    \end{aligned}
        \end{equation}
  \end{proof}

  \begin{proof}[Proof of~\autoref{lm:genExamplesaaab}]
    The proof follows the same strategy as for~\autoref{lm:genExamplesaaaa}. When all scalars are positive, we have exactly five roots with non-zero valuation: $d_2-d_1, d_3-d_1, d_3-d_2, d_5-d_4, d_1+d_2+d_3$. In addition to the 58 linear expressions in~\eqref{eq:gen_aaaa} having valuation zero, the genericity conditions  require four additional linear expressions to have valuation zero, namely:
    \begin{equation}\label{eq:gen_aaab}
      2u_4+u_5+u_6+u_8, 2u_4+u_1+u_5+u_6+u_8, 2u_4+u_8, 2u_1+2u_4 +u_5+u_6+u_8.\qedhere
      \end{equation}
    \end{proof}
\begin{remark}\label{rem:PuiseuxChoice} For $\KK=\PS$ the genericity condition for both maximal cone representatives will be satisfied for $u_1=6$, $u_2=3$, $u_3 = -9$, $u_4=11$, $u_5=1$, $u_6= u_8 = -57$, and $u_7=57$.
\end{remark}

Our next result says that with only one exception, each triple of Cross functions has a member whose expected valuation is achieved generically for all 23 non-apex cone representatives of $\Naruki$. \autoref{pr:newExpValCr15} provides a conceptual explanation behind our failed search for the triple $\{\Cross{15}, \Cross{16}, \Cross{17}\}$ associated to the symbol $x_{53}$ in $\ted$.

\begin{proposition}\label{pr:relevantCrossFunctions}
  Each triple of Cross functions associated to any anticanonical triangle other than  $x_{53}$ contains a member whose expected valuation is achieved generically along the 23 non-apex Naruki cone representatives. The collection of 44 representatives equals
  \begin{equation*}
    \cR_0\!:=\!\{\Cross{3i}\colon
    i\in \{0,\ldots, 45\}\!\smallsetminus\!\{5,9,12,13,24, 25, 32, 36\}\} \cup \{ \Cross{k}\colon k=28, 37, 41, 73, 76, 97, 110\}.
        \end{equation*}
  \end{proposition}

\begin{proof} The result follows from~\textcolor{blue}{Lemmas}~\ref{lm:genExamplesaaaa} and~\ref{lm:genExamplesaaab}. Indeed, given a cone $\sigma$ among the 23 non-apex representatives we pick generic parameters $d_1,\ldots, d_6 \in \KK$ associated to its baricenter (i.e., where $r_i=0$ or $1$ depending on the nature of $\sigma$) using either~\eqref{eq:gen_aaaa} or~\eqref{eq:gen_aaab}. In turn, we evaluate each of the 135 Cross functions expressed in the $d_i$'s and compute their valuations, comparing them to the expected ones. Gaps arise for only 20 such functions. We record these functions and the  Naruki cone representatives where discrepancies arise, and find a member on each of the 44 triples exhibiting no gaps. They are listed  in $\cR_0$.  The only exception is the triple corresponding to the symbol $x_{53}$, namely $\{\Cross{15}, \Cross{16}, \Cross{17}\}$. The explicit computations for $\KK=\PS$ are available in the Supplementary material.
\end{proof}

 Setting aside the special triple of Cross functions, the difference between the expected and generic valuations for the remaining Cross functions outside $\cR_0$ is rooted in the formulas from \autoref{pr:ratiosOfCross}. Their expected valuations should not be computed from~\eqref{eq:expval}, but rather as the sum the expected valuation of is representative in $\cR_0$ and a prescribed linear function on $\Naruki$ with coefficients in $\frac{1}{3}\Z$ arising from~\eqref{eq:crRatios} after acting on it by $\Wgp$. For example, consider
  \begin{equation}\label{eq:Cr36vs37}
    \frac{\Cross{36}^3}{\Cross{37}^3} = \frac{\Yos{17}^6\,\Yos{18}^4\,\Yos{26}^2\,\Yos{27}^2\,\Yos{30}^4}{\Yos{5}^2\,\Yos{19}^3\,\Yos{20}\,\Yos{21}\,\Yos{22}^2\,\Yos{28}^3\,\Yos{29}^3\,\Yos{31}^3}.
  \end{equation}
On $\sigma = \operatorname{(aa_2a_3a_4)}$ we have $\expVal(\Cross{36}) = \val(\Yos{38})$ and $\expVal(\Cross{37}) = \val(\Yos{34})$.
  Since $\Cross{37}$ lies in $\cR_0$, its expected valuation is achieved generically. However, a direct computation reveals that $3(\val(\Yos{38}) - \val(\Yos{34}))$ is strictly less than  the valuation of the right-hand side of~\eqref{eq:Cr36vs37} on $\sigma^\circ$. This means that the expected valuation of $\Cross{36}$ is never achieved on $\sigma^{\circ}$ and we should always replace any appearance of $\val(\Cross{36})$ by the sum of $\val(\Cross{37})$ and the linear function on $\Naruki$ defined by the exponent vector on~\eqref{eq:Cr36vs37}, scaled by $1/3$. The precise formulas are available in the Supplementary material. This viewpoint will play a crucial role in~\textcolor{blue}{Sections}~\ref{sec:tropical-lines-trivial} and~\ref{sec:comb-types-tree}.

  \begin{remark}\label{rm:fullListReps}
    Our construction does not lead to a preferred choice among  $\{\Cross{15},\Cross{16},\Cross{17}\}$. We make an arbitrary decision and set our  45 representing  Cross functions as
    \begin{equation}\label{eq:45relevant}
      \cR :=\cR_0 \cup \{\Cross{15}\},
    \end{equation}
    where $\cR_0$ is as in~\autoref{pr:relevantCrossFunctions}. By construction, each function in $\cR$ corresponds to a unique anticanonical triangle indexed by a symbol in $\ted$. These 45 functions  will allow us to distinguish between stable and unstable tropical cubic surfaces. We return to this subtle point in~\autoref{sec:comb-types}.
  \end{remark}
  
    We use the set $\cR$ to construct a point $\underline{q}^{\cR} \in \R^{45}/\rspanone$ as follows.  Each anticanonical triangle $t\in \ted$ correspons to a unique Cross function in $\cR$, which we label $\Cross{t}$.  We set
    \begin{equation}\label{eq:translationVector}
    \underline{q}^{\cR}_t = -\val(\Cross{t}) \qquad  \text{ for each } T\in \ted.
\end{equation}
  This point will play a prominent role in the construction of all boundary trees on stable tropical cubic surfaces (see~\autoref{lm:decomposing10x10matrices}.)

\section{Extra tropical lines on tropical cubic del Pezzos for the apex of the fan $\Naruki$}\label{sec:tropical-lines-trivial}

\autoref{sec:comb-extra-trop} describes the combinatorics of tropical lines on anticanonical cubic surfaces $\Trop X$ meeting the interior of $\TPr^{44}$. In particular,~\autoref{thm:candidateTropLines} shows that any such line has five boundary points, realized as the tropicalization of the five nodes in the link of a vertex of the Schl\"afli graph. When the point recording the valuation of all 40 Yoshida functions lies outside the apex of the Naruki fan, these potential extra tropical lines are ruled out by a simple convexity argument. In this section we discuss the apex point, where the methods from~\autoref{sec:trop-lines-trop} fail. Knowledge about the valuations of the 45  Cross functions from the set $\cR$ in~\eqref{eq:45relevant} will allow us to bypass this issue and prove~\autoref{thm:extraLinesApex}.

Throughout this section we fix a cubic surface $X\subset \P^{44}$ whose associated Yoshida functions in $\P^{39}$ have the same valuation. Without loss of generality, we assume it to be zero. Since by~\autoref{lm:CoordinatesOfNodes} the coordinates of all 135 classical nodes are Laurent monomials in Yoshida and Cross functions, this assumuption will greatly simplify the computations in this section. Indeed, the coordinates of all 135 tropical nodes will be linear in the valuations of all functions in $\cR$.

Following~\autoref{rm:markings},  we use the collection $\sed$ of 27 symbols associated to the exceptional curves on $X$ to index all  potential extra tropical lines. Our first result determines when the five boundary points of a potential line $\Trop \ell_E$ for $E$ in $\sed$ are tropically collinear.

The quintuple of boundary points of each potential line $\Trop \ell_E$ (for $E$ in $\sed$) gives five disjoint sets $B_1,\ldots, B_5$ associated to the nine $\infty$ coordinates of each boundary point. These sets determine all five rays $e_{B_1}, \ldots, e_{B_5}$ in the recession fan of $\Trop \ell_E$.

The following is the first main result of this section. It gives a precise statement for the first half of \autoref{thm:extraLinesApex}.

\begin{theorem}\label{thm:extraLineCross} Given a symbol $E$ in $\sed$, consider the tuple of five nodes in $\Trop X$ obtained from the edges in the link of $E$ in the Sch\"afli graph. We let  $\{\Cross{i_1}, \ldots, \Cross{i_5}\}$ be  the set of all Cross functions in $\cR$ associated to the five anticanonical triangles in $\ted$ containing $E$.
  Then, the five points  are tropically collinear in $\TPr^{44}$ if and only if the five functions $\Cross{i_j}$ for $j=1,\ldots, 5$ listed above have equal valuation. If so, the unique tropical line through them is a star tree with five leaves.
  \end{theorem}

\begin{proof} Since $\Wgp$ acts transitively on both $\sed$ and  the Schl\"afli graph, it suffices to prove the result for $s=\E{1}$. Following~\autoref{tab:crosses}, the five Cross functions from the statement become $\Cross{9}$, $\Cross{33}$, $\Cross{42}$, $\Cross{78}$ and $\Cross{111}$. Their associated anticanonical triangles are  $x_{12}$, $x_{16}$, $x_{14}$, $x_{13}$ and $x_{15}$, respectively. Notice that these are also the triangles in $\ted$ containing $\E{1}$.

  By~\autoref{thm:candidateTropLines}, the five boundary points in the potential tropical line $\Trop \ell_{\E{1}}$  are the tropicalization of the nodes $\F{1j}\cap \Gc{j}$ for $j=2,\ldots, 6$.  By~\autoref{cor:BoundaryPoints}, each of these nodes has exacly nine $\infty$ coordinates. We record them by the sets $B_j=\{X_{ij} \colon i\neq j\} \cup \{Y_{1jklmn}\colon k,l,m,n\}$ ($j=2,\ldots, 6$) which partition the set of 45 symbols $T$.

The computations in the Supplementary material reveal that up to translation by the point  $\underline{q}^{\cR} \in \R^{45}/\rspanone$ from~\eqref{eq:translationVector}, and rearranging columns, the five tropical nodes listed above correspond to the rows of the matrix
  \begin{equation*}\label{eq:linkOfE1}
   M\!\!:=\! \left(\begin{array}{ccccc|ccccc}
  \!\val(\Cross{42})\! & \!\val(\Cross{42})\!
  & \!\infty \!&\! \val(\Cross{42}) \!&\! \val(\Cross{42}) \!&\!    \infty \!&\! 0 \!&\! 0 \!&\! 0 \!&\! 0\!\\
\!\val(\Cross{78})\! &\! \infty\! & \!\val(\Cross{78})  \!&\! \val(\Cross{78}) \!&\! \val(\Cross{78}) \!&\!    0 \!&\!  \infty \!&\! 0 \!&\! \vdots \!&\! \vdots\! \\
\!\infty\! & \!\val(\Cross{9})\! & \! \val(\Cross{9})  \!&\! \val(\Cross{9}) \!&\! \val(\Cross{9})  \!&\!    \vdots  \!&\! 0 \!&\!  \infty  \!&\! 0 \!&\! \vdots\!  \\  
\!\val(\Cross{33})\! & \! \val(\Cross{33})\!  & \!\val(\Cross{33}) \!&\! \!\val(\Cross{33})\! \!&\! \infty  \!&\!    \vdots  \!&\! \vdots  \!&\! 0 \!&\!  \infty  \!&\! 0\! \\
\!\val(\Cross{111})\! & \! \val(\Cross{111})\!  &\! \val(\Cross{111}) \!&\! \infty \!&\! \val(\Cross{111})   \!&\!    0  \!&\! 0  \!&\! 0 \!&\! 0 \!&\!  \infty   
        \end{array}
    \right),
  \end{equation*}
where each entry to the right of the vertical divide encodes a $1\times 8$ block matrix. The first five column of $M$ correspond to the anticanonical triangles $x_{12}, \ldots, x_{16}$. 

    By~\autoref{pr:tropicalLinesbyMinors}, the five rows of $M$ are tropically collinear if and only if all its $3\times 3$ minors are tropically collinear. In particular, the following four minors
    \[\left(
    \begin{array}{ccc}
      \val(\Cross{42}) \!&\! \infty \!&\! 0 \\
      \val(\Cross{78}) \!&\! 0 \!&\! \infty  \\
      \infty \!&\! 0 \!&\! 0
    \end{array}
    \right) \quad \text{ and } \quad 
   \left(
    \begin{array}{ccc}
      \val(\Cross{42})  \!&\! \infty \!&\! 0\\
      \infty \!&\! 0 \!&\! 0\\
      \val(\Cross{i}) \!&\! 0 \!&\! \infty 
    \end{array}
    \right) \text{ for } i = 9, 33, 111
.        \]
    arising by combining the first or second column of $M$ with a pair of suitable $0/\infty$-columns from the right hand side of $M$ must be tropically singular. This will be the case if and only if
    \[
    \val(\Cross{9}) = \val(\Cross{33}) = \val(\Cross{42}) = \val(\Cross{78}) =\val(\Cross{111}).    
    \]

    Conversely, if these equalities hold, the matrix $M$ can be shifted to a matrix with rows $e_{B_2}, \ldots, e_{B_6}$. Therefore, the tropical line $\Trop \ell_{\E{1}}$ is a translation of a fan with rays $e_{B_j}$  for $j=2,\ldots, 6$. 
\end{proof}

It is important to stress out that the proof of~\autoref{thm:extraLineCross} only determines the collinearity of the five boundary points of $\Trop \ell_{E}$. It does not address the question of whether $\Trop \ell_{E}$ lies in $\Trop X$. Indeed, unless $\Trop X$ is a stable tropical cubic surface, there is no guarantee that $\Trop \ell_{E}\subset \Trop X$.  We conclude:
\begin{corollary}\label{cor:27bound}
  Any tropical cubic surface $X$ with Yoshida functions associated to the apex of $\Naruki$ contains at most 27 extra tropical lines in its interior, all of which are star trees with five rays each. This bound is attained for the unique stable surface $\Trop X$ corresponding to the apex of $\Naruki$.
\end{corollary}
\begin{proof}
  Assume the quintuple points associated to a given symbol $E$ in $\sed$ are tropically collinear. The proof of~\autoref{thm:extraLineCross} shows that up to translation, these boundary points are vectors with nine $\infty$ coordinates each. All other entries are 0.
As we will see in~\autoref{sec:comb-types}, if all 45 elements in $\cR$ have valuation 0, $\Trop X$ becomes the unique stable tropical cubic surface associated to the apex of $\Naruki$. Then, $\Trop X$ is the cone over the Schl\"afli graph and  $\Trop \ell_{E}\subset\Trop X$. The bound is attained.
\end{proof}

\begin{remark}\label{rm:CrossfunctionsE1} It is worth pointing out that the five Cross functions associated to a symbol $E$ in $\sed$ play a key role in the metric structure of the boundary tree $\Trop E$ (see~\autoref{fig:nonGenApexTrees}.) Furthermore, the 40 finite coordinates of the central vertex of $\Trop E$ match those in the unique vertex of $\Trop \ell_E$. The remaining five coordinates of this central vertex equal 0. This observation is consistent with the expected combinatorial structure of unstable tropical cubic surfaces and would help us determine whether $\Trop \ell_E$ lies in $\Trop X$ or not. We postpone providing an answer to this question for future work.
\end{remark}

We end this section by turning to the question of lifting combinations of cycles supported on non-boundary tropical line in  $\Trop X$ to effective curves in $X\subset \P^{44}$, thus addressing the second half of~\autoref{thm:extraLinesApex}. Our next two results discuss two particular instances: a tropical cycle supported on a single extra tropical line, or a combination of two such lines. The monomial map~\eqref{eq:alphaBar} will yield the same answer for tropicalizations of $X$ induced by Cox embeddings (see~\autoref{thm:universalCoxL}.)

\begin{theorem}\label{thm:noCyclesLift} Assume that the tropical line $\Trop \ell_E$ associated to a symbol $E$ in $\sed$ exists and lies in the interior of $\Trop X$. Then, no tropical cycle supported on $\Trop \ell_E$ lifts to an effective curve on $X$.
\end{theorem}
\begin{proof} The group $\Wgp$ acts transitively on all 27 symbols in $\sed$ and the quintuple of Cross functions determining $\Trop \ell_E$. Thus, it suffices to show this result for a particular choice of $E$, e.g. $\E{1}$.  The leaves and rays of $\Trop \ell_{\E{1}}$ are provided in the proof of \autoref{thm:extraLineCross}.

A tropical cycle supported on  $\Trop \ell_{\E{1}}$ has an integer  multiplicity on each ray and it satisfies the balancing condition at the origin: the five primitive vectors for each ray scaled by their multiplicity should add up to a multiple of the all-ones vector. The disjoint support property the the five rays forces these multiplicities to agree. Thus, we write the cylce as  $m\cdot \Trop \ell_{\E{1}}$ for some integer $m$. Note that the tropical cycle will be effective whenever $m\geq 1$.  We claim any such cycle (effective or not) cannot be lifted to an effective curve $\Ccurve$ on $X$.  We argue by contradiction. 

By construction, all boundary points on the curve $\Ccurve$ tropicalize to one of the five boundary points on $\Trop \ell_{\E{1}}$. In particular, $\Ccurve$ contains five boundary points $p_2, \ldots, p_6$.~\textcolor{blue}{Lemmas}~\ref{lm:boundaryTX} and~\ref{lm:IntersectionBoundaryLines} imply that $p_j=\F{1j}\cap \Gc{j}$ for each $j$. Furthermore, the intersections on the tropical side yield
  \begin{equation}\label{eq:CboundaryPoints}
\Ccurve \cap \Gc{1}=\emptyset\;,\qquad \Ccurve \cap \E{k} =\emptyset \;\text{ for all }k\neq 1\;\quad \text{ and }\quad \Ccurve \cap \F{ik} = \emptyset \;\text{ for all }1<i<k.
  \end{equation}

  Since the 27 exceptional curves $\E{1},\ldots, \Gc{6}$ generate the effective cone,  we  write the class of $\Ccurve$  as
\[ 
[\Ccurve] = \sum_{i} a_i [\E{i}] + \sum_{i<j} b_{ij} [\F{ij}] + \sum_{i} c_i
[\Gc{i}] \quad \text{ for some } a_i,b_{ij}, c_i\geq 0.
\]
Conditions~\eqref{eq:CboundaryPoints} translate to the following system of 16 linear equations: 
 \[ \left\{
  \begin{aligned}
       0&= [\Ccurve] \cdot [\Gc{1}]  = \sum_{i\neq 1} a_i + \sum_{j> 1} b_{1j} - c_1 ,\\
0 &=      [\Ccurve]\cdot [\E{k}]  = -a_k + \sum_{i< k} b_{ik} + \sum_{i> k} b_{ki}+ \sum_{i\neq k} c_i   \quad \text{ for }k > 1,\\
           0 &= [\Ccurve]\cdot [\F{ik}]  = a_i + a_k -b_{ik} + \!\!\!\!\!\! \sum_{\{p,q\}\cap \{i,k\}=\emptyset}\!\!\!\!\! b_{pq} + c_{i} + c_k  \quad \text{ for }1<i<k .
    \end{aligned}\right.
  \]
A simple calculation with \sage, available in the Supplementary Material, confirms that the system has rank 6 but has no solutions in the positive orthant other than the trivial one. Therefore, no multiple of the tropical line $\Trop \ell_{\E{1}}$  lifts to an effective curve in $X$ as we wanted to show.
\end{proof}

\begin{theorem}\label{thm:noPairsLift}No pair $\Trop \ell_E\cup \Trop \ell_{E'}$ of extra tropical lines on $\Trop X$ associated to symbols $E,E'\in \sed$ lifts to a conic on $X$.
\end{theorem}
\begin{proof} By B\'ezout's Theorem, every conic curve $\Ccurve$ on $X$ is residual: the $\P^2$ containing it is the span of three non-collinear points on $\Ccurve$. By design, each boundary point of $\Trop \ell_E$ lifts to a unique point in the surface $X$, namely one of the five nodes associated to the link of $E$ in the Schl\"afli graph.

  Without loss of generality, we  assume $E = \E{1}$. We claim the five classical nodes in the link of $\E{1}$ are not coplanar. Instead, they span the same $\P^3$ as $X$. To prove this, we construct a $4\times45$ matrix using the nodes $\F{1i}\cap\Gc{i}$ for $i=2,4,5,6$ and show that the $4\times 4$ minor with columns  labeled by the anticanonical triangles $y_{152436}, y_{152634}, y_{162435}$ and $y_{162534}$ in $\ted$ is the following Laurent monomial in Yoshida and Cross functions
  \[ \frac{2\,\Cross{35}\,\Yos{0}\,\Yos{5}^4\,\Yos{7}^2\,\Yos{9}\,\Yos{11}\,\Yos{15}^2}{
   \Cross{39}\Cross{40}\Cross{41}\Cross{54}\Cross{55} \Cross{56}\Cross{61}\Cross{71}\Yos{3}^2\,\Yos{4}\,\Yos{10}\,\Yos{12}\,\Yos{13}^2\,\Yos{17}
  }.
  \]
  Since $X$ is smooth without Eckardt points and it spans a $\P^3$, the minor is non-singular and therefore the matrix has rank four.
  \end{proof}
\begin{remark}\label{rm:notropicalProof} It is worth noticing that the previous result cannot be proven purely by tropical means. After translations, the $10\times 45$ tropical matrices obtained from the ten boundary points of $\Trop \ell_E\cup \Trop \ell_{E'}$ involve the valuations of the five Cross functions associated to $E$ and $E'$. These sets overlap if $E$ and $E'$ intersect in $X$. In that case, in order to lift the union of these two tropical lines to a conic in $X$, these nine Cross functions must have equal valuation. Similarly, if $E$ and $E'$ do not intersect, the two sets of Cross functions do not overlap, and a $4\times 4$-minor computation forces all ten relevant Cross functions to have valuation zero if this minor were to be tropically singular.

  In both situations, all $4\times 4$-minors become tropically singular and the $10\times 45$ matrix  admits one non-singular $3\times 3$-minor (see Supplementary material.) Thus, the matrix has tropical rank exactly three. Since the Kapranov rank (i.e.\ one plus the projective dimension of the smallest tropical linear space containing all ten boundary points) may exceed the tropical rank~\cite[Theorem 1.4]{dev.san.stu:05}, we cannot conclude that $\Trop \ell_E\cup \Trop \ell_{E'}$ spans a $3$-dimensional tropical linear space in $\TPr^{44}$.
\end{remark}

The proof of~\autoref{thm:extraLinesApex} follows by combining~\textcolor{blue}{Theorems}~\ref{thm:extraLineCross}, \ref{thm:noCyclesLift} and \ref{rm:notropicalProof}.

\section{Boundary trees on anticanonical stable tropical cubic surfaces}\label{sec:comb-types-tree}

The boundary of a smooth anticanonical tropical cubic surface in $\TPr^{44}$ with no Eckardt points is an arrangement of 27 metric trees meeting along  135 leaves. As will be discussed in~\autoref{thm:BoundaryDeterminesEverything}, this arrangement and its metric determines the combinatorics of the tropical surface. 
Each tree is labeled by its classical counterpart and has ten leaves, one for every other line meeting the label.

In this section we discuss the combinatorics and metric structure on these trees and prove the last third of~\autoref{thm:Naruki}.  More precisely, we show that the lengths of the edges of each tree are piecewise linear functions on the Naruki fan. Our main contribution is~\autoref{tab:treeLabeling}, which contains the leaf labeling and  formulas for all  edge lengths on the two $\Wgp$-orbit representatives of all maximal Naruki cones. Continuity determines the  data for all  lower-dimensional cones. The first two parts of \autoref{thm:Naruki} are addressed in~\autoref{sec:trop-moduli-space}.

Throughout, we let $X$ be a smooth cubic surface over $\KK$ without Eckardt points. We let $\Yos{}\in \P^{39}_{\KK}$ encode its associated Yoshida functions and fix $\underline{p} := \trop(\Yos{}) \in \Naruki$. We aim
to determine the boundary trees on $\Trop X$ and their metric structure in terms of $\underline{p}$.
  Our choice of  24 representatives of  cones in $\Naruki$ is given by the cones
$\operatorname{(aa_2a_3a_4)}$ and $\operatorname{(aa_2a_3b)}$  from~\eqref{eq:coneReps} together with their faces.  Following the notation from~\autoref{tab:treeLabeling}, we let $r_1,\ldots, r_4 $ (respectively, $r'_4$) be the non-negative scalars (with values in $\Gamma$) associated to the point $\underline{p}$ whenever it lies in the first (respectively, second) cone. As in~\autoref{lm:genExamplesaaaa}, we let  $r_I := \sum_{i\in I} r_i$ for each $I\subset\{1,\ldots, 4\}$.

  \subsection{Boundary trees for top-dimensional cells} Recall from \autoref{def:surfaceTypes}, that there are two types of tropical cubic surfaces associated to  maximal cells in $\Naruki$. Our first result discusses the combinatorics and boundary metric structure in these two fundamental cases.

  \begin{theorem}\label{thm:treeTypes} There are two combinatorial types of boundary metric trees on an $\operatorname{(aa_2a_3a_4)}$  surface and three combinatorial types on an $\operatorname{(aa_2a_3b)}$ surface, as seen in Figures~\ref{fig:treesaa2a3a4} and~\ref{fig:treesaa2a3b}. They come in many symmetry classes. The leaf labeling  for each tree is completely determined by its behavior on two $\Wgp$-orbit representatives of maximal cones in the Naruki fan (see ~\autoref{tab:treeLabeling}.)
  \end{theorem}

The proof of this statement is computational in nature. \autoref{alg:treesFromLeaves} plays a central role. All required scripts, input and output files are available in the Supplementary material. The remainder of this section describe the steps involved in the process. Rather than restricting our attention to the maximal cells, we describe the general procedure to treat all 24 cone representatives at once. The outcome for lower-dimensional cones is discussed in~\autoref{sec:boundary-trees-lower}.

As a starting point, we provide the framework for carrying out all tropical convex hull computations. In order to use~\autoref{alg:treesFromLeaves} we must determine the ten leaves of each boundary tree in $\TPr^{44}$ and choose a suitable projection to $\TPr^{9}$ to see each of them as a generic tropical line.

To achieve the first task, we turn to~\autoref{lm:IntersectionBoundaryLines}. Each of the 135 leaves on our collection of 27 boundary trees on $\Trop X$ is the tropicalization of a unique classical node on the surface $X$.  Explicit formulas for  135 classical nodes on  $X$ involving Laurent monomials in the 40 Yoshidas and the 135 Cross functions were given in~\autoref{lm:CoordinatesOfNodes}. We use~\autoref{pr:ratiosOfCross} to write the tropicalization of these monomials as linear functions in the valuations of all Yoshidas but only the 45 relevant Cross functions from the set $\cR$ given in~\eqref{eq:45relevant}.

The main obstacle to construct the 135 tropical nodes was faced already in~\autoref{sec:trop-lines-trop}. Indeed, depending on the nature of $\underline{p}$, we might not be able to determine all 45 coordinates of each tropical node. Note that the point $\underline{p}$ in $\TPr^{39}$ will only prescribe the \emph{expected} valuations of the 45 Cross functions from $\cR$, as seen in \autoref{rm:expValsCross}. Their actual valuations will depend on the values of the Yoshida functions encoded in $\Yos{}$. In particular, they can vary among the different surfaces $X$ associated to $\underline{p} \in \Naruki$. \autoref{thm:aaaaIsStable} will show that this does not happen if $\underline{p}$ lies in  the relative interior of a maximal cone of $\Naruki$.

  \begin{figure}[htb]
  \includegraphics[scale=0.5]{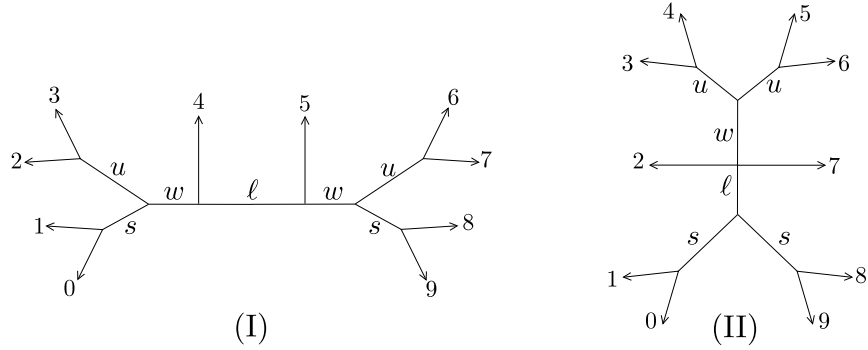}
  \caption{The two combinatorial types of trees for the Naruki cone \textbf{(aa2a3a4)}. The four edge lengths and the labeling of all ten leaves are provided in Table~\ref{tab:treeLabeling}. The type (I) metric tree comes in (4,4,4,4,4,4) symmetry classes, whereas the type (II) one has (1,1,1) symmetry classes.\label{fig:treesaa2a3a4}}
\end{figure}

\begin{figure}[htb]
    \includegraphics[scale=0.5]{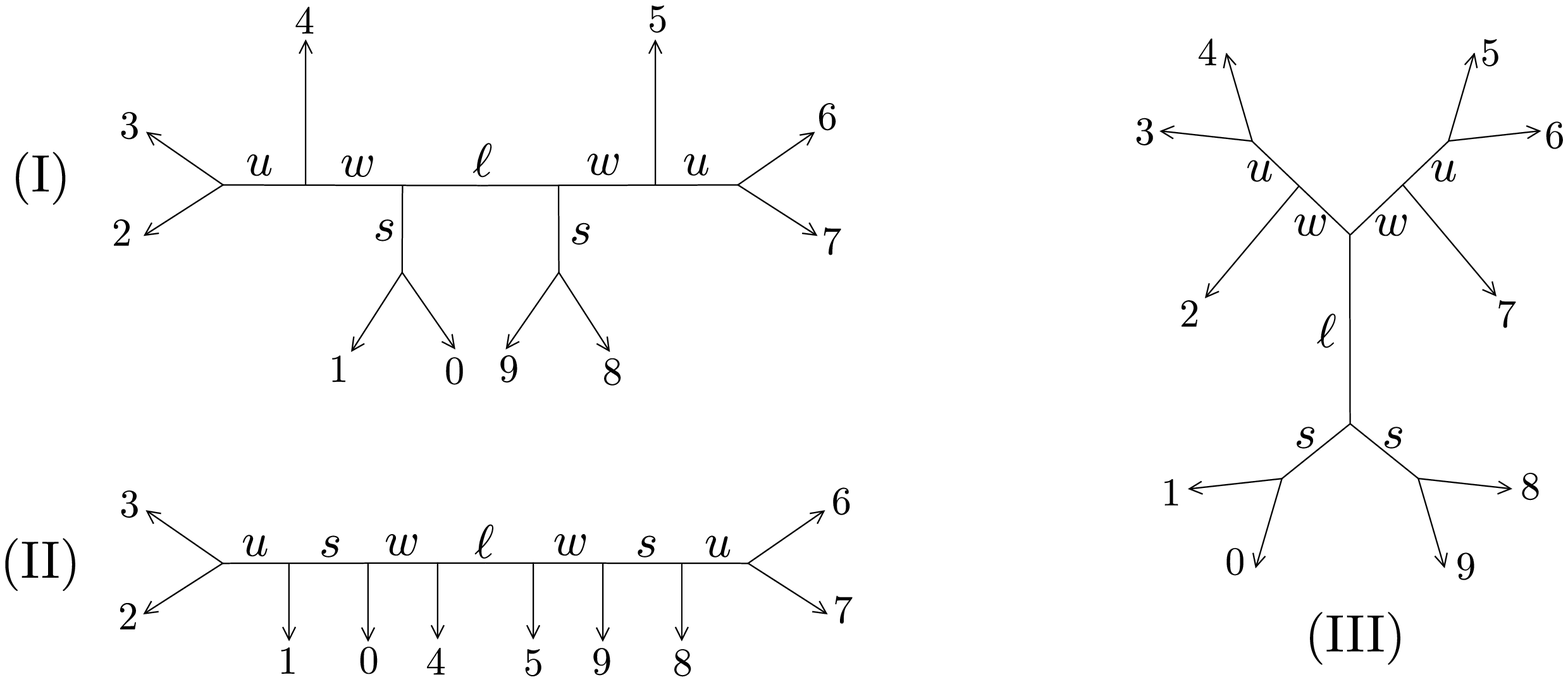}
  \caption{The three combinatorial types of trees for the Naruki cone \textbf{(aa2a3b)}. The four edge lengths and the labeling of all ten leaves are given in Table~\ref{tab:treeLabeling}. The type (I) metric tree comes in (2,2,2,2,2,2) symmetry classes,  the type (II) has (4,4,4) symmetry classes, whereas the type (III) has (1,1,1) symmetry classes.\label{fig:treesaa2a3b}}
\end{figure}

\begin{remark}\label{rm:byproduct}
  As we saw in~\autoref{pr:relevantCrossFunctions},  the triple of Cross functions associated to $x_{53}$ exhibit positive gaps between the expected an true valuations on certain cones in $\Naruki$ even subject to genericity conditions. This suggests that the formula proposed in \eqref{eq:expval} for the expected valuations on this triple is too na\"ive and needs to be adjusted when considering these troubling cones. The correct formula emerges by tropical convexity considerations (see~\autoref{pr:newExpValCr15}.)
\end{remark}

By construction, each boundary tree is indexed by some exceptional curve $E$ in the set $\sed$ from~\autoref{rm:markings}. Our next task is to determine  the appropriate projection
\begin{equation}\label{eq:pi_E}
  \pi_E\colon \TPr^{44} \dashrightarrow \TPr^{9}
\end{equation}
that turns $\Trop E$ into a generic tree in $\TPr^{9}$. ~\autoref{lm:recoverLineFromProjection} allows us to recover $\Trop E \subset \TPr^{44}$ from this generic tree and the ten leaves in $\TPr^{44}$, independently of the chosen projection.

By~\autoref{cor:BoundaryPoints}, $\Trop E$ lies in the boundary strata with values $\infty$ along the five coordinates indexed by anticanonical triangles in $\ted$ containing $E$. Furthermore, the remaning 40 are partititioned into ten blocks of four coordinates each. Each block will give the $\infty$ coordinates on a fixed leaf of $\Trop E$. If the leaf is indexed by a symbol $E'$ in $\sed$, these four coordinates are precisely the anticanonical triangles in $\ted$ containing $E'$ but not $E$. The map $\pi_E$ from~\eqref{eq:pi_E} is determined by choosing the lexicographically minimal coordinate on each of the ten blocks. By construction, the $10\times 10$ matrix recording the leaves of $\Trop E$ has $\infty$-valued entries only along its diagonal.

\begin{example}\label{ex:coordsE1}
  When $E=\E{1}$, the exceptional curves meeting $E$  are given by  the set
  \[
  \mathcal{L}:=\{\Gc{3},\Gc{6},\F{16},\F{12}, \F{13},\Gc{4},\F{15}, \Gc{5},\Gc{2},\F{14}\}.\]
  This induces the choice of (ordered) coordinates for $\pi_E$
\begin{equation}\label{eq:ECoords}
  \{x_{23},x_{26},x_{61},x_{21},x_{31},x_{24}, x_{51},   x_{25}, x_{32},  x_{41} \}.
\end{equation}
The set $\mathcal{L}$ will determine how to label a given leaf: simply choose the exceptional curve in $\mathcal{L}$ associated to the unique $\infty$-coordinate on the leaf in $\TPr^9$.
\end{example}

To avoid dealing with undefined entries on the tropical nodes, we express each entry of the $10\times 10$ matrix of leaves for each boundary tree as a linear function of the valuation of the 40 Yoshidas and the 45 relevant Cross valuations.  The following result will allow us to read off this information in a way that is best suited for determining the variation of a fixed boundary tree as we move along $\Naruki$.

\begin{lemma}\label{lm:decomposing10x10matrices} Each $10\times 10$ matrix recording the ten leaves of the tropical line $\Trop E$ in $\TPr^{9}$ associated to a symbol in $\sed$ can be decomposed as a usual sum of three matrices $M^s$, $M^{Cr}$ and $M^{\Yos{}}$, where:
  \begin{enumerate}[(i)]
  \item $M^s$ is a $10\times 10$ matrix obtained by multiplying the column vector $\mathbf{1}$ of size $10\times 1$ by the projection of the row vector $\underline{q}^{\cR}$ from~\eqref{eq:translationVector} to $\TPr^{9}$.  We refer to it as the shifting matrix.
    
  \item $M^{Cr}$ is a symmetric matrix with only ten non-zero entries. Each pair of symmetric entries is associated to the valuation of a Cross function in $\cR$ indexed by one of the five anticanonical triangles containing $E$. For each triangle $t$, the rows involving $\Cross{t}$ are indexed by the two additional exceptional curves belonging to $t$. We call $M^{Cr}$ the matrix of Cross components.
    \item  $M^{\Yos{}}$ is a matrix with $\infty$ entries on its diagonal. The rest of the entries are linear functions in the valuations of all 40 Yoshida functions. We call it the matrix of Yoshida components.
  \end{enumerate}
  \end{lemma}
\begin{proof}   To take advantage of the $\Wgp$-action, we work with the $10\times 45$  matrix $M$ of leaves of $\Trop E$ in $\TPr^{44}$. We claim that $M$ admits a decomposition as $M^s + M^{Cr} + M^{\Yos{}}$ where each summand has the analogous properties as their $10\times 10$ counterparts (ignoring the symmetry property for $M^{Cr}$.) If so, the projection $\pi_E$ applied to each of the matrices will yield the result for the $10\times 10$ matrix.

  Our prior discussion in this section shows that each finite entry of $M$ is a linear function of the valuations of $\Yos{}$ and the Cross functions in $\cR$. We let $M^{\Yos{}}$ be the matrix collecting the $\infty$-entries in $M$ (nine per row) and the summands on each of the remaning entries involving Yoshida valuations. To finish, we must decompose $M':=M- M^{\Yos{}}$ as $M^s + M^{Cr}$. It suffices to do so for a single $E$.

  We define $M^s$ as the multiplication of the column vector $\mathbf{1}$ with 45 rows by the row vector $\underline{q}^{\cR}$ from~\eqref{eq:translationVector}. Explicit computations available in the Supplementary material shows that $M^{Cr}:=M'-M^s$ is a sparse  matrix with 40 non-zero entries. Each of its rows is indexed by a symbol $E'$ in $\sed$ meeting $E$. We let $t$ be the unique anticanonical triangle in $\ted$ containing both $E$ and $E'$. We let $E''$ be the third exceptional curve in $t$. The four anticanonical triangles in $\ted$  containing $E''$ but not $E$ will index the columns in $M^{Cr}$ with non-zero values along the row indexed by $E'$. Moreover, the value on these four entries is $\val(\Cross{t})$.

  For example, when $E=\E{1}$ and $E'=\Gc{3}$, it follows that $t =x_{13}$, $\Cross{t} = \Cross{78}$, $E''=\F{13}$ and the four triangles in $\ted$ are $x_{31}, y_{132456}, y_{132546}$ and $y_{132645}$. A direct computation verifies the validing of the statement for the remaining four pairs of lines meeting $\E{1}$.
  After projection to $\TPr^9$, the matrix $M^{Cr}$ becomes symmetric. This concludes our proof.
\end{proof}

\begin{example}\label{ex:EMaxConeCross}   We describe the data from~\autoref{lm:decomposing10x10matrices} for $E = \E{1}$. The coordinates of $\TPr^9$ and the associated set $\mathcal{L}$ of exceptional curves is the one from~\autoref{ex:coordsE1}.
  
  First, the matrix $M^s$ is determined by the row vector $\underline{q} = -(\val(\Cross{i}))_{i\in I} \in \TPr^9$, where $I$ is the ordered set
  $I=\{105, 63, 45, 114, 0, 21, 24, 93, 84, 18\}$.
  This ordering agrees with the assignment of Cross functions to the set of anticanonical triangles in~\eqref{eq:ECoords}.
  Second, the matrix $M^{Cr}$ involves the five  functions $\Cross{i}$ for $i= 9, 33, 42, 78, 111$ associated to the ordered set of anticanonical triangles $\{x_{12},x_{16}, x_{14}, x_{13}, x_{15}\}$ containing $\E{1}$. The labeling of the coordinates in $\TPr^9$ and the set $\mathcal{L}$ imply the following assignment for each symmetric pair of coordinates on $M^{Cr}$:

  \begin{minipage}[l]{0.32\linewidth}
    \[\begin{aligned}
            M^{Cr}_{1,5} & = M^{Cr}_{5,1} = \val(\Cross{78}), \\
            M^{Cr}_{2,3} & = M^{Cr}_{3,2} = \val(\Cross{33}),
    \end{aligned}
    \]
    \end{minipage}
     \begin{minipage}[l]{0.32\linewidth}
    \[\begin{aligned}
      M^{Cr}_{4,9} & = M^{Cr}_{9,4} = \val(\Cross{9}), \\
      M^{Cr}_{6,10} & = M^{Cr}_{10,6} = \val(\Cross{42}),
    \end{aligned}\]
         \end{minipage}
   \begin{minipage}[l]{0.3\linewidth}
    \[\begin{aligned}
    M^{Cr}_{7,8} & = M^{Cr}_{8,7} = \val(\Cross{111}).\\
    & 
  \end{aligned}
    \]   \end{minipage}

   \noindent
   The first pair corresponds to the curves $\Gc{3}$ and $\F{13}$ mentioned in the proof of~\autoref{lm:decomposing10x10matrices}.
\end{example}

\begin{proof}[Proof of~\autoref{thm:treeTypes}] The proof is computational. In what follows we describe the main steps required to determine the input data for~\autoref{alg:treesFromLeaves} for each $E\in \sed$  when $\underline{p}$ is in the relative interior of a maximal cone $\sigma$ from~\eqref{eq:coneReps}. The choice of coordinates for $\TPr^9$ is fixed by the projection $\pi_E$ from~\eqref{eq:pi_E}. The output will be the tree $\Trop E\subset \TPr^9$ recorded as a list of pairs. Each pair records a vertex in $\TPr^9$  and the primitive direction of its   adjacent edge  pointing towards the trees' center.

  The computation of tropical convex hulls commutes with translations. Thus, rather than working with the $10\times 10$ matrix $M$ of leaves for each tropical line $\Trop E$, we replace it with the sum $M^{Cr} + M^{\Yos{}}$ of matrices from~\autoref{lm:decomposing10x10matrices}. By~\autoref{thm:aaaaIsStable} we know that when restricted to the relative interior of $\sigma$, the valuations of all Cross functions in $\cR$ agree with their expected values. By~\eqref{eq:expval}, the latter equals the valuation of a suitable Yoshida function. We conclude that on $\sigma^{\circ}$, each finite entry of $M^{Cr} + M^{\Yos{}}$ is a linear function on $\sigma$. Therefore, we can express it as a $\Q$-linear function in the scalars $r_1,\ldots, r_4$, and $r'_4$. We view each non-diagonal entry in the field $\Q(r_1,\ldots, r_4,r_4')$.

  We use the above matrix as the input for~\autoref{alg:treesFromLeaves} (our implementation allows values in any field.) The output confirms the leaf labelling and combinatorial types of each tree obtained by combining~\autoref{tab:treeLabeling} with Figures~\ref{fig:treesaa2a3a4} and \ref{fig:treesaa2a3b}. The metric formulas  follow from~\autoref{lm:treemetric}. They are responsible for the symmetric classes discussed in the statement.
\end{proof}

\afterpage{
  \clearpage
\thispagestyle{empty}
  \begin{landscape}
\begin{table}[htb]
    \begin{tabular}{|c||c|c|c|c|c|c|c|c|c|c||c|c|c|c|c||c|c|c|c|c|}     \hline\footnotesize{Curve}\normalsize& 0 & 1 & 2 & 3 & 4 & 5 & 6 & 7 & 8 & 9 &\!\footnotesize{\textbf{($\operatorname{\mathbf{aa_2a_3a_4}}$)}}\normalsize\!& $u$ & $s$ & $w$ & $l$ &
     \footnotesize{\textbf{($\operatorname{\mathbf{aa_2a_3b}}$)}} \normalsize
& $u$ & $s$ & $w$ & $l$ 
      \\\hline  \hline
    $E_1$ & $F_{12}$ & $G_6$ & $G_5$ & $G_4$ & $G_3$ & $F_{13}$ & $F_{14}$ & $F_{15}$ & $F_{16}$ & $G_2$ & (I) & $r_{34}$ & $r_4$ & $r_{234}$ & $r_1$ & (II) & $r_3$ & $r_4'$ & $r_{23}$ & $r_1$ \\\hline
    $E_2$  & $F_{12}$ & $G_6$ & $G_5$ & $G_4$ & $G_3$ & $F_{23}$ & $F_{24}$ & $F_{25}$ & $F_{26}$ & $G_1$ & (I) & --- & --- & --- & --- & (II) & --- & --- & --- & ---  \\\hline
    $F_{13}$ & $G_1$ & $F_{26}$ & $F_{25}$ & $F_{24}$ & $E_1$ & $G_3$ & $F_{56}$ & $F_{46}$ & $F_{45}$ & $E_3$  & (I) & --- & --- & --- & --- & (II) & --- & --- & --- & --- \\\hline
    $F_{23}$ & $E_3$ & $F_{45}$ & $F_{56}$ & $F_{46}$ & $G_3$ & $E_2$ & $F_{15}$ & $F_{14}$ & $F_{16}$ & $G_2$ & (I) & --- & --- & --- & --- & (II) & --- & --- & --- & --- \\ \hline \hline
    $E_3$ &$F_{23}$ & $F_{13}$ & $G_5$ & $G_4$ & $G_6$ & $F_{36}$ & $F_{34}$ & $F_{35}$ & $G_1$ & $G_2$ & (I) &  $r_{34}$ &$r_{1234}$ & $r_4$ & $r_{23}$ & (I)  & $r_{3}$& $r_{123}$ & $r_4'$ & $r_{23}$ \\\hline
    $F_{12}$  & $E_1$ & $E_2$ & $F_{46}$ & $F_{56}$ & $F_{45}$ & $F_{36}$ & $F_{34}$ & $F_{35}$ & $G_1$ & $G_2$ & (I) & --- & --- & --- & --- & (I) & --- & --- & --- & --- \\\hline \hline
    $F_{45}$ & $E_5$ & $E_4$ & $F_{26}$ & $F_{16}$ & $F_{36}$ & $F_{12}$ & $F_{23}$ & $F_{13}$ & $G_5$ & $G_4$ & (I) & $r_{1234}$& $r_{34}$   & $r_4$ & $r_{23}$ & (I) & $r_{123}$ & $r_3$ & $r_4'$ & $r_{23}$ \\\hline
    $G_{6}$  & $E_5$ & $E_4$ & $F_{26}$ & $F_{16}$ & $F_{36}$ & $E_3$ & $E_1$ & $E_2$ & $F_{46}$ & $F_{56}$ & (I) & --- & --- & --- & --- & (I) & --- & --- & --- & --- \\\hline \hline
    $E_4$ & $F_{46}$ & $G_5$ & $F_{24}$ & $F_{14}$ & $F_{34}$ & $G_3$ & $G_1$ & $G_2$ &$F_{45}$ & $G_6$ & (I) & $r_{1234}$ & $r_{234}$ & $r_4$ & $r_{3}$ & (I) & $r_{123}$ & $r_{23}$ & $r_4'$ & $r_3$ \\\hline
    $E_{5}$  & $F_{56}$ & $G_4$ & $F_{25}$ & $F_{15}$ & $F_{35}$ & $G_3$ & $G_1$ & $G_2$ & $F_{45}$ & $G_6$ & (I) & --- & --- & --- & --- & (I) & --- & --- & --- & ---\\\hline \hline
    $F_{34}$ & $F_{25}$ & $F_{15}$ & $F_{56}$ & $G_4$ & $E_4$ & $G_3$ & $E_3$ & $F_{12}$ & $F_{26}$ & $F_{16}$ & (I) & $r_{234}$ & $r_{1234}$ & $r_4$ & $r_3$ & (I) & $r_{23}$ & $r_{123}$ & $r_4'$ & $r_3$\\\hline
    $F_{35}$  & $F_{24}$ & $F_{14}$ & $F_{46}$ & $G_5$ & $E_5$ & $G_3$ & $E_3$ & $F_{12}$ & $F_{26}$ & $F_{16}$ & (I) & --- & --- & --- & --- & (I) & --- & --- & --- & ---\\\hline \hline
    $F_{14}$ & $F_{35}$ & $E_4$ & $F_{56}$ & $G_4$ & $F_{36}$ & $F_{25}$ & $E_1$ & $F_{23}$ & $G_1$ & $F_{26}$ & (I) & $r_{234}$ & $r_4$ & $r_{34}$ & $r_{12}$ & (II) & $r_{23}$ & $r_4'$ & $r_3$ & $r_{12}$ \\\hline
    $F_{15}$  & $F_{34}$ & $E_5$ & $F_{46}$ & $G_5$ & $F_{36}$ & $F_{24}$ & $E_1$ & $F_{23}$ & $G_1$ & $F_{26}$ & (I) & --- & --- & --- & --- & (II) & --- & --- & --- & ---\\\hline
    $F_{24}$ & $F_{35}$ & $E_4$ & $F_{56}$ &$G_4$ & $F_{36}$ & $F_{15}$ & $E_2$ & $F_{13}$ & $G_2$ & $F_{16}$ &  (I) & --- & --- & --- & --- & (II) & --- & --- & --- & ---\\\hline
    $F_{25}$ & $F_{34}$ & $E_5$  & $F_{46}$ & $G_5$ & $F_{36}$ & $F_{14}$ & $E_2$ & $F_{13}$ & $G_2$ & $F_{16}$ &  (I) & --- & --- & --- & --- & (II) & --- & --- & --- & ---\\ \hline \hline
    $F_{16}$ & $F_{35}$ & $F_{34}$ & $F_{45}$ & $G_6$ & $E_6$ & $G_1$ & $E_1$ & $F_{23}$ & $F_{25}$ & $F_{24}$ & (I) &  $r_{234}$ & $r_{34}$ & $r_{4}$ & $r_{123}$ &  (I) & $r_{23}$ & $r_3$ & $r_4'$ & $r_{123}$ \\\hline
    $F_{26}$  & $F_{35}$ & $F_{34}$ & $F_{45}$ & $G_6$ & $E_6$ & $G_2$ & $E_2$ & $F_{13}$ & $F_{15}$ & $F_{14}$ &   (I) & --- & --- & --- & --- & (I) & --- & --- & --- & --- \\\hline \hline
    $G_1$ & $E_3$ & $F_{12}$ & $E_5$ & $E_4$ & $E_6$ & $F_{16}$ & $F_{14}$ & $F_{15}$ & $E_2$ & $F_{13}$ & (I)  & $r_{34}$ & $r_{234}$ & $r_4$ & $r_{123}$ & (I)& $r_3$ & $r_{23}$ & $r_4'$ & $r_{123}$ \\\hline
    $G_2$  & $E_3$ & $F_{12}$ & $E_5$ & $E_4$ & $E_6$ & $F_{26}$ & $F_{24}$ & $F_{25}$ & $E_1$ & $F_{23}$ &   (I) & --- & --- & --- & --- & (I) & --- & --- & --- & --- \\\hline \hline
    $F_{46}$ & $E_4$ & $F_{35}$ & $F_{25}$ & $F_{15}$ & $E_6$ & $G_4$ & $F_{23}$ & $F_{13}$ & $F_{12}$ & $G_6$ &  (I) & $r_{1234}$ & $r_4$ & $r_{34}$ & $r_2$ & (II) & $r_{123}$ & $r_4'$ & $r_3$ & $r_2$  \\\hline
    $F_{56}$  & $E_5$ & $F_{34}$ & $F_{24}$ & $F_{14}$ & $E_6$ & $G_5$ & $F_{23}$ & $F_{13}$ & $F_{12}$ & $G_6$ & (I) & --- & --- & --- & --- & (II) & --- & --- & --- & ---  \\\hline
    $G_{4}$ & $E_5$ & $F_{34}$ & $F_{24}$ & $F_{14}$ & $E_6$ & $F_{46}$ & $E_1$ & $E_2$ & $E_3$ & $F_{45}$ &(I) & --- & --- & --- & --- & (II) & --- & --- & --- & ---  \\\hline
    $G_{5}$ & $E_4$ & $F_{35}$ & $F_{25}$ & $F_{15}$ & $E_6$ & $F_{56}$ & $E_1$ & $E_2$ & $E_3$ & $F_{45}$ & (I) & --- & --- & --- & --- & (II) & --- & --- & --- & ---  \\ \hline \hline
    $E_6$ & $F_{46}$ & $G_5$ & $F_{36}$ & $F_{26}$ & $F_{16}$ & $G_1$ & $G_2$ & $G_3$ & $F_{56}$ & $G_4$ & (II) & $r_{123}$ & $r_2$ & $r_4$ & $r_{34}$ & (III) & $r_{123}$ & $r_2$ & $r_4'$ & $r_3$ \\ \hline \hline
    $F_{36}$ & $F_{25}$ & $F_{15}$ & $E_6$ & $F_{45}$ & $G_6$ & $E_3$ & $F_{12}$ & $G_3$ & $F_{24}$ & $F_{14}$ & (II) & $r_{23}$ & $r_{12}$ & $r_4$ &  $r_{34}$ &(III) & $r_{23}$ & $r_{12}$ & $r_4'$ & $r_3$  \\ \hline \hline
    $G_3$ & $F_{23}$ & $F_{13}$ & $E_6$ & $E_5$ & $E_4$ & $F_{34}$ & $F_{35}$ & $F_{36}$ & $E_1$ & $E_2$ & (II) & $r_3$ & $r_1$ & $r_4$ & $r_{234}$ & (III) & $r_3$ & $r_1$ & $r_4'$ & $r_{23}$ \\ \hline 
  \end{tabular}
  \caption{Labeling, combinatorial type and metric structure of all  27 trees for adjacent orbit-representatives of the two maximal Naruki cones \textbf{($\mathbf{aa_2a_3a_4}$)} and \textbf{($\mathbf{aa_2a_3b}$)}, and their faces. The edge lengths are sums of subsets of the scalars $r_1,r_2, r_3, r_4$  and $r_1,r_2, r_3, r_4'$, where $r_1$ is the scalar for the ray $a$, $r_2$, $r_3$ and $r_4$ are the scalar for the rays $a_2$, $a_3$ and $a_4$, and $r_4'$ is the scalar for the ray $b$. For example, $r_{34}:=r_3+r_4$ and $r_{234}:=r_2+r_3+r_4$. This table together with \textcolor{blue}{Figures}~\ref{fig:treesaa2a3a4} and~\ref{fig:treesaa2a3b} describe all 27 metric trees on each cone. The action of $\Wgp$ recovers the arrangements on the remaining cones.\label{tab:treeLabeling}}
\end{table}\clearpage
  \end{landscape}}
 \normalsize

  Next, we illustrate one  instance of \autoref{thm:treeTypes} building on~\autoref{ex:EMaxConeCross}.    In the next subsection, we analyze the behavior of the metric trees for lower dimensional cones.

\begin{example}\label{ex:EMaxCone} We fix $E=E_1$ and $\sigma = \operatorname{(aa_2a_3a_4)}$.  We write down the $10\times 10$ matrix of shifted leaves of $\Trop E_1$ associated to the baricenter $\underline{p}$ of $\sigma$. We let $J$ be the ordered set labeling its columns.
 
  \[ [M^{Cr} + M^{\Yos{}}]|_{\underline{p}}:=
  \begin{blockarray}{rrrrrrrrrr}
    x_{23}& x_{26}& x_{61}& x_{21}& x_{31}& x_{24}& x_{51}& x_{25}& x_{32}& x_{41} \\
  \begin{block}{(rrrrrrrrrr)}
    \infty&  15&  14&  15&  12&  14&  9&  14&  14&  9\; \\
12&  \infty&  7&  12&  5&  10&  2&  10&  7&  2\; \\
15&  11&  \infty&  11&  9&  10&  9&  10&  15&  9\; \\
15&  15&  10&  \infty&  8&  13&  5&  13&  10&  5\; \\
19&  15&  15&  15&  \infty&  14&  10&  14&  15&  10\; \\
4&  3&  -1&  3&  -3&  \infty&  -6&  4&  -1&  -6\; \\
8&  4&  7&  4&  2&  3&  \infty&  3&  7&  4\; \\
12&  11&  7&  11&  5&  12&  2&  \infty&  7&  2\; \\
18&  14&  18&  14&  12&  13&  12&  13&  \infty&  12\; \\
6&  2&  5&  2&  0&  1&  2&  1&  5&  \infty\; \\
  \end{block}
  \end{blockarray}\;.
  \]

  The set $J$ completely determines the labeling on the rows. From top to bottom, we see the leaves associated to  $\mathcal{L}$ from~\autoref{ex:coordsE1}. This data gives pairs $([i], \operatorname{row}_i)$ for $i=0,\ldots, 9$. The rest of the tree $\Trop \E{1}$ is recorded by the coordinates of all vertices together with  lists $I$ encoding the  direction $-\sum_{i\in I}e_i$ of the inward edge adjacent to each vertex.
  
      {
      \begin{minipage}[l]{0.46\linewidth}
\[
\begin{aligned}
& ([1,3],[12, 12, 7, 12, 5, 10, 2, 10, 7, 2]), \\
& ([5,7],[4, 3, -1, 3, -3, 4, -6, 4, -1, -6]), \\
& ([1,3,5,7],[12, 11, 7, 11, 5, 10, 2, 10, 7, 2]), \\
& ([0,1,3,5,7],[19, 15, 14, 15, 12, 14, 9, 14, 14, 9]), \\
    \end{aligned}
      \]\end{minipage}
      \begin{minipage}[l]{0.48\linewidth}
        \[\begin{aligned}
& ([2,8],[15, 11, 15, 11, 9, 10, 9, 10, 15, 9]), \\
& ([6,9],[8, 4, 7, 4, 2, 3, 4, 3, 7, 4]), \\
& ([2,6,8,9],[15, 11, 14, 11, 9, 10, 9, 10, 14, 9]),\\
& ([2,4,6,8,9],[19, 15, 15, 15, 13, 14, 10, 14, 15, 10]).\!\!\!\!\!
    \end{aligned}
    \]
      \end{minipage}
      }\normalsize

      \noindent
  This data yields the type (I) tree from~\autoref{fig:treesaa2a3a4}, with the labeling given by $\mathcal{L}$ and  $r_i=1$ for all~$i$.
  \end{example}

The proof of~\autoref{thm:treeTypes} has the  following important consequence:
  \begin{corollary}\label{cor:PLstructure}
The metric structure on each tree of $\Trop X$ is linear when restricted to the relative interior of each maximal Naruki cone.
  \end{corollary}

  \noindent
  The analogous statement for tropicalizations of cubic surfaces induced by Cox embeddings can be found in \cite[Remark 3.5]{ren.sha.stu:16}.

\subsection{Boundary trees for lower dimensional cones}
\label{sec:boundary-trees-lower} The difficulty in extending the proof method of~\autoref{thm:treeTypes} to lower-dimensional cones in $\Naruki$ relies on our ability to determine the value of the Cross component matrices $M^{Cr}$ on these cones. For stable tropical cubics this step requires the knowledge of the generic expected valuations of the 45 relevant Cross functions on them, as we discussed in~\autoref{sec:find-repr-cross}. 
Determining formulas for these expected valuation on the whole Naruki fan is a subtle task. As we show later in this section, cancellations are guaranteed to occur in many cases, thus forcing a jump in the valuation. Besides certifying the remaining cases of~\autoref{thm:Naruki}, a second main result in this subsection is the characterization of the expected valuations of all Cross functions on the 24 orbit representatives of Naruki cones.

By~\textcolor{blue}{Propositions}~\ref{pr:ratiosOfCross} and~\ref{pr:relevantCrossFunctions}, it is enough to determine the expected valuation of a single Cross function associated to $x_{53}$. Following~\autoref{rm:fullListReps} we choose $\Cross{15}$.
Our next result describes its expected valuations on the faces of the two cones from~\eqref{eq:coneReps}.

\begin{proposition}\label{pr:newExpValCr15}
  The na\"ive value $\val(\Yos{32})$ for the expected valuation  of  $\Cross{15}$ is never achieved  along  the relative interior of the Naruki cone representatives $\operatorname({aa_2a_4)}$, $\operatorname{(aa_4)}$, $\operatorname{(a_2a_4)}$ and $\operatorname{(a_4)}$. The correct expected value along these four open cones is provided by $\val(\Yos{32}) + \varepsilon$, where 
  \begin{equation}\label{eq:varepsilonGap}
  \varepsilon =\frac{1}{3}\val\Big (
  \frac{\Yos{5}\,\Yos{18}\,\Yos{19}^3\,\Yos{20}^2\,\Yos{22}^4\,\Yos{24}^3\,\Yos{28}^6\,\Yos{31}^3}{\Yos{17}^3\,\Yos{21}^4\,\Yos{23}^3\,\Yos{25}^6\,\Yos{26}\,\Yos{27}^4\,\Yos{30}^2}\Big ).
  \end{equation}
The new value is realized generically along these four cones. It is only strictly larger than $\Yos{32}$ along non-apex cones having one of the aforementioned cones as a face.
\end{proposition}
\begin{proof} Formula~\eqref{eq:expval} shows that the expected valuation of $\Cross{15}$ on all cones in \eqref{eq:coneReps} is $\val(\Yos{32})$.
  Using~\textcolor{blue}{Lemmas}~\ref{lm:genExamplesaaaa} and~\ref{lm:genExamplesaaab} we confirm that this expected valuation is achieved generically in the relative interior of  all 24 Naruki cone representatives except for the four cones listed in the statement. We use tropical convexity to determine the correction factor in these four cases. To this end, we compute the symbols $E$ in $\sed$ for which the matrix of Cross components $M^{Cr}$ associated to the leaves of $\Trop E$ features $\Cross{15}$. Since this Cross function is associated to $x_{53}$ in $\ted$, there are precisely three cases to consider: $\E{5}$, $\F{35}$ and $\Gc{3}$.

  En each of these cases, other Cross functions with undetermined valuations can be found in the matrix $M^{Cr}$. We replace each one of them by their expected valuation plus a parameter $g_i$ for $i=0,\ldots, 4$, that accounts for the potential gaps. For $\E{5}$ and $\Gc{3}$, $g_0 = \val(\Cross{15})-\val(\Yos{32})$, whereas for $\F{35}$, the gap becomes $g_1$. Our goal is to determine positive lower bounds for $g_0$ and $g_1$.

  For each cone $\sigma$ in the list of four exceptions, we compute the $10\times 10$ matrices of leaves for $\Trop \E{5}$, $\Trop \F{35}$ and $\Trop \Gc{3}$ using the coordinate projection from~\eqref{eq:pi_E}. We specialize each sum $M^{Cr}+ M^{\Yos{}}$  at the baricenter of $\sigma$ and search for non-singular $3\times3$-tropical minors that would violate the tropical collinearity criterion from~\autoref{pr:tropicalLinesbyMinors}. The list of all these minors is available in the Supplementary material. The matrices for $\Trop \E{5}$ and $\Trop \F{35}$ yield a total of 40 problematic minors, whereas $\Trop \Gc{3}$ has only 34. We are only interested in those minors that involve the gap $g_i$ for $\val(\Cross{15})$. From each list, we pick one of these minors that has a particular shape. We use the same rows and columns for each of the four matrices associated to the same exceptional curve. The explicit matrices are shown in~\autoref{tab:badMinors}.

    \begin{table}[htb]
  \centering
  \begin{tabular}{|c|c|c|c|c|c|c|}
    \hline \!Line\!& Rows & Cols & $\operatorname{(aa_2a_4)}$ & $\operatorname{(aa_4)}$ & $\operatorname{(a_2a_4)}$ & $\operatorname{(a_4)}$ 
    \\
    \hline
       $\E{5}$ & $[1,3,6]$ & $[0,1,3]$ &
\scalebox{0.9}{\!\!$\begin{pmatrix}
    -2 & \infty & -1 \\
    -9 & -9 & \infty \\
    -4 & -4 & g_0\!-\!4
  \end{pmatrix}$}\normalsize\!\! &
\scalebox{0.9}{\!\!$\begin{pmatrix}
    -1 & \infty & 0 \\
    -6 & -6 & \infty \\
    -2 & -2 & g_0\!-\!2
  \end{pmatrix}$}\normalsize \!\!&
    \scalebox{0.9}{\!\!$\begin{pmatrix}
    -2 & \infty & -1 \\
    -8 & -8 & \infty \\
    -4 & -4 & g_0\!-\!4
      \end{pmatrix}$}\normalsize \!\!&

        \scalebox{0.9}{\!\!$\begin{pmatrix}
    -1 & \infty & 0 \\
    -5 & -5 & \infty \\
    -2 & -2 & g_0\!-\!2
          \end{pmatrix}$\!\!}\normalsize
            \\
            \hline

       $\F{35}$ & $[2,4,5]$ & $[0,2,4]$ &
\scalebox{0.9}{\!\!$\begin{pmatrix}
    -2 & \infty & 0 \\
    2 & 3 & \infty \\
    -9 & -8 & g_1\!-\!8
  \end{pmatrix}$}\normalsize\!\! &
\scalebox{0.9}{\!\!$\begin{pmatrix}
    -1 & \infty & 1 \\
    1 & 2 & \infty \\
    -6 & -5 & g_1\!-\!5
  \end{pmatrix}$}\normalsize \!\!&
    \scalebox{0.9}{\!\!$\begin{pmatrix}
    -2 & \infty & 0 \\
    1 & 2 & \infty \\
    -8 & -7 & g_1\!-\!7
      \end{pmatrix}$}\normalsize \!\!&

        \scalebox{0.9}{\!\!$\begin{pmatrix}
    -1 & \infty & 1 \\
    0 & 1 & \infty \\
    -5 & -4 & g_1\!-\!4
          \end{pmatrix}$\!\!}\normalsize
            \\
            \hline

         $\Gc{3}$ & $[2,3,7]$ & $[0,2,3]$ &
\scalebox{0.9}{\!\!$\begin{pmatrix}
    -6 & \infty & -5 \\
    2 & 3 & \infty \\
    -4 & -3 & g_0\!-\!4
  \end{pmatrix}$}\normalsize\!\! &
\scalebox{0.9}{\!\!$\begin{pmatrix}
    -4 & \infty & -3 \\
    1 & 2 & \infty \\
    -2 & -1 & g_0\!-\!2
  \end{pmatrix}$}\normalsize \!\!&
    \scalebox{0.9}{\!\!$\begin{pmatrix}
    -5 & \infty & -4 \\
    1 & 2 & \infty \\
    -4 & -3 & g_0\!-\!4
      \end{pmatrix}$}\normalsize \!\!&

        \scalebox{0.9}{\!\!$\begin{pmatrix}
    -3 & \infty & -2 \\
    0 & 1 & \infty \\
    -2 & -1 & g_0\!-\!2
          \end{pmatrix}$\!\!}\normalsize
            \\
            \hline
  \end{tabular}
  \caption{One minor for each combination of line and cone involving $\Cross{15}$. They become tropically singular when replacing  $g_i$
    with $g_i\!+\!1$. From top to bottom, the labeling triples for all rows and columns  are $([\Gc{4}, \F{35}, \Gc{3}],[x_{12}, x_{14}, x_{35}])$, $([\F{24}, \Gc{3}, \E{5}], [x_{31}, x_{24}, x_{13}])$ and $([\E{4}, \F{35}, \E{5}], [x_{12}, x_{41}, x_{35}])$, respectively.
\label{tab:badMinors}}
  \end{table}

    When viewing the corresponding minors in the original matrices $M^{Cr}+M^{\Yos{}}$ on $\sigma$,  these tropical permanents only involve three finite terms. Each of them is a linear expression in the valuations of all Yoshidas and the parameter $g_i\!:=\val(\Cross{15})-\val(\Yos{32})$. The coefficient of $g_i$ is always 1. We verify that  the terms not featuring $g_i$ agree when restricted to $\sigma$. Since these minors must be tropically singular for convexity reasons, we obtain a lower bound for $g_i$ that is linear in the valuations of the Yoshida functions. Each row of the table gives a different bound, but they all agree on each cone $\sigma$. This function is precisely the value of $\varepsilon$ in~\eqref{eq:varepsilonGap}.

    A direct computation reveals that $\varepsilon$ is strictly positive on all non-apex cones containing one of the four cones from the statement as a face.     In all other cases, $\varepsilon$ vanishes.
    We verify that $\varepsilon$ has value one  on each of the four baricenters, as we expected from the generic choice of parameters from~\autoref{lm:genExamplesaaaa}. A similar computation certifies the analogous result along the relative interior of each of these four cones. We conclude that on these cones, the new expected valuation for $\Cross{15}$ is $\val(\Yos{32}) + \varepsilon$ and it is achieved generically.
\end{proof}

The knowledge of more precise formulas for the expected valuations of each Cross function in $\cR$ allows us to define the appropriate genericity conditions on classical cubic surfaces required to predict the boundary structure of  their tropical counterparts in $\TPr^{44}$.

\begin{definition}\label{def:genericX}
Let $X\subset \P^{44}$ be a smooth cubic surface without Eckardt points. We say $X$ is \emph{generic with respect to the set $\cR$} if the following conditions hold:
\begin{enumerate}[(i)]
  \item the expected valuations of all 44 Cross functions for $X$ lying on the set $\cR_0$ from~\autoref{pr:relevantCrossFunctions} are achieved; and 
  \item the expected valuation for $\Cross{15}$ given by~\autoref{pr:newExpValCr15} is also attained for $X$.
\end{enumerate}
\end{definition}

Below is the main result in this section, restating the last claim in \autoref{thm:Naruki}:
\begin{theorem}\label{thm:PL}  Assume that the cubic surface $X\subset \P^{44}$ is generic with respect to the set $\cR$. The leaf labeling and metric on each boundary tree of $\Trop X$ can be recovered from~\autoref{tab:treeLabeling}. Furthermore, the metric structure on the boundary trees of tropical cubic surfaces in $\TPr^{44}$ is a continuous piecewise linear function on the Naruki fan when restrcted to the above genericity constraints.
  \end{theorem}
\begin{proof} We follow the proof strategy from~\autoref{thm:treeTypes}, and encode the leaves of each boundary tree by $M^{Cr} + M^{\Yos{}}$. The genericity assumption allows us to write all entries of this $10 \times 10$ matrix as  piecewise linear functions on the 40 Yoshida valuations, with domains of lineality given by the relative interiors of all Naruki cones.

  Working with the 24 cone representatives arising from~\eqref{eq:coneReps} we rewrite each entry in terms of the tuples of scalars $(r_1,\ldots, r_4)$ or $(r_1,\ldots, r_4')$, depending on the nature of the cone containing $\underline{p}$. \autoref{alg:treesFromLeaves} yields formulas for all vertices of each tree in terms of these scalars. They can be found in the Supplementary material.

  By simple inspection we verify that the entries of all vertices of each tree are continuous piecewise linear functions on the 24 cone representatives. The functions are linear along the relative interior of each cone. Thus, the edge lengths are also continuous piecewise linear functions on $\Naruki$. Precise formulas can be obtained by specilization of those in~\autoref{tab:treeLabeling}. The combinatorial type of each tree is determined from the corresponding trees on the two maximal cones by contracting all edges of length zero.   This concludes our proof.
\end{proof}

\begin{remark}It is worth pointing out that these genericity assumptions affect only 223 trees in tropical surfaces associated to non-apex Naruki cones and all 27 trees for surfaces associated to the apex of $\Naruki$. If we wish to obtain the combinatorial type and metric of each tree on $\Trop X$ for non-generic cases, we can do so as follows. Rather than working with the expected valuations of the five Cross functions appearing in $M^{Cr}$, we replace each  $\val(\Cross{i_j})$ with $\expVal(\Cross{i_j}) + g_j$ for $j=0,\ldots, 4$, as we did in the proof of~\autoref{pr:newExpValCr15}.

  The parameters $g_j$ are all non-negative, and their values depend on the Yoshida functions associated to $X$. In most situations (146 out of 250), the values of $\val(\Yos{})$ force all but one parameter $g_i$ to vanish. Only 32 trees will involve all five parameters $g_0,\ldots, g_4$. We let $g_{i_1}, \ldots, g_{i_s}$ be the potential non-vanishing parameters for a fixed $E\in \sed$.

  Running our implementation of~\autoref{alg:treesFromLeaves} on the field $\Q(r_1,\ldots, r_4, r_4', g_{i_1}, \ldots, g_{i_s})$ leads to new combinatorial types for the tree $\Trop E$, where the extra edges appearing have lengths $g_{i_1},\ldots, g_{i_s}$. The results can be found in the Supplementary material. In particular, when the  cone associated to $X$ is the apex, all 27 boundary tree aquire five bounded edges, as seen in \autoref{fig:nonGenApexTrees}. These edges will have length zero whenever $X$ is generic with respect to $\cR$.
\end{remark}
\begin{figure}[htb]
  \includegraphics[scale=0.65]{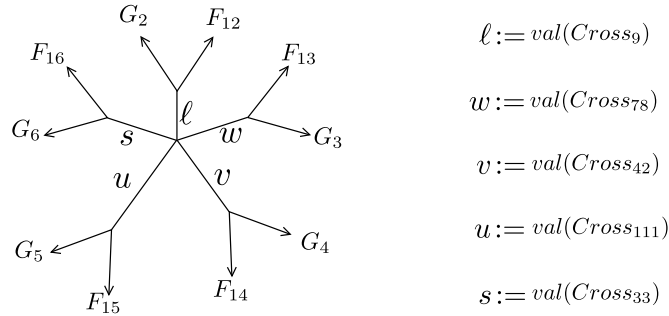}
  \caption{The boundary metric tree $\Trop \E{1}$ in $\Trop X$. In the stable case, all bounded edges have length zero and $\Trop \E{1}$ becomes a star tree with ten leaves.\label{fig:nonGenApexTrees}}
\end{figure}

\section{Boundary tree arrangements from planar configurations}\label{sec:planarConf}
\autoref{sec:comb-types-tree} discusses how to label the leaves of the 27  boundary metric trees on all anticanonical stable tropical cubic surfaces in $\TPr^{44}$.  In this section we present an alternative way to determine this data for $\operatorname{(aa_2a_3a_4)}$ and $\operatorname{(aa_2a_3b)}$ tropical cubic surfaces when we input  a configuration of six tropically generic  points in $\TPr^2$. These methods can be derived from~\cite[Theorem 4.4]{ren.sha.stu:16}.

We start by discussing a central property of the 27 boundary trees. By~\autoref{lm:coverLtoP1}, each line $\ell$ on
a smooth cubic surface admits a 2-to-1 cover $\ell \to \P^1$. Such map induces an involution on $\ell$. The same holds for their
tropicalization.  The
combinatorial structure of each tree $\Trop \ell$ described in~\autoref{thm:treeTypes} allows us the view the involution as a symmetry. Here is the precise statement:

\begin{corollary}\label{cor:invOnTrees}
  Each of the marked 27 trees on an anticanonical tropical cubic surface in $\TPr^{44}$ has an involution. The correspondence is given by the vertical symmetry of each tree that identifies the leaves on~\autoref{tab:treeLabeling} labeled by $i$ and $9-i$, for $i=0,\ldots, 4$.  
\end{corollary}

\begin{remark}
The analogous result for the universal Cox embedding was established in~\cite[Proposition 2.4]{ren.sha.stu:16}. This involution is closely related to the 2--to--1 cover $\ell\to \P^1$ from~\autoref{lm:coverLtoP1} which we view over the field $\widehat{L}$. Indeed we represent $\P^1$ as a line $\ell'$ on $\P^4$, where we fix the five projective coordinates of $\P^4$ to be the products of pairs in $\sed$ forming one of the five anticanonical triangles containing $\ell$. The line is then determined by the ten trinomials in the ideal over $\widehat{L}$ from~\autoref{thm:universalCoxL} involving these five products. Tropicalization turns $\ell'$ into a metric  tree in $\TPr^4$ with five leaves and the cover becomes the map $\Trop \ell \to \Trop \ell'$. The fiber over each leaf on the target space consists of two leaves labeled by the pair in $\sed$ determining the unique $\infty$ coordinate on the target leaf. These two symbols are exchanged by the involution.
\end{remark}

Since the involution swaps pairs of lines meeting a fixed one, it seams feasible that such involution could be achieved by an element of $\Wgp$. The following statement provides a negative answer:

\begin{proposition}\label{pr:invNotWE6}
  The involutions on the 27 metric trees cannot be realized by elements of $\Wgp$.
\end{proposition}
\begin{proof}
  It suffices to show the statement for a single tree, say $G_1$. We argue by contradiction. Consider an element of $\Wgp$ realizing the involution for the tree $\Trop G_1$. By construction, the involution identifies $E_j$ and $F_{1j}$ for $j=2,\ldots, 6$. A calculation with~\sage~available in the Supplementary material confirms that an element of $\Wgp$ fixing $G_1$ can only satisfy four of these five constraints. Furthermore, such element is unique. For example, if we verify the condition for $j=2,3,4,5$, this unique element of $\Wgp$  fixes $\Gc{1}$, $\E{6}$, $\F{16}$ and identifies the remaining 24 symbols in $\sed$ as follows:
  \[
E_1 \leftrightarrow G_6\;;\;  F_{23}\leftrightarrow F_{45}\,, \; F_{24}\leftrightarrow F_{35}\,, \; F_{25}\leftrightarrow F_{34}\,; \; E_j\leftrightarrow F_{1j} \text{ and } G_j\leftrightarrow F_{j6} \text{ for }j=2,\ldots, 5.\qedhere
  \]
\end{proof}

Recall from~\autoref{sec:moduli-cubics} that smooth cubic del Pezzo surfaces are successive blow-ups of $\P^2$ along six generic points $p_1,\ldots, p_6$. After a Cremona transformation, we set $p_1,p_2$ and $p_3$ to be the torus fixed points $(1:0:0)$, $(0:1:0)$ and $(0:0:1)$. Genericity forces the remaining three points $p_4$, $p_5$ and $p_6$ to lie in the dense torus. We set $P_i := \trop(p_i)\in \TPr^2$ for each $i=1,\ldots, 6$ and  assume this configuration of six distinct  points  in $\TPr^2$ is \emph{tropically generic}. That is, any two points lie in a unique tropical line, any three are not tropically collinear, and any five lie in a unique tropical conic which, furthermore, avoids the  sixth point.

  The choice of $p_1, p_2$ and $p_3$ determines the Newton polytopes and tropicalization of the plane curves associated to all $\F{ij}$ and $\Gc{i}$, obtained by blowing down our six exceptional divisors $\E{1},\ldots, \E{6}$. They are depicted in~\autoref{fig:NPAndTropCurves}. 
By abuse of notation we refer to these tropical lines and conics in $\TPr^2$ as $\F{ij}$ and $\Gc{i}$, respectively. By construction, $\F{ij}$ passes through $P_i$ and $P_j$, and $\Gc{i}$ avoids $P_i$.  

\begin{figure}[htb]
    \includegraphics[scale=0.4]{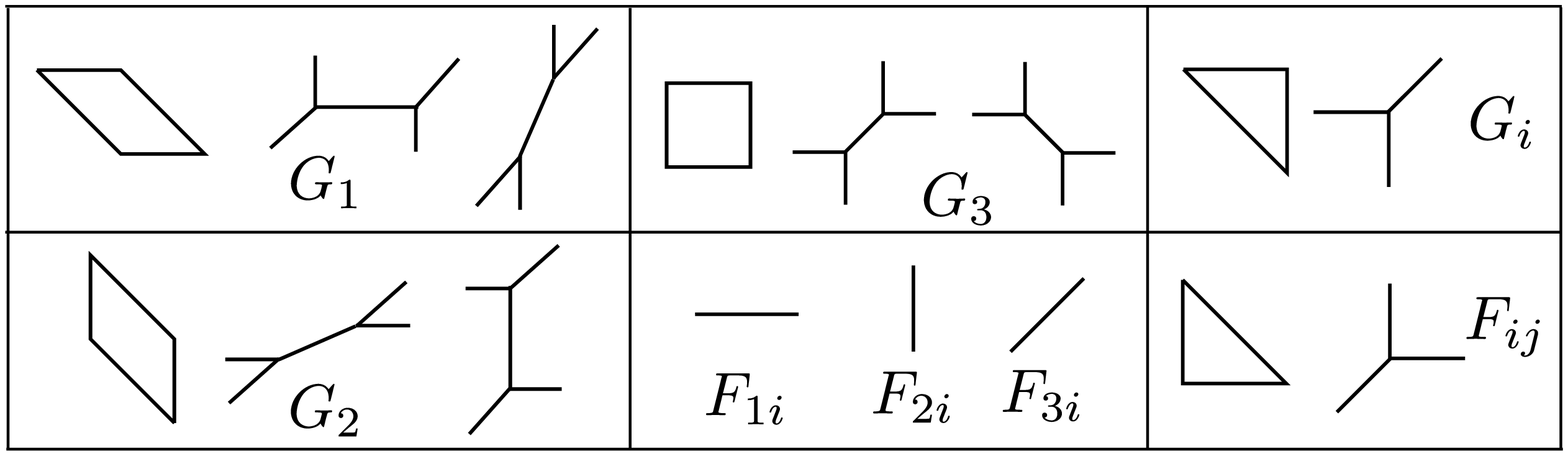}
    \caption{Newton polytopes and tropical curves $\F{1i}, \F{2i},\F{3i}$, $\Gc{i}$ and $\F{ij}$  ($3\leq i<j\leq 6$).\label{fig:NPAndTropCurves}}. 
  \end{figure}

The genericity conditions on $P_1,\ldots, P_6$ ensure that the associated tropical cubic surface $\Trop X$ in $\TPr^{44}$ is either an $\operatorname{(aa_2a_3a_4)}$ or an $\operatorname{(aa_2a_3b)}$ surface.  By analogy we  refer to the configurations yielding each type as $\operatorname{(aa_2a_3a_4)}$ or $\operatorname{(aa_2a_3b)}$ configurations, accordingly. 
Our goal is to determine the  type from location of $P_4,P_5$ and $P_6$. By~\autoref{thm:treeTypes}, we can distinguish the type by the number of combinatorial types of boundary trees.   Our next statement,  closely related to the constructions in~\cite[Section 5]{ren.sha.stu:16},  shows that  we can  characterize all but three of all 27 trees from this data. 

\begin{proposition}\label{pr:planarConf} The metric trees associated to all 27 lines on $X$ except $\E{4},\E{5}, \E{6}$ is completely determined by the position of the six points $P_1,\ldots, P_6$ whenever they are tropically generic.  
    \end{proposition}
\begin{proof} Our prior discussion allows us to assume that $P_1$, $P_2$ and $P_3$ are the tropicalization of the three torus fixed points in $\P^2$. 
Up to $\Sn{3}\times \Sn{3}$-symmetry there are six cases to analyze, namely $\E{1}$,  $\F{13}$, $\F{35}$, $\F{45}$, $\Gc{2}$ and $\Gc{5}$.  Each of the reference tropical plane curves is depicted with dashed lines in~\autoref{fig:Confaa2a3b}.
Since the curves $\F{12}$, $\F{13}$, and $\F{23}$ are in the boundary of $\TPr^2$, we draw them in the boundary of our rectangles, between each of the three points representing $E_1$, $E_2$ and $E_3$. Similarly, to each of the exceptional curves $\E{1}, \E{2}$ and $\E{3}$ we associate a horizontal, vertical or diagonal dashed line  passing through one of $(-N:0:0)$, $(0,-N,0)$ or $(N:N:0)$, respectively,  for $N$ large enough. The dashed curve associated to $\E{1}$ can be seen in the top-left of \autoref{fig:Confaa2a3b}.

  Each metric tree $\Trop \ell$ will be constructed from the input dashed tropical curve by analyzing its intersection with the ten tropical plane curves associated to the symbols $E$ in $\sed$ meeting $\ell$ on the cubic surface. Each pairwise intersection consists of either finitely many points counted with multiplicity (by tropical B\'ezout's Theorem) or a combination of points and edges for non-proper intersections. By abuse of notation we use $\ell$ and $E$ to denote their associated planar tropical curves.

  In what follows, we describe a strategy to mark the intersection
  points on each planar dashed curve $\ell$.  First, if the
  intersection between $E$ and $\ell$ is a single point (with
  multiplicity) in $\TPr^2$, we label the point by $E$. This happens,
  for example, for the dashed curve $\F{35}$ and the curve $\F{16}$ in
  the bottom-left of~\autoref{fig:Confaa2a3b}.  Second, assume the
  curve $E$ meets the dashed curve $\ell$ in exactly two points (with
  multiplicity), and one of them is one of the marked points $P_i$ for
  $i=4, 5, 6$.  In this situation, we label the unmarked
  intersection point by $E$, and label the marked point $P_i$ by
  $\E{i}$. This occurs, for example with the dashed curve $\Gc{5}$ and
  the curve associated to $\F{56}$ in the bottom-right picture.

  The intersections of the dashed curve $\ell$ with $\E{1}$, $\E{2}$
  and $\E{3}$ occur as the leaves of $\ell$ at the end of a ray with
  direction $(-1:0:0)$, $(0:-1:0)$ and $(1:1:0)$,
  respectively. Similarly, the intesection between $\ell$ and either
  $\F{12}$, $\F{13}$, or $\F{23}$ are obtained as the endpoints of
  rays with directions $(-1:-1:0)$, $(0:1:0)$ or $(1:0:0)$,
  respectively. This can be seen in the top- and bottom-left pictures
  in~\autoref{fig:Confaa2a3b} corresponding to the dashed curves
  $\F{45}$ and $\Gc{5}$.

Note that it is possible to have the same intersection point between $\ell$ and more than one undashed curve. If so, we label  the point by all such curves. This is the case, for example, for the point in the dashed curve $\F{13}$ labeled by $\F{25}$ and $\F{56}$ in the middle-left picture. Finally, when a curve $E$ meets a dashed curve $\ell$ along an edge or two (as it occurs for the curve $\F{25}$ and the dashed curve $\Gc{2}$ in the middle-right picture), we label the whole non-transverse intersection as $E$.

  \begin{figure}[htb]
  \includegraphics[scale=0.8]{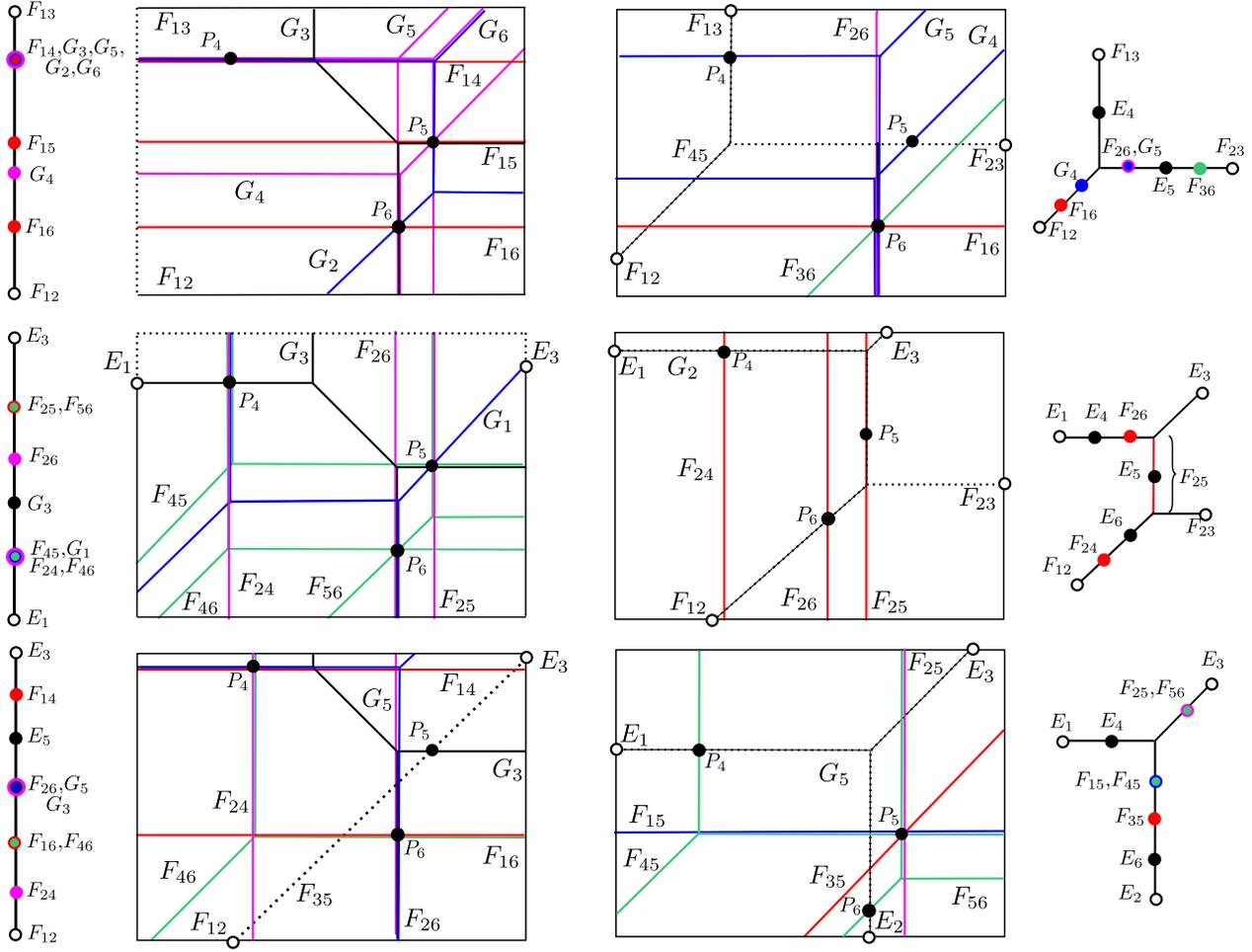}
  \caption{From left to right and top to bottom: the induced labelings for dashed tropical plane curves build from  $\E{1}$, $\F{45}$, $\F{13}$, $\Gc{2}$, $\F{35}$ and $\Gc{5}$ for the  $\operatorname{(aa_2a_3b)}$ configuration $\{(-10:10:0),(2:5:0), (0:0:0) \}$ in $\TPr^2$. \label{fig:Confaa2a3b}}
  \end{figure}

  Whenever we see exactly ten intersection points, we label each of them by the unique symbol in $\sed$  corresponding to the containing tropical curve. Each point will be adjacent to a leg with a labeled leaf and furthermore, $\Trop \ell$ will become a caterpillar tree, i.e.\ a type (II) tree on~\autoref{fig:treesaa2a3b}. In particular, this ensures that we have  an $\operatorname{(aa_2a_3b)}$ configuration. This occurs, for example, for the tree $\Trop \Gc{1}$ and the planar configuration where $P_4=(-10:10:0)$, $P_5=(2:5:0)$ and $P_6=(0:0:0)$, seen in~\autoref{fig:Confaa2a3b}.

  In all remaining cases, we will have between six and nine labeled points and one or two labeled edges. The number  depends on the input planar point configuration. Each single-label $E$ on a point gives rise to a leg with leaf $E$ on the tree $\Trop \ell$. Each multi-labeled point will give rise to a labeled branch emanating from this point. To recover $\Trop \ell$ and its labeling we must determine the topology and labeling on all these branches. The metric structure on $\Trop \ell$ will be obtained as a by-product. By construction, it depends solely on the input planar configuration.

  The involution of each metric tree described in~\autoref{cor:invOnTrees} will allow us to determine the missing information. 
The procedure to characterize branches can be illustrated with the top-left picture in~\autoref{fig:Confaa2a3b}. There is only one branch to compute, namely, that  emanating from the five-labeled vertex adjacent to $\F{13}$. 
  We start from the points labeled $\F{12}$ and $\F{16}$. By construction, they are both adjacent to a common vertex in the final tree (i.e., they form a cherry.) Since the involution  swaps this pair  with $\Gc{2}$ and $\Gc{6}$, we know the branch must have  a cherry  with its two leaves labeled by the later pair. Finally, since the points labeled $\Gc{4}$, $\F{15}$ and $\F{13}$ will be swapped with $\F{14}$, $\Gc{5}$ and $\Gc{3}$ under the involution, we must attach three legs to the branch along its spine (with leaves labeled in that order) and before the cherry computed earlier. We conclude that $\Trop\E{1}$ is a Type (II) tree on an $\operatorname{(aa_2a_3b)}$ cone.

Once all branches of a tree have been computed, we turn to the labeled edges. The involution determines the unique location on the edge to place the legs corresponding to each symbol in the edge-label. For example, in the middle-right of~\autoref{fig:Confaa2a3b}, the leg corresponding to $\F{25}$ in the dashed curve $\Gc{2}$  must be placed so that the distance between the markings $\E{5}$ and $\E{4}$ on the dashed curve agrees
with that of $\F{25}$ and $\F{24}$.  This concludes our proof.
  \end{proof}

The leaf labeling for the trees $\Trop E_4$, $\Trop E_5$ and $\Trop E_6$ which were  not discussed in \autoref{pr:planarConf} can be recovered by combining the action of $\Wgp$ with Cremona transformations of $\P^2$. Alternatively, we can simply consider the action this group on the leaf labels given by~\autoref{tab:treeLabeling}. 

For example, consider the two configurations $\{(-12:6:0),(-7:4:0), (0:0:0)\}$, and $\{(-10:10:0),(2:5:0), (0:0:0)\}$ in $\TPr^2$. The later can be seen in \autoref{fig:Confaa2a3b}. By \autoref{pr:planarConf} and~\autoref{thm:treeTypes}, we can determine the nature of these configutations by recording how many types of trees appear among the 24 we can compute. A explicit computation allows us to determine the first is an $\operatorname{(aa_2a_3a_4)}$, whereas the second one is an $\operatorname{(aa_2a_3b)}$ configuration. We label each of the 24 trees on each case by the corresponding symbol in $\sed$ and the tree type (either (I), (II) or (III).) We write down five sets, each containing the symbols in $\sed$ yielding a tree of each type:
\begin{equation}\label{eq:labelSets}
  \begin{aligned}
    \mathscr{L}_{I}&:=\sed\smallsetminus\{\E{4},\E{5},\E{6},\F{35},\Gc{3} \},
    \qquad 
    \mathscr{L}_I':=\{\E{2}, \E{3}, \F{14}, \F{15}, \F{23}, \F{26}, \F{36}, \Gc{1}, \Gc{4}, \Gc{5}\}, \\
  \mathscr{L}_{II}& := \{\F{35},\Gc{3}\},  \;\;     \mathscr{L}_{II}':= \{\E{1}, \F{12}, \F{13}, \F{16},  \F{24}, \F{25}, \F{34}, \F{46},\F{56}, \Gc{2}, \Gc{6} \},\;\;   \mathscr{L}_{III}' =   \mathscr{L}_{II}.
  \end{aligned}
\end{equation}
The two sets for the $\operatorname{(aa_2a_3a_4)}$-configuration are indicated by $\mathscr{L}_{I}$, $\mathscr{L}_{II}$. The remaining three correspond to the $\operatorname{(aa_2a_3b)}$ one.

To determine the three missing trees on each case and their labelings, we first find the elements of the Weyl group $\Wgp$ sending each of these five sets in~\eqref{eq:labelSets} to the sets of symbols labelling each tree type in~\autoref{tab:treeLabeling}. For our two sample configurations, there are 1152 and 96 such elements, respectively. Second, we let these elements act on the 24 pairs ($E$, tree type) for each the planar configuration and record those giving  the labeling pairs in the table as outputs. This yields 16 and eight elements for each case, respectively; as many as the size of the stabilizers of each maximal cone in the Naruki fan.

The inverse of the any of these elements in $\Wgp$ determines the labeling on the three missing  trees  by acting on the three unaccounted  trees from~\autoref{tab:treeLabeling}. The labeling and tree types will be independent of the chosen Weyl group element. For example, if we pick the single element determined by the $6\times 6$ matrices $\sigma_0$ and $\sigma_1$  below (corresponding to the action on the simple roots $\{\alpha_1,\ldots, \alpha_6\}$) and we  act on  the triples of trees indexed by $\{\E{2}, \E{6}, \F{35}\}$ and $\{\E{3}, \Gc{3}, \F{15}\}$, respectively, we obtain the labeling on $\Trop\E{4}$, $\Trop\E{5}$ and $\Trop\E{6}$:
\[
  {\sigma_0}:=
\left(\begin{array}{rrrrrr}
  0 & 1 & 0 & 0 & -1 & 1\\
  0 & 0 & 0 & 1 & -2 & 1\\
  -1 & 1 & 0 & 1 & -2 & 1\\
  -1 & 1 & 0 & 1 & -3 & 2\\
  0 & 1 & -1 & 1 & -2 & 1\\
  0 & 0 & -1 & 1 & -1 & 1
\end{array}\right) \qquad \text{ and } \qquad
  {\sigma_1}:=
\left(\begin{array}{rrrrrr}
  0 & -1 & 0 & 1 & 0 & -1\\
  -1 & 0 & 0 & 1 & 0 & -1\\
  -1 & -1 & 0 & 1 & 1 & -2\\
  -1 & -1 & -1 & 2 & 1 & -2\\
  0 & 0 & -1 & 1 & 1 & -2\\
  0 & 0 & 0 & 0 & 1 & -1
  \end{array}\right).
\]
In particular, we conclude that $\Trop\E{4}$ has types (I) and (II)  for each sample configurations, whereas $\Trop\E{5}$ has types (II) and (III), respectively. Finally,  $\Trop\E{6}$ is of type (I) in both cases.    For further details on the ~\python\ and \sage\ implementations, we refer to the Supplementary Material.

\section{The moduli space of stable tropical cubic surfaces}\label{sec:trop-moduli-space}

~\autoref{thm:Naruki} provides a modular interpretation of the Naruki fan in terms of stable tropical surfaces. The goal of this section is to prove this statement. In particular, we provide a concrete test to determine when a tropical cubic surface is stable.

We start by describing the general framework. We fix a $\KK$-value for the 40 Yoshida functions $\Yos{} = (\Yos{0}:\ldots:\Yos{39}) \in \Gm^{40}/\Gm$ where the Cross functions do not vanish, and let $X_{\Yos{}}$ be the smooth cubic surface associated to $\Yos{}$. \autoref{lm:alwaysGeneric} ensures that the defining ideal of  $X_{\Yos{}}$ in $\P^{44}_{\KK}$ is obtained by specializing the generators of the ideal in \autoref{thm:anticanL} at the prescribed point $\Yos{}$. The tropical variety $\Trop X_{\Yos{}}\subset \TPr^{44}$ is obtained from this ideal.

By construction, $X_{\Yos{}} \subset \P^{44}$ is the compactification of the very affine variety $X_{\Yos{}}^{\circ} \in \Gm^{45}/\Gm$ obtain by remove from $X_{\Yos{}}$ its 27 exceptional curves. Furthermore, $X_{\Yos{}}$ is normal and $\Q$-factorial and the boundary of $X_{\Yos{}}$ in $\P^{44}$ is divisorial with simple normal crossings due to the absence of Eckardt points on $X_{\Yos{}}$. In this context, the theory of geometric tropicalization from~\cite[\S 2]{hac.kee.tev:09} provides a way to construct $\Trop X_{\Yos{}}$ from the divisorial valuation determined by each line in the boundary of $X_{\Yos{}}$ evaluated on a $\Z$-basis of the cocharacter lattice of $\Gm^{45}/\Gm$. For further details about this construction, we refer to~\cite[Section 2]{lux.qu:11} and~\cite[Chapter 6.5]{mac.stu:15}.

The anticanonical embedding is ideally suited for this tropicalization method, since the information regarding the 27 divisorial valuations for each $X_{\Yos{}}$ can be recovered from the boundary  of $\Trop X_{\Yos{}}$. We conclude:

\begin{theorem}\label{thm:BoundaryDeterminesEverything}
  The tropical cubic surface $\Trop X_{\Yos{}}$ is determined by its boundary line  arrangement.
\end{theorem}

\noindent
Further  details will  appear on a subsequent paper. We remark that this result also holds for stable tropical cubic surfaces when considering Cox embeddings (see~\cite[Lemma 3.3]{ren.sha.stu:16}.) In that context, the authors prove the statement by an explicit construction of tropical cubic surfaces via tropical modifications of $\TPr^2$, in the spirit of the techniques discussed in~\autoref{sec:planarConf}.  

\smallskip

Our next objective is to explain the role of the Naruki fan as a tropical moduli space. Tropicalizations of both $\Yos{}$ and $X_{\Yos{}}$ yield the point $\underline{p}=\val(\Yos{})$ in $\Naruki$ and the tropical surface $\Trop X_{\Yos{}}$:
\begin{equation}\label{eq:universalDiagram}
  \begin{gathered}
    \xymatrixcolsep{3pc}
  \xymatrix{
\hspace{-3ex}\Gm^{40}/\Gm  \ni \Yos{} \ar@<4ex>[d]_{\trop} \ar[rr] & & X_{\Yos{}} \subset \P^{44}_{\KK} \ar@<-3ex>[d]^{\trop}\\
\quad   \Naruki \ni \underline{p} \ar@{-->}[rr]^{\exists\,?}
& &  \Trop X_{\Yos{}}\subset \TPr^{44}
}
  \end{gathered}
\end{equation}
It is natural to ask if we can define a dashed map sending $\underline{p}$ to $\Trop X_{\Yos{}}$ that makes the diagram~\eqref{eq:universalDiagram} commute. The following result provides an answer:
\begin{proposition}\label{pr:ModularInterpretation}
  The dashed arrow in~\eqref{eq:universalDiagram} assigns to $\underline{p}$ the tropical surface $\Trop X_{\Yos{}}$ associated to a generic point $\Yos{}$ in the fiber over $\underline{p}$. The genericity conditions  are determined by the valuations of the 45 relevant Cross functions in $\cR$ and are always valid on the relative interior of the maximal cones in $\Naruki$. In particular, the diagram commutes when restricted to the fibers of these open cones.
\end{proposition}
\begin{proof} The result is a direct consequence of~\autoref{thm:BoundaryDeterminesEverything}. As we saw in \autoref{sec:comb-types-tree}, the arrangement of metric trees in the boundary of $\Trop X_{\Yos{}}$ is uniquely determined by the tropicalization of the 135 classical nodes of $X_{\Yos{}}$. In turn, these tropical nodes depend solely on $\underline{p}$ and the valuation of the 45 relevant Cross functions in $\cR$.  For a generic choice of $\Yos{}$ in the fiber over  $\underline{p}$, the valuation of all Cross functions agree with the expected ones. By combining~\textcolor{blue}{Theorems}~\ref{thm:PL} and~\ref{thm:BoundaryDeterminesEverything} we conclude that $\Trop X_{\Yos{}}$ is constant when restricted to generic points in the fiber over $\underline{p}$.  \autoref{thm:aaaaIsStable} shows that the expected valuation of each Cross function in $\cR$ is always achieved along the relative interior of any top-dimensional cell in $\Naruki$.
\end{proof}

\begin{theorem}  \label{thm:aaaaIsStable}
  The expected valuations of all 45 relevant Cross functions agree with their actual valuations along the relative interiors of all maximal cones in the Naruki fan.
\end{theorem}
\begin{proof}
 By construction, each Cross function can be written in four different ways as a difference of (signed) Yoshida functions. If no ties in the valuation are observed on a given expression $\Yos{i}-\Yos{j}$, the valuation on the Cross function is completely determined: it equals $\min\{\val(\Yos{i}), \val(\Yos{j})\}$. 

  A simple computation available in the Supplementary material certifies that on the representing $\operatorname{(aa_2a_3b)}$ cone from~\eqref{eq:coneReps}, all Cross functions admit one expression with no valuation ties. However, the situation is different on the sample cone $\operatorname{(aa_2a_3a_4)}$. As we saw in~\autoref{sec:bergman-fan-Naruki-fan} there are precisely three Cross functions  ($\Cross{i}$ for $i=36,37,38$ in~\autoref{tab:crosses}) where ties are observed for all four expressions. They correspond to the anticanonical triangle $x_{36}$ in $\ted$. Thus, to verify the valuations on these three Cross functions are the expected ones, it suffices to treat only one of them, e.g.\ $\Cross{37}$. This is the content of~\autoref{lm:Cr37}.
\end{proof}

  Following the notation from~\autoref{sec:find-repr-cross}, we fix a splitting $\gamma\mapsto t^{\gamma}$ of the valuation of $\KK$. We define the \emph{initial form} of an element $\alpha$ in $\KK$ as the residue class
  \begin{equation*}\label{eq:init}
  \init(\alpha) = \overline{t^{-\val(\alpha)} \alpha} \text{ in } \resK.
  \end{equation*}

\begin{lemma}\label{lm:Cr37} The valuation of $\Cross{37}$  equals $\val(\Yos{34})$ on the relative interior of the $\operatorname{(aa_2a_3a_4)}$ cone from~\eqref{eq:coneReps}. In particular, its expected valuation is always attained.
\end{lemma}
\begin{proof} Fix a point $\underline{p}$ in the relative interior of $\operatorname{(aa_2a_3a_4)}$. 
  Our goal is to show that the valuation of the Eckardt quintic $Q_{36}$  is independent on the choice of parameters $d_1,\ldots, d_6$ as long as the valuation of the induced Yoshida functions equal $\underline{p}$. 

  The data provided by the three tables from~\autoref{sec:coble-covariants} give the identities
  \begin{equation*}
    Q_{36} = -\frac{\Yos{34} + \Yos{8}}{(d_1+d_3+d_5)(d_1-d_5)(d_2-d_4)(d_2+d_3+d_4)} = \frac{\operatorname{Sum}_{34} + \operatorname{Sum}_{8}}{(d_5-d_1)},
  \end{equation*}
  where $\operatorname{Sum}_{34}$ and $\operatorname{Sum}_{8}$ are products of six roots in $\Phi^+$. More specifically,
  \begin{equation*}\label{eq:sum34_8}
    \begin{aligned}
      \operatorname{Sum}_{34} &:=  (d_6-d_1) (d_5-d_3) (d_1+d_2+d_6) (d_1+d_4+d_6) (d_2+d_4+d_5) (d_3+d_5+d_6),\\
      \operatorname{Sum}_{8} &:= (d_3-d_1) (d_6-d_5)  (d_1+d_2+d_4) (d_1+d_3+d_6) (d_2+d_5+d_6) (d_4+d_5+d_6).
      \end{aligned}
  \end{equation*}
  In the notation of~\autoref{tab:roots}, we rewrite these ordered products as
\begin{equation*}\label{eq:sum34_8_rootNumbers}
  \operatorname{Sum}_{34}= r_{24}\,r_{10}\,r_{17}\,r_{27}\,r_{26}\,r_{33}\quad \text{ and } \quad\operatorname{Sum}_{8}= \,r_{9}\,r_{5}\,r_{7}\,r_{23}\,r_{31}\,r_{34}.
\end{equation*}

  The result is a direct consequence of the following statement and our assumption that $\charF \resK \neq 2$:
  
  \begin{claim}\label{cl:centralC37} $\init(\operatorname{Sum}_{34}) = \init(\operatorname{Sum}_{8})$ in $\resK$ for all choices of $d_1,\ldots, d_6$ yielding $\underline{p} \in \operatorname{(aa_2a_3a_4)}^{\circ}$.
  \end{claim}
  
  The proof of this claim is tedious, yet elementary. Each choice of parameters $d_1,\ldots, d_6$ yields a point  in the fiber over $\underline{p}$ of the tropical Yoshida map $\trop(m)$ from~\eqref{eq:narukiFan}. By~\autoref{pr:fibersOfYoshida}, the roots associated to each such choice produce points in one of the 66 cones $\tau$ in the support of the Bergman fan $\cB$ constituting the preimage of the open cone $\operatorname{(aa_2a_3a_4)}^{\circ}$.
  
  Given a cone $\tau$ as above, we write all points in the fiber over $\underline{p}$ as a non-negative linear combination of the rays generating the closure of $\tau$ in $\R^{36}/\rspanone$. The scalars  will be subject to constraints analogous to those in~\eqref{eq:constraintFibers}.  A direct computation will allow us find a bijection between the six factors of $\operatorname{Sum}_{34}$ and those in $\operatorname{Sum}_{8}$ (dependent on $\tau$) so that their initial forms agree up to sign. This bijection will use the knowledge of the valuations of all 36 positive roots in $\Phi^+$ for each point in $\tau$ arising from the set of scalars used to express this point. In the end, all negative signs cancel out and we obtain the identity in~\autoref{cl:centralC37}.

  It is worth noticing that the bijection for a given cone will often work for others. This will be the case when the positivity constraints on the scalars is valid for more that one cone.  As a consequence of this, the number of cases to consider reduces to 34. The Supplementary material contains the explicit bijection for each case.

  In what follows, we carry out the required computation to prove~\autoref{cl:centralC37} in one instance, which involves the cones $\tau_0$ and $\tau_1$ from~\autoref{rm:sampleTausForCr37}. We set $\tau_i': = \tau_i\cap \trop(m)^{-1}(\operatorname{(aa_2a_3a_4)}^{\circ})$. The closure of both cones is spanned by seven rays. To lie in $\tau_i'$, the scalars $p_i$ for $i=0,\ldots, 6$ are subject to the linear constraints in~\eqref{eq:ineqstausprimes}. We identify a point $\underline{p}$ in $\tau_i$ with the vector of scalars $(p_0,\ldots, p_6)$.
  
  Each $\underline{p}$ in  $\tau_i$ allows us to express the valuation of all 36 positive roots in $\Phi^{+}$ as linear functions on the scalars $p_i$. \autoref{tab:cr37} records these values for all 12 roots appearing in the factorization of $\operatorname{Sum}_{34}$ and $\operatorname{Sum}_{8}$, and four additional roots. Notice that the values depend on the the ambient cone. As usual, we set $p_{I} := \sum_{i\in I} p_i$ for each $I\subset \{0,\ldots, 6\}$.

  \begin{table}[htb]
  \centering
  \begin{tabular}{|c||c|c|c|c|c|c||c|c|c|c|c|c|| c|c|c|c|}
 \hline \!Roots\!& $r_{10}$ & $r_{17}$ & \!$r_{24}$\! & \!$r_{26}$\! & \!$r_{27}$\! & $r_{33}$ & $r_{5}$ & $r_{7}$ & \!$r_{9}$\! & \!$r_{23}$\! & \!$r_{31}$\! & $r_{34}$ & $r_8$ & $r_{14}$ & $r_{32}$ & $r_{35}$\\  \hline
 $\tau_0$ & $p_{0456}$ & $p_{456}$ & $0$ & $0$ & $0$ & $0$ & $p_{456}$ & $p_{0456}$ & $0$ & $0$ & $0$ & $0$ & $p_{456}$ & $p_{456}$ & $0$ & $0$  \\
 $\tau_1$ & $p_{0456}$ & $ 0$ & $0$ & $0$ & $0$ & $p_{456}$ & $0$ & $p_{0456}$ & $0$ & $0$ & $0$ & $p_{456}$  & $0$ & $0$ & $p_{456}$ & $p_{456}$ \\
\hline
  \end{tabular}
  \caption{Valuation of all relevant roots in $\Phi^+$ for two of the cones in the fiber of the tropical Yoshida map over the relative interior of $\operatorname{(aa_2a_3a_4)}$ in terms of the scalars $p_0,\ldots,p_6$. The root numbers agree with those in~\autoref{tab:roots}. \label{tab:cr37}}
  \end{table}

  \noindent
  The following  five roots  have nonzero valuations. The formulas on both $\tau_0'$ and $\tau_1'$ agree:
  \begin{equation}\label{eq:restVals}
    \begin{aligned}
      \val(r_0) & = p_{123}, \quad \val(r_1)\! =\! 2\,p_0 + p_{2} + p_{3} + 3\, p_{4} + p_{5} + 3\,p_{6} ,\; \; \val(r_3)\!=\! \val(r_{13})\! = p_{0456} \; \text{ and}\\
      \val(r_4) &= 2\,p_0 +  p_{3} + 3\, p_{4} + p_{5} + p_{6}.
    \end{aligned}
    \end{equation}

  The values in~\autoref{tab:cr37} allow us to construct a partial bijection between the first set of six roots and the second one, by pairing identical columns that have one non-zero value. Indeed:

  \begin{claim}\label{cl:r10etc} $\init(r_{10}) = \init(r_3) = \init(r_7)$,  $\init(r_{17}) = \init(r_{14}) = \init(r_5)$ and $\init(r_{33}) = \init(r_{34})$.
  \end{claim}

  \noindent
  These  identities follow from the strong non-Archimedean triangle inequality on $\KK$. We write:
  \[r_{10} = r_4 + r_3, \quad r_7 = r_1 + r_3 , \quad r_{17} = r_{1} + r_{14}, \quad r_{5} = -r_{10} + r_{14}, \quad r_{33} = -r_3 + r_{34}.
  \]
  By~\eqref{eq:restVals} and~\eqref{eq:ineqstausprimes}, we know that $\val(r_{10}) = \val(r_{3}) < \val(r_4)$, thus $\val(r_{10}) = \val(r_3)$ and $\init(r_{10}) = \init(r_3)$. Similarly,
  $\val(r_{7}) = \val(r_{3}) < \val(r_1)$ implies $\init(r_{7}) = \init(r_3)$.
 The second identity follows from 
  \[
  \val(r_{1}) = \val(-r_{10}) = p_{0456} > \val(r_{17}) = \val(r_{14}) = \val(r_5).
  \]
  Note that these expression on the right-hand side  is different for $\tau_0'$ and $\tau_1'$. Finally,
  \[\val(-r_3) = p_{0456} > \val(r_{33}) = \val(r_{34})
\]
  gives the last identity in~\autoref{cl:r10etc}.

  To complete the partial bijection between the roots from $\operatorname{Sum}_{34}$ and $\operatorname{Sum}_{8}$ we must identify the remaining zero columns in \autoref{tab:cr37}. This will be a consequence of the following:
  \begin{claim}\label{cl:r24etc}
    $\init(r_{24} ) = \init(r_9)$, $\init(r_{26}) = -\init(r_{23})$ and $\init(r_{27}) = \init(r_6) = -\init(r_{31})$.
  \end{claim}
  To prove these identities we use the same methods as in ~\autoref{cl:r10etc}. In this case, we write:
  \begin{equation}\label{eq:idsForR24}
r_{24} = r_{14} + r_{9}, \quad r_{26} = -r_{23} + r_{35}, \quad r_{27} = r_{17} + r_{6}\quad \text{ and }\quad  r_{31} = -r_6 + r_{34}.
  \end{equation}
By~\eqref{eq:restVals} and~\eqref{eq:ineqstausprimes}, the valuation of $r_{14}$, $r_{35}$, $r_{17}$ and $r_{34}$ equal $p_{0145}>0$, whereas the roots $r_{24}$, $r_{9}$, $r_{26}$, $-r_{23}$, $r_{27}$, $r_6$ and $r_{31}$ have valuation zero. These conditions and~\eqref{eq:idsForR24} prove~\autoref{cl:r24etc}.

The bijection between the roots of $\operatorname{Sum}_{34}$ and $\operatorname{Sum}_8$ follows by combining Claims~\ref{cl:r10etc} and~\ref{cl:r24etc}. 
\end{proof}

\autoref{thm:aaaaIsStable} justifies our definition of stability for tropical cubic surfaces. It implies that $\Naruki$ is the moduli space of tropical stable cubic surfaces in $\TPr^{44}$. Our next result provides an effective way to determine this stability property: the surface $X_{\Yos{}}$ must be generic in the sense of \autoref{def:genericX}. 

\begin{corollary} \label{cor:checkStability}  Anticanonical tropical cubic surfaces are stable whenever the valuation of all 135 Cross functions agree with the expected ones.  \end{corollary}

\begin{remark} When considering those cubic surfaces $X_{\Yos{}}$ where $\underline{p}\in \Naruki$ lies in the 24 cone representatives determined by~\eqref{eq:coneReps} it suffices to check that the valuation of the 45 relevant Cross functions in $\cR$ agree with the expected ones. In particular, the only tropical stable cubic surface associated to the apex of $\Naruki$ is obtained by assuming $\KK$ is trivially valued.
  \end{remark}

\begin{remark} The proof of~\autoref{lm:Cr37} provides a method for finding cubic surfaces with non-stable tropical counterparts. Indeed,  it is enough to find parameters $d_1,\ldots, d_6$  giving a point in the Bergman fan $\cB$ lying on the fiber of $\underline{p}$ that force cancellations in the expected leading term of a fixed binomial in the positive roots $\Phi^+$ of $\EGp{6}$. The number of computations required to build these parameters grows with the codimension of the smallest cone of $\Naruki$ containing $\underline{p}$ so it will not be very effective as we move deeper into the apex of this fan.
\end{remark}

We end this section by discussing the behavior of the boundary trees when $\underline{p} = \mathbf{0}$ is the apex of $\Naruki$. In the stable case, each of these trees is a star tree with ten leaves. As the valuations of the Cross functions become higher than expected, ~\autoref{alg:treesFromLeaves}  produces trees on $\Trop X_{\Yos{}}$  with five bounded edges adjacent to a central vertex and two leaves adjacent to the other end of each edge. These two leaves  are swapped by the involution in~\autoref{cor:invOnTrees}. Furthermore, the valuations of the five relevant Cross functions associated to the five anticanonical triangles containing a given exceptional curve $E$ give the lengths of the  five bounded edges on $\cT E$. \autoref{fig:nonGenApexTrees} shows the tree $\cT \E{1}$ with its leaf labelling and metric.

Notice that, for dimensional reasons, these five valuations cannot be arbitrary and a linear relation among the edge lengths of $\cT E$ must occur. Such relation does not arise from an algebraic relation among the associated Cross functions but can be derived from numerical examples. The existence of cubic surfaces with Eckardt points ensures that for some choice of $\Yos{}$ with valuation zero the valuation of some of these five Cross functions will be strictly positive. This guarantees the existence of non-stable tropical surfaces for the apex of $\Naruki$. We leave the construction of a new tropical moduli space adapted to these unstable tropical surfaces for future work.

\section{Combinatorial types of stable anticanonical tropical del Pezzo cubic surfaces}\label{sec:comb-types}

As we discussed in~\autoref{thm:BoundaryDeterminesEverything} and~\autoref{pr:ModularInterpretation}, the combinatorics and metric structure on the boundary of stable tropical del Pezzo cubic surfaces in $\TPr^{44}$ without Eckardt points is completely determined by the Naruki fan. Furthermore,  \autoref{lm:IntersectionBoundaryLines} showed  the intersection complex of the tree arrangement in the boundary of each $\Trop X$ is encoded in the \emph{Schl\"afli graph}, just as it happened with the tropicalization induced by  the Cox embedding~\cite[Section 2]{ren.sha.stu:16}. 
In this section, we extend these similarities to the interior of each such surface. Our goal is to prove the first half of~\autoref{thm:Naruki}, namely that the combinatorics type of each surface agrees with that induced by Cox embedding, described explicitly in~\cite[Table 1]{ren.sha.stu:16}.

We address this question by working with the universal cubic surface $X_{\widehat{L}}$ over the field $\widehat{L}$ from~\eqref{eq:fields} with its Cox and anticanonical embeddings  described in~\textcolor{blue}{Theorems}~\ref{thm:universalCoxL} and~\ref{thm:anticanL}. In what follows we discuss the precise connection between them. Our choice of markings $\ted$ and $\sed$ for the 45 anticanonical triangles and the 27 exceptional curves from~\autoref{rm:markings}  yields an $\Wgp$-equivariant degree three monomial map
\begin{equation}\label{eq:alphaMap}
\alpha\from  \spec \KK[\sed] \to \spec \KK[\ted].
\end{equation}
Its exponents vectors are encoded by a rank 21 matrix of
size $45\times 27$ 
with 0/1 entries and three nonzero
entries per row. The matrix, which is denoted by $A$, is recorded in the Supplementary material. By~\autoref{rm:gradingAndAntiCanMap}, the map $\alpha$ is
compatible with the natural gradings on the two coordinate rings. By construction, $\alpha$ induces the equivariant monomial map between the Cox and anticanonical embeddings of  $X_{\widehat{L}}$.

The natural gradings on $\widehat{L}[\ted]$ and $A(X_{\widehat{L}})$ discussed in~\autoref{sec:anti-canon-emb}  induce natural torus actions on $\KK[\sed]$ and $\KK[\ted]$.
More precisely,  the action of an element $\underline{t}=(t_0,t_1,\ldots, t_6)$  of the 7-dimensional multiplicative split $\KK$-torus $\GG{7}$ on $\spec \KK[\sed]$ is determined by the $\Z^7$-grading from~\eqref{eq:Z^7grading} as follows:
\begin{equation}\label{eq:G7action}
  \underline{t} * \E{i}\! =\! t_i \E{i}
\;  ; \quad \underline{t}*\Gc{i} \!=\! t_0^2
(\prod\limits_{k\neq i} t_k)^{-1}\Gc{i} \quad(1\leq i \leq 6) \; ;\quad
\underline{t}*\F{ij} \!= \!t_0(t_it_j)^{-1}
\F{ij}  \;(1\leq i<j\leq 6).
\end{equation}
The action on $\KK[\ted]$ is obtained by combining the latter with the definition of the 45 symbols $x_{ij}$, $y_{ijklmn}$ in $\ted$ as degree three monomials in $\sed$.

The rank seven sublattice $\Lambda$ of $\Z^{27}$ inducing the action from~\eqref{eq:G7action} is saturated and contains the all-ones vector. At the cocharacter level, the all-ones vector action is obtained from the cocharacter  $t_0^3t_1\cdots t_5$.
The grading induced by it identifies $\proj(\KK[\ted])$ with $\P^{26}$. It inherits an action by the torus $\GG{7}/\Gm$.
The action of  $t_0$ on each variable in $\ted$  is given by scalar multiplication  by $t_0^3$. It follows that  $\proj(\KK[\sed])$ is a $\underline{3}$-weighted projective space, which by abuse of notation we denote as $\P^{44}$. The monomial map $\alpha$ from~\eqref{eq:alphaMap} is compatible with the torus actions discussed above and induces a monomial degree three map on the quotient space:
\begin{equation}\label{eq:alphaBar}
\ovalpha\from  \P^{26}/(\GG{7}/\Gm) \simeq \proj(\KK[\sed])/(\GG{7}/\Gm) \to \proj(\KK[\ted]) \simeq \P^{44}.
\end{equation}

Each smooth del Pezzo cubic $X$ with no Eckardt points defined  over the field $\KK$ is obtained from $X_{\widehat{L}}$ by specializing the parameters $d_1,\ldots, d_6$ away from the vanishing locus of the product of all Yoshida and Cross functions from~\autoref{tab:yoshida} and~\autoref{tab:crosses}. Furthermore, for any such choice, the corresponding del Pezzo cubic  embeds in  $\P^{26}/(\GG{7}/\Gm)$. This perspective extends the embedding of the very affine surface $X^{0}$ in $\Gm^{27}/\Gm^7$ over $\KK$ introduced in~\cite[Section 2]{ren.sha.stu:16}, where $X^0$ is obtained from $X$ by removing its 27 exceptional curves. 
Composing with  $\ovalpha$ yields
\begin{equation*}
 X \hookrightarrow \P^{26}/(\GG{7}/\Gm) \xrightarrow{\ovalpha} \P^{44}.
\end{equation*}

\begin{remark}\label{rm:adaptedVsNonAdapted}
  Notice that the choice of coordinates on $\P^{44}$ is given by the marking $\ted$, rather than the scaled variables in the set $T$ from~\eqref{eq:YoshidaAdaptedCoords}. The tropicalization of $X$ in $\TPr^{44}$ induced by the latter is obtained by translating the tropicalization of $X$ with respect to the embedding in ~\autoref{thm:anticanK}  by the image of the vector $\trop(\underline{Q}):=(\val(Q_{ij}), \val(Q_{ijklmn})\colon ij, \, ijklmn)\in \R^{45}/\rspanone$. This simple operation preserves the combinatorial types.
\end{remark}

We now turn our attention to the combinatorial types of anticanonical stable tropical del Pezzo cubics in $\TPr^{44}$ discussed in the first part of~\autoref{thm:Naruki}. Our next statement shows that  these  types match the classification  for Cox embeddings, as predicted by  Ren, Shaw and Sturmfels in~\cite[Section 5]{ren.sha.stu:16}. The remainder of this section will be devoted to its proof.

\begin{theorem}\label{thm:combTypesAreTheSame}
  The combinatorics of  stable tropical del Pezzo cubics without Eckardt points in $\TPr^{44}$  obtained by the anticanonical and Cox embeddings agree.
\end{theorem}

\begin{proof} By~\autoref{rm:adaptedVsNonAdapted} it suffices to show this statement for the embedding induced by the marking $\ted$. The result follows by translating the classification of tropical cubics induced by  the Cox embeeding~\cite[Table 1]{ren.sha.stu:16} to the anticanonical embedding in $\P^{44}$ using the map $\ovalpha$ from~\eqref{eq:alphaBar}.  ~\autoref{pr:injectivityNewEmbedding} shows the combinatorics of both tropicalizations is the same. 
  \end{proof}

The explicit relation between the Cox and anticanonical embeddings of the universal cubic surface is governed by the map $\ovalpha$ from~\eqref{eq:alphaBar}. We are thus compelled to study its behavior under tropicalization.
We let $\ovLambda:=\Lambda/\zspan{\mathbf{1}}$ be the lattice inducing the action of $\GG{7}/\Gm$ on $\P^{26}$ obtained from~\eqref{eq:G7action}. This action allows us to consider the quotient space $\TPr^{26}/\ovLambda_{\R}$, where $\ovLambda_{\R}:=\ovLambda\otimes_{\Z}\R$.
By functoriality with respect to monomial maps, the tropicalization of $\ovalpha$ yields a linear map
\begin{equation}\label{eq:tropovalpha}
\trop(\ovalpha)\from \TPr^{26}/\ovLambda_{\R} \to \TPr^{44}
\end{equation}
having the same associated $45\times 27$-matrix $A$ as~\eqref{eq:alphaMap} did. The map $\trop(\ovalpha)$ is well defined since the preimage of $\zspan{\mathbf{1}}$ under $A$ is the lattice $\Lambda$.

Our next result  ensures that the combinatorics of $\Trop X\subset \TPr^{26}/\ovLambda_{\R}$ are preserved under $\trop(\ovalpha)$:
\begin{proposition}\label{pr:injectivityNewEmbedding}
  The tropical map $\trop(\ovalpha)$ from~\eqref{eq:tropovalpha}  is injective on each  $\Trop X \subset \TPr^{26}/\ovLambda_{\R}$.
\end{proposition}
\begin{proof}  The definition of the map $\ovalpha$ from~\eqref{eq:alphaBar}  is compatible with  the boundary structure on the source and target spaces by~\autoref{cor:BoundaryPoints} and~\autoref{lm:boundaryTX}. 
  This compatibility is preserved under tropicalization. In particular, the preimages of distinct  strata of $\TPr^{44}$ are disjoint strata of $\Trop X \subset  \TPr^{26}/\ovLambda_{\R}$.  Thus, it suffices to check injectivity on each strata of  $\Trop X$ arising from the stratification of $\TPr^{26}/\ovLambda_{\R}$.

  We start by discussing the big open cell $(\R^{27}/\rspanone)/\ovLambda_{\R}$. Injectivity follows from  the definition of the map $\alpha^{\#}$ between the coordinate rings in~\eqref{eq:alphaMap} because \[
  \ker(\trop(\alpha^{\#}))\cap (\R^{27}/\rspanone) =   \ker(A)\cap (\R^{27}/\rspanone) = \ovLambda_{\R}.
  \]
  The last identity is checked by a simple matrix multiplication in the Supplementary material.

  In the stable case, ~\autoref{lm:IntersectionBoundaryLines} implies that the recession fan of $\Trop X$ in both $\TPr^{26}/\ovLambda_{\R}$ and $\TPr^{44}$ is the cone over the Schl\"afli graph.  Hence, points in the boundary of $\Trop X$  in $\TPr^{26}/\ovLambda_{\R}$ lie in at most two hyperplanes at infinity.  There are two types of boundary strata up to $\Wgp$-symmetry, each determined by the number of $\infty$-coordinates. We choose our two representatives as those associated to the curve $\E{1}$ and the pair $(\E{1}, \F{12})$, respectively. Keeping the notation from the proof of~\autoref{thm:anticanL}, they are  defined inside $\Trop X \subset \TPr^{26}/\ovLambda_{\R}$ by the conditions $e_{1}=\infty$ and $e_{1}=f_{12}=\infty$, respectively.   The latter consists of a single point (the tropicalization of the node $\E{1}\cap \F{12}$), so injectivity follows automatically. 

  On the relative interior of the strata of $\Trop X$ defined by $e_1=\infty$, a point is defined by the remaining 26 coordinates. Its image under $\trop(\ovalpha)$ will have $\infty$ coordinates precisely at the five anticanonical triangles $x_{1i}$ for $i=2,\ldots, 6$. Thus, the image on the strata will be completely determined by the  $35\times 26$ submatrix $A'$ of $A$ associated to the remaining coordinates. A simple calculation available in the Supplementary material shows that the projection of the lattice $\Lambda$ to $\Z^{26}$ is a saturated rank-7 lattice $\Lambda'$ containing the all-ones vector.  Furthermore,  $A'$ has rank 20 and it contains the $\R$-span of $\Lambda'$. We conclude that $A'$ is injective on the quotient space $\R^{26}/\Lambda'_{\R}$. The same holds for  $\trop(\ovalpha)$ and the boundary strata of $\Trop X$ induced by $e_1$, as we wanted to show.
  \end{proof}

\appendix
\section{Computations}
\label{sec:appendix1}

\subsection{Coble covariants}\label{sec:coble-covariants}
\textcolor{blue}{Tables}~\ref{tab:roots},~\ref{tab:yoshida}, and~\ref{tab:crosses} list our choice of positive roots, Yoshida functions, and Cross functions, respectively.

\begin{table}[htb]
  \centering
    \begin{tabular}{| c| c| c | c | c | c | }
    \hline
    Roots $r_0$ to $r_5$ & Roots $r_6$ to $r_{11}$ & Roots $r_{12}$ to $r_{17}$ & Roots $r_{18}$ to $r_{23}$ & Roots $r_{24}$-$r_{29}$ & Roots $r_{30}$ to $r_{35}$\\
    \hline
    $    - d_{1} + d_{2} $&$ - d_{2} + d_{4} $&$ - d_{1} + d_{4} $&$ - d_{2} + d_{6} $&$ - d_{1} + d_{6} $&$ d_{2} + d_{4} + d_{6} $\\
    $    d_{1} + d_{2} + d_{3} $&$ d_{1} + d_{2} + d_{4} $&$ d_{1} + d_{2} + d_{5} $&$ - d_{1} + d_{5} $&$ d_{2} + d_{3} + d_{6} $&$ d_{2} + d_{5} + d_{6} $\\
    $    - d_{2} + d_{3} $&$ - d_{4} + d_{6} $&$ - d_{3} + d_{6} $&$ d_{1} + d_{3} + d_{5} $&$ d_{2} + d_{4} + d_{5} $&$ d_{3} + d_{4} + d_{6} $\\
    $    - d_{3} + d_{4} $&$ - d_{1} + d_{3} $&$ d_{1} + d_{3} + d_{4} $&$ d_{1} + d_{4} + d_{5} $&$ d_{1} + d_{4} + d_{6} $&$ d_{3} + d_{5} + d_{6} $\\
    $    - d_{4} + d_{5} $&$ - d_{3} + d_{5} $&$ d_{2} + d_{3} + d_{4} $&$ d_{2} + d_{3} + d_{5} $&$ d_{1} + d_{5} + d_{6} $&$ d_{4} + d_{5} + d_{6} $\\
    $    - d_{5} + d_{6} $&$ - d_{2} + d_{5} $&$ d_{1} + d_{2} + d_{6} $&$ d_{1} + d_{3} + d_{6} $&$ d_{3} + d_{4} + d_{5} $&$ d_{1} + \dots + d_{6}$\\
    \hline
  \end{tabular}
   \caption{A choice of 36 positive roots of $\E6$.}
  \label{tab:roots}
\end{table}

\begin{table}[htb]
  \centering
  \begin{tabular}{|c| c| c| c|}
    \hline
  $k\!={0}$ to ${9}$ &  ${10}$ to ${19}$ &  ${20}$ to ${29}$ &   ${30}$ to ${39}$ \\
    \hline
    $r_{10} r_{11} r_{12} r_{2} r_{22} r_{24} r_{27} r_{35} r_{8} $&$ r_{13} r_{16} r_{19} r_{2} r_{20} r_{25} r_{27} r_{34} r_{8} $&$ r_{0} r_{10} r_{12} r_{14} r_{33} r_{35} r_{5} r_{6} r_{7} $&$ r_{0} r_{13} r_{14} r_{23} r_{25} r_{29} r_{34} r_{4} r_{7} $\\
    $r_{1} r_{18} r_{21} r_{23} r_{29} r_{30} r_{31} r_{4} r_{9} $&$ r_{1} r_{12} r_{16} r_{2} r_{21} r_{27} r_{31} r_{33} r_{5} $&$ r_{1} r_{15} r_{26} r_{28} r_{30} r_{33} r_{5} r_{6} r_{9} $&$ r_{11} r_{14} r_{23} r_{24} r_{26} r_{35} r_{4} r_{6} r_{9} $\\
    $r_{15} r_{17} r_{18} r_{20} r_{25} r_{26} r_{34} r_{4} r_{9} $&$ r_{11} r_{13} r_{15} r_{22} r_{23} r_{30} r_{34} r_{8} r_{9} $&$ r_{0} r_{1} r_{2} r_{34} r_{35} r_{4} r_{5} r_{8} r_{9} $&$ r_{10} r_{12} r_{15} r_{17} r_{18} r_{21} r_{22} r_{30} r_{33} $\\
    $r_{16} r_{19} r_{2} r_{24} r_{28} r_{3} r_{35} r_{5} r_{6} $&$ r_{14} r_{19} r_{20} r_{25} r_{26} r_{28} r_{32} r_{6} r_{7} $&$ r_{0} r_{10} r_{17} r_{20} r_{22} r_{32} r_{34} r_{7} r_{8} $&$ r_{11} r_{15} r_{17} r_{22} r_{24} r_{26} r_{28} r_{3} r_{32} $\\
    $r_{1} r_{19} r_{2} r_{21} r_{22} r_{28} r_{30} r_{32} r_{8} $&$ r_{11} r_{12} r_{13} r_{14} r_{15} r_{25} r_{26} r_{27} r_{33} $&$ r_{11} r_{12} r_{14} r_{21} r_{22} r_{23} r_{31} r_{32} r_{7} $&$ r_{10} r_{16} r_{17} r_{20} r_{24} r_{26} r_{27} r_{33} r_{6} $\\
    $r_{12} r_{15} r_{2} r_{22} r_{25} r_{28} r_{34} r_{5} r_{7} $&$ r_{13} r_{14} r_{16} r_{19} r_{21} r_{23} r_{30} r_{33} r_{6} $&$ r_{12} r_{14} r_{18} r_{19} r_{2} r_{21} r_{25} r_{35} r_{4} $&$ r_{0} r_{10} r_{17} r_{18} r_{24} r_{29} r_{3} r_{35} r_{4} $\\
    $r_{0} r_{13} r_{15} r_{16} r_{17} r_{3} r_{33} r_{34} r_{5} $&$ r_{16} r_{17} r_{2} r_{21} r_{22} r_{23} r_{24} r_{34} r_{4} $&$ r_{0} r_{1} r_{10} r_{13} r_{27} r_{29} r_{30} r_{33} r_{8} $&$ r_{0} r_{11} r_{13} r_{14} r_{19} r_{3} r_{32} r_{35} r_{8} $\\
    $r_{10} r_{18} r_{19} r_{20} r_{30} r_{35} r_{6} r_{8} r_{9} $&$ r_{11} r_{12} r_{15} r_{18} r_{3} r_{31} r_{35} r_{5} r_{9} $&$ r_{10} r_{22} r_{23} r_{24} r_{28} r_{29} r_{30} r_{6} r_{7} $&$ r_{0} r_{1} r_{28} r_{29} r_{3} r_{31} r_{32} r_{5} r_{7} $\\
    $r_{16} r_{20} r_{23} r_{31} r_{34} r_{5} r_{6} r_{7} r_{9} $&$ r_{1} r_{2} r_{24} r_{25} r_{26} r_{27} r_{28} r_{29} r_{4} $&$ r_{13} r_{15} r_{18} r_{19} r_{25} r_{28} r_{29} r_{3} r_{30} $&$ r_{0} r_{1} r_{14} r_{17} r_{21} r_{26} r_{32} r_{33} r_{4}$\\
    \hline
\end{tabular}
   \caption[Yoshida functions]{Yoshida functions $\Yos{k}$ for $0 \leq k \leq 39$ expressed as products of roots. The order of the nine factors on each expression matches the ordering given by \sage.}
  \label{tab:yoshida}
\end{table}

\begin{table}[htb]
  \centering
  {\scriptsize
    \begin{tabular}{|c|c|c|c|c|c|c|c|c|}
\hline
  $k\!={0}$ to ${14}$ &  ${15}$ to ${29}$&  ${30}$ to ${44}$ &  ${45}$ to ${59}$ &  ${60}$ to ${74}$ & 75 to 89 & 90 to 104 & 105 to 119 & 120 to 134\\
  \hline
$ -\Yos{12} + \Yos{36} $&$ \Yos{17} - \Yos{32} $&$ -\Yos{10} + \Yos{4} $&$ -\Yos{1} + \Yos{38} $&$ -\Yos{0} - \Yos{35} $&$ \Yos{28} - \Yos{32} $&$ \Yos{1} + \Yos{9} $&$ \Yos{18} - \Yos{35} $&$ \Yos{31} - \Yos{36} $\\
$ -\Yos{2} + \Yos{35} $&$ -\Yos{27} - \Yos{31} $&$ -\Yos{34} - \Yos{38} $&$ -\Yos{34} + \Yos{7} $&$ -\Yos{10} - \Yos{30} $&$ -\Yos{23} + \Yos{37} $&$ -\Yos{31} + \Yos{34} $&$ -\Yos{36} + \Yos{4} $&$ \Yos{1} - \Yos{26} $\\
$ \Yos{20} - \Yos{8} $&$ -\Yos{39} + \Yos{4} $&$ -\Yos{19} + \Yos{37} $&$ \Yos{17} - \Yos{33} $&$ \Yos{37} - \Yos{39} $&$ -\Yos{10} + \Yos{11} $&$ -\Yos{22} + \Yos{23} $&$ -\Yos{17} + \Yos{5} $&$ \Yos{34} - \Yos{9} $\\
$ \Yos{30} - \Yos{36} $&$ -\Yos{12} - \Yos{29} $&$ \Yos{2} - \Yos{30} $&$ \Yos{18} - \Yos{19} $&$ \Yos{12} - \Yos{8} $&$ \Yos{2} - \Yos{25} $&$ -\Yos{3} + \Yos{33} $&$ -\Yos{2} + \Yos{31} $&$ -\Yos{31} - \Yos{39} $\\
$ \Yos{1} - \Yos{17} $&$ -\Yos{21} + \Yos{37} $&$ \Yos{15} - \Yos{21} $&$ -\Yos{24} + \Yos{4} $&$ \Yos{28} + \Yos{3} $&$ -\Yos{0} + \Yos{12} $&$ \Yos{14} + \Yos{20} $&$ -\Yos{23} - \Yos{27} $&$ -\Yos{25} - \Yos{4} $\\
$ \Yos{27} - \Yos{33} $&$ -\Yos{19} - \Yos{7} $&$ \Yos{24} - \Yos{39} $&$ \Yos{21} - \Yos{8} $&$ -\Yos{11} + \Yos{14} $&$ -\Yos{3} + \Yos{8} $&$ \Yos{39} + \Yos{7} $&$ \Yos{10} - \Yos{15} $&$ \Yos{20} - \Yos{23} $\\
$ -\Yos{21} - \Yos{32} $&$ \Yos{27} + \Yos{35} $&$ -\Yos{26} + \Yos{37} $&$ \Yos{7} - \Yos{8} $&$ \Yos{23} - \Yos{38} $&$ -\Yos{28} - \Yos{5} $&$ \Yos{33} - \Yos{5} $&$ -\Yos{1} + \Yos{19} $&$ -\Yos{0} - \Yos{27} $\\
$ -\Yos{39} - \Yos{9} $&$ \Yos{4} + \Yos{9} $&$ -\Yos{34} - \Yos{8} $&$ -\Yos{26} + \Yos{30} $&$ \Yos{31} + \Yos{8} $&$ \Yos{1} - \Yos{24} $&$ -\Yos{22} + \Yos{35} $&$ \Yos{15} - \Yos{20} $&$ -\Yos{15} + \Yos{36} $\\
$ -\Yos{16} + \Yos{18} $&$ \Yos{5} + \Yos{6} $&$ -\Yos{12} - \Yos{33} $&$ -\Yos{11} + \Yos{25} $&$ \Yos{14} - \Yos{19} $&$ -\Yos{23} + \Yos{26} $&$ -\Yos{8} + \Yos{9} $&$ \Yos{16} + \Yos{5} $&$ \Yos{1} - \Yos{39} $\\
$ -\Yos{29} + \Yos{9} $&$ -\Yos{24} + \Yos{30} $&$ -\Yos{18} - \Yos{26} $&$ \Yos{18} - \Yos{39} $&$ -\Yos{17} + \Yos{36} $&$ -\Yos{16} - \Yos{35} $&$ \Yos{29} + \Yos{39} $&$ \Yos{15} - \Yos{24} $&$ \Yos{12} - \Yos{15} $\\
$ \Yos{12} - \Yos{4} $&$ \Yos{26} - \Yos{39} $&$ -\Yos{16} - \Yos{23} $&$ -\Yos{3} - \Yos{36} $&$ -\Yos{11} + \Yos{38} $&$ -\Yos{4} - \Yos{7} $&$ \Yos{23} - \Yos{6} $&$ \Yos{0} - \Yos{26} $&$ \Yos{22} - \Yos{3} $\\
$ \Yos{25} + \Yos{30} $&$ -\Yos{17} - \Yos{28} $&$ \Yos{13} - \Yos{15} $&$ -\Yos{12} + \Yos{16} $&$ -\Yos{13} + \Yos{19} $&$ \Yos{19} - \Yos{24} $&$ \Yos{13} - \Yos{3} $&$ \Yos{3} - \Yos{37} $&$ -\Yos{13} + \Yos{39} $\\
$ \Yos{35} - \Yos{7} $&$ \Yos{12} - \Yos{24} $&$ \Yos{14} - \Yos{31} $&$ -\Yos{12} - \Yos{27} $&$ \Yos{13} - \Yos{28} $&$ -\Yos{13} - \Yos{30} $&$ \Yos{14} - \Yos{39} $&$ -\Yos{15} + \Yos{6} $&$ \Yos{22} - \Yos{7} $\\
$ -\Yos{37} - \Yos{4} $&$ -\Yos{19} + \Yos{2} $&$ \Yos{10} - \Yos{16} $&$ -\Yos{14} - \Yos{18} $&$ \Yos{1} + \Yos{8} $&$ \Yos{22} - \Yos{39} $&$ \Yos{22} + \Yos{6} $&$ -\Yos{13} + \Yos{37} $&$ -\Yos{34} - \Yos{6} $\\
$ -\Yos{10} - \Yos{6} $&$ \Yos{11} - \Yos{21} $&$ \Yos{1} + \Yos{28} $&$ -\Yos{17} + \Yos{3} $&$ \Yos{38} + \Yos{6} $&$ -\Yos{29} - \Yos{3} $&$ \Yos{15} - \Yos{4} $&$ -\Yos{25} - \Yos{35} $&$ -\Yos{13} + \Yos{5} $\\
\hline
\end{tabular}
   }
  \caption{Cross functions $\Cross{k}$ for $0 \leq k \leq 134$ in terms of Yoshida functions. We list one out of four possible linear binomial expressions giving each Cross function. The remaining ones are available in the Supplementary material. Each consecutive triple corresponds to one anticanonical triangle, ordered by $x_{31}$, $x_{45}$, $y_{162345}$, $x_{12}$, $y_{142635}$, $x_{53}$, $x_{41}$, $x_{24}$, $x_{51}$, $x_{43}$, $y_{152346}$, $x_{16}$, $x_{63}$, $y_{152436}$, $x_{14}$, $x_{61}$, $y_{132456}$, $x_{65}$, $y_{152634}$, $y_{123456}$, $y_{162435}$, $x_{26}$, $x_{56}$, $y_{162534}$, $x_{46}$, $x_{54}$, $x_{13}$, $y_{123546}$, $x_{32}$, $x_{52}$, $x_{35}$, $x_{25}$, $y_{123645}$, $x_{36}$, $x_{34}$, $x_{23}$, $x_{62}$, $x_{15}$, $x_{21}$, $y_{132645}$, $y_{142536}$, $x_{64}$, $x_{42}$, $y_{142356}$ and $y_{132546}$.
}
  \label{tab:crosses}
\end{table}
\normalsize
\subsection{$\mathbf{\Wgp}$ action}
They Weyl group $\Wgp$ acts on the roots $\pm r_0, \dots,\pm r_{35}$ in the standard way.
The action on the roots induces an action an action on the Yoshida functions $\Yos{0}, \dots, \Yos{39}$, and on the Cross functions $\Cross{0}, \dots, \Cross{134}$.

\begin{table}[H]  \centering
  \begin{tabular}{| c | c | c | c | c | c |}
  \hline
 \!{gen.}\normalsize\!&  \Large{$\sigma_1$}\normalsize & \Large{$\sigma_2$}\normalsize & \Large{$\sigma_3$}\normalsize & \Large{$\sigma_4$}\normalsize & \Large{$\sigma_5$}\normalsize \\\hline
 $\Sn{6}$ & $(1\, 2)$ & $(2\, 3)$ & $(3\,4)$ &\!$(4\,5)$\!& $(5\,6)$\\\hline
  \!{signs}\normalsize\!&  $Y_{142536} \leftrightarrow -Y_{152436}$ &                                 $Y_{162534} \leftrightarrow -Y_{162435}$ & $Y_{142356} \leftrightarrow -Y_{142536}$ & \!None\! & $Y_{142536} \leftrightarrow -Y_{142635}$
  \\
& &     $Y_{132645} \leftrightarrow -Y_{123645}$ &
$Y_{132456} \leftrightarrow -Y_{142356}$ & &            $Y_{162534} \leftrightarrow -Y_{152634}$\\   & & & & & $Y_{123546} \leftrightarrow -Y_{123546}$\\\hline   
\end{tabular}
\smallskip

\begin{tabular}{| c | c | c |}
  \hline
  \!gen.\!\!& action on $\fh_6^* = \langle  d_1, \dots, d_6 \rangle$ & action on the scaled anticanonical coordinates from $T$ \\\hline
  \Large{$\sigma_6$}\normalsize &  \!\!\!\!  $\small{\begin{pmatrix}
    1 & 0 & 0 & 1/3 & 1/3 & 1/3\\
    0 & 1 & 0 & 1/3 & 1/3 & 1/3 \\
    0 & 0 & 1 & 1/3 & 1/3 & 1/3 \\
    0 & 0 & 0 & 1/3 & -2/3 & -2/3 \\
    0 & 0 & 0 & -2/3 & 1/3 & -2/3\\
    0 & 0 & 0 & -2/3 & -2/3 & 1/3
  \end{pmatrix}}$\normalsize \!\!&
\!\!\!\!\!\!    $\begin{array}{lll} X_{12} \leftrightarrow -X_{13} ;&
                X_{14} \leftrightarrow X_{14} ;&
               X_{15} \leftrightarrow X_{15}  ; \\
               X_{16} \leftrightarrow X_{16}  ; &
               X_{21} \leftrightarrow -X_{23}  ; &
               X_{24} \leftrightarrow X_{24}  ; \\
               X_{25} \leftrightarrow X_{25}  ; &
               X_{26} \leftrightarrow X_{26}  ; &
               X_{31} \leftrightarrow -X_{32}  ; \\
               X_{34} \leftrightarrow X_{34}  ; &
               X_{35} \leftrightarrow X_{35}  ; &
               X_{36} \leftrightarrow X_{36}  ; \\
               X_{41} \leftrightarrow Y_{142356}  ; &
               X_{42} \leftrightarrow Y_{132456}  ; &
               X_{43} \leftrightarrow Y_{123456}  ;\\
               X_{45} \leftrightarrow -X_{65}  ; &
               X_{46} \leftrightarrow -X_{56}  ; &
               X_{51} \leftrightarrow Y_{152346} ;\\
               X_{52} \leftrightarrow Y_{132546}  ; &
               X_{53} \leftrightarrow Y_{123546}  ; &
               X_{54} \leftrightarrow -X_{64}   ; \\
               X_{61} \leftrightarrow Y_{162345}  ; &
               X_{62} \leftrightarrow Y_{132645}  ;& 
               X_{63} \leftrightarrow -Y_{123645}  ; \\
                Y_{142536} \leftrightarrow Y_{142536} ; &
                Y_{142635} \leftrightarrow Y_{142635} ; &
                Y_{152436} \leftrightarrow Y_{152436} ;\\ 
                Y_{152634} \leftrightarrow Y_{152634}  ; &
                Y_{162435} \leftrightarrow Y_{162435}; &
                Y_{162534} \leftrightarrow Y_{162534} . 
  \end{array}$\!\!\!\!
                                                                                              \\\hline
            \end{tabular}

   \caption{Action of $\Wgp$ via involutions. Each $\sigma_i$ with $i=1,\ldots, 5$ acts on the $d_i$'s by permuting subscripts and on the scaled anticanonical coordinates from the set $T$ by signed subscript permutations. The negative signs are listed. The action of $\sigma_6$ on the roots is obtained from the action on $\langle  d_1, \dots, d_6 \rangle$ by left multiplication by the matrix on the table. Its action on  the variables in $T$ is linear and determined by signed correspondences in the table.
  \label{tab:action}}
\end{table}

The identification of $\EGp{6}$ with the orthogonal complement of the canonical divisor in the Picard lattice of the cubic surface determines the action of $\Wgp$ on its markings.
We fix the correspondence that takes $E_i - E_j$ to $d_i - d_j$.
This yields an action of $\Wgp$ on the set $\{E_i, F_{ij}, G_j \}_{ij}$ of exceptional curves, and likewise, on the set $\{E_i F_{ij} G_j \}_{ij} \cup \{F_{ij}F_{kl}F_{mn}\}_{i,j,k,l,m,n}$ of anticanonical triangles.
The later induces a \emph{signed} action  on the variables $\{X_{ij}, Y_{ijklmn}\}$ of the universal anticanonical ring of~\autoref{thm:anticanK}.
The sign choice is not canonical (see the discussion preceeding~\eqref{eq:YoshidaAdaptedCoords}.)

\autoref{tab:action} describes the action of a set of generators of $\Wgp$ on the roots, the set of exceptional curves, and the set of variables of the universal anticanonical ring, which fixes our sign convention.
Our  6 generators $\{\sigma_0,\ldots, \sigma_5\}$ depend on a choice of labeling of the $\E6$ Dynkin diagram as in~\autoref{fig:E6label}. The  vertex $i$ induces a simple root with label $r_i$, and the generator $\sigma_i$ corresponds to  the reflection about the hyperplane perpendicular to the root $r_i$.

  \bibliographystyle{abbrv}
 \def\cprime{$'$}

 \label{sec:biblio}

  \bigskip
\noindent
\textbf{\small{Authors' addresses:}}
\smallskip
\

\noindent
\small{M.A.\ Cueto,  Mathematics Department, The Ohio State University, 231 W 18th Ave, Columbus, OH 43210, USA.
\\
\noindent \emph{Email address:} \url{cueto.5@osu.edu}}
\vspace{2ex}

\noindent
\small{A. Deopurkar, Mathematical Sciences Institute, Australian National University, John Dedman Building 27, Union Lane, Canberra  ACT  2601, Australia. 
\\
\noindent \emph{Email address:} \url{anand.deopurkar@anu.edu.au}

\end{document}